\documentclass[12pt,a4paper]{amsart}
\usepackage{amssymb}

\calclayout

\newcommand{\+}{\protect\nobreakdash-}
\renewcommand{\.}{\mskip.5\thinmuskip}
\renewcommand{\:}{\colon}

\DeclareMathOperator{\Spec}{Spec}
\DeclareMathOperator{\PSupp}{PSupp}
\DeclareMathOperator{\Hom}{Hom}
\DeclareMathOperator{\Ext}{Ext}
\DeclareMathOperator{\Tor}{Tor}
\DeclareMathOperator{\id}{id}
\DeclareMathOperator{\im}{im}
\DeclareMathOperator{\coker}{coker}

\newcommand{\ot}{\otimes}
\newcommand{\rarrow}{\longrightarrow}
\newcommand{\larrow}{\longleftarrow}

\newcommand{\bu}{{\text{\smaller\smaller$\scriptstyle\bullet$}}}
\newcommand{\lrarrow}{\.\relbar\joinrel\relbar\joinrel\rightarrow\.}
\newcommand{\llarrow}{\.\leftarrow\joinrel\relbar\joinrel\relbar\.}

\newcommand{\qcoh}{{\operatorname{\mathsf{--qcoh}}}}
\newcommand{\modl}{{\operatorname{\mathsf{--mod}}}}
\newcommand{\modr}{{\operatorname{\mathsf{mod--}}}}
\newcommand{\contra}{{\operatorname{\mathsf{--contra}}}}

\newcommand{\tors}{{\operatorname{\mathsf{-tors}}}}
\newcommand{\ctra}{{\operatorname{\mathsf{-ctra}}}}

\newcommand{\vfl}{{\mathsf{vfl}}}
\newcommand{\fl}{{\mathsf{fl}}}
\newcommand{\lp}{{\mathsf{lp}}}
\newcommand{\qs}{{\mathsf{qs}}}
\renewcommand{\b}{{\mathsf{b}}}

\newcommand{\F}{\mathcal F}
\renewcommand{\O}{\mathcal O}
\newcommand{\J}{\mathfrak J}
\renewcommand{\P}{\mathfrak P}

\newcommand{\sA}{\mathsf A}
\newcommand{\sC}{\mathsf C}
\newcommand{\sD}{\mathsf D}
\newcommand{\sE}{\mathsf E}
\newcommand{\sF}{\mathsf F}
\newcommand{\sG}{\mathsf G}
\newcommand{\sM}{\mathsf M}
\newcommand{\sN}{\mathsf N}
\newcommand{\sS}{\mathsf S}
\newcommand{\sT}{\mathsf T}

\newcommand{\boZ}{\mathbb Z}
\newcommand{\boQ}{\mathbb Q}
\newcommand{\boR}{\mathbb R}
\newcommand{\boL}{\mathbb L}

\renewcommand{\r}{\mathbf r}
\newcommand{\s}{\mathbf s}
\renewcommand{\t}{\mathbf t}

\newcommand{\m}{\mathfrak m}
\newcommand{\p}{\mathfrak p}
\newcommand{\q}{\mathfrak q}

\newcommand{\oR}{{\overline R}}
\newcommand{\oS}{{\overline S}}
\newcommand{\oF}{{\overline F}}
\newcommand{\oI}{{\overline I}}
\newcommand{\oN}{{\overline N}}
\newcommand{\oP}{{\overline P}}

\newcommand{\fR}{{\mathfrak R}}

\newcommand{\Section}[1]{\bigskip\section{#1}\medskip}
\theoremstyle{plain}
\newtheorem{thm}{Theorem}[section]

\newtheorem{lem}[thm]{Lemma}
\newtheorem{prop}[thm]{Proposition}
\newtheorem{cor}[thm]{Corollary}
\newtheorem*{nip}{Noetherian Induction Principle}
\newtheorem{mt}[thm]{Main Theorem}
\newtheorem{ml}[thm]{Main Lemma}
\newtheorem{mpr}[thm]{Main Proposition}
\newtheorem{tml}[thm]{Toy Main Lemma}
\newtheorem{tmp}[thm]{Toy Main Proposition}
\newtheorem{subl}[thm]{Sublemma}

\theoremstyle{definition}
\newtheorem{ex}[thm]{Example}

\newtheorem{rem}[thm]{Remark}
\newtheorem{defn}[thm]{Definition}

\begin{document}

\title{Flat morphisms of finite presentation \\ are very flat}

\author{Leonid Positselski and Alexander Sl\'avik}

\address{L.P.: Institute of Mathematics of the Czech Academy of
Sciences, \v Zitn\'a~25, 115~67 Prague~1, Czech Republic; and
\newline\indent Laboratory of Algebraic Geometry, National Research
University Higher School of Economics, Moscow 119048; and
\newline\indent Laboratory of Algebra and Number Theory, Institute
for Information Transmission Problems, Moscow 127051, Russia; and
\newline\indent Department of Mathematics, Faculty of Natural Sciences,
University of Haifa, Mount Carmel, Haifa 31905, Israel}

\email{positselski@yandex.ru}

\address{A.S.: Charles University, Faculty of Mathematics and Physics,
Department of Algebra, Sokolovsk\'a~83, 186~75 Prague~8,
Czech Republic}

\email{Slavik.Alexander@seznam.cz}

\begin{abstract}
 Principal affine open subsets in affine schemes are an important tool
in the foundations of algebraic geometry.
 Given a commutative ring $R$, \,$R$\+modules built from the rings of
functions on principal affine open subschemes in $\Spec R$ using
ordinal-indexed filtrations and direct summands are called
\emph{very flat}.
 The related class of \emph{very flat quasi-coherent sheaves} over
a scheme is intermediate between the classes of locally free and
flat sheaves, and has serious technical advantages over both.
 In this paper we show that very flat modules and sheaves are ubiquitous
in algebraic geometry: if $S$ is a finitely presented commutative
$R$\+algebra which is flat as an $R$\+module, then $S$ is
a very flat $R$\+module.
 This proves a conjecture formulated in the February~2014 version of
the first author's long preprint on contraherent cosheaves~\cite{Pcosh}.
 We also show that the (finite) very flatness property of a flat module
satisfies descent with respect to commutative ring homomorphisms of
finite presentation inducing surjective maps of the spectra.
\end{abstract}

\maketitle

\tableofcontents

\setcounter{section}{-1}
\section{Introduction}
\medskip

\subsection{{}}
 The abelian category of modules over a ring has enough projective
objects, which can be used in order to construct left derived functors.
 Right derived functors can be constructed using injective modules.

 The abelian category of quasi-coherent sheaves on a scheme also has
enough injective objects but, generally speaking, no projectives.
 On a quasi-projective scheme, there are enough invertible sheaves;
and some other schemes may have enough locally free sheaves, of
finite or infinite rank.

 Locally free sheaves of possibly infinite rank, or what is
the same over a locally Noetherian scheme~\cite[Corollary~4.5]{Ba},
locally projective sheaves, actually form a quite remarkable class.
 In particular, the projectivity of a module over a commutative ring
is a local property, satisfying descent with respect to Zariski
coverings, and in fact, even with respect to faithfully flat
ring morphisms~\cite{RG,Pe}.
 But what if there are not enough locally projective sheaves?

 A Noetherian scheme is said to have the \emph{resolution property} if
it has enough locally free sheaves of finite rank~\cite{To}.
 It seems to be an open question whether there exist algebraic
varieties that do not have the resolution property.

 On the other hand, on any quasi-compact semi-separated scheme there
are enough flat quasi-coherent sheaves~\cite[Section~2.4]{M-n},
\cite[Lemma~A.1]{EP}.
 But arbitrary flat modules over a commutative ring are much more
complicated, from the homological point of view, than
the projective modules.
 In particular, flat modules over a ring that is not necessarily
Noetherian of finite Krull dimension form an exact category of
infinite homological dimension.

\subsection{{}}
 In fact, the derived category of the exact category of flat modules
over an associative ring was studied and shown to be equivalent to
the homotopy category of projective modules~\cite{N-f}.
 Extending this approach to schemes, one can define the \emph{mock
homotopy category of} (nonexistent) \emph{projective quasi-coherent
sheaves} on a quasi-compact semi-separated scheme as the derived
category of the exact category of flat quasi-coherent
sheaves~\cite{M-th}.

 One of the potential shortcomings of this approach is that it is not
directly applicable to matrix factorizations.
 Matrix factorizations are a kind of curved differential modules, so
the conventional construction of the derived category does not make
sense for them.
 Instead, one has to use \emph{derived categories of the second kind},
such as the \emph{coderived} and the \emph{absolute derived}
categories~\cite{Or,EP}.
 However, for an exact category of infinite homological dimension,
the coderived category may differ from the derived
category~\cite[Section~2.1]{Psemi}, \cite[Examples~3.3]{Pkoszul}.

 One of the approaches suggested in~\cite[Remark~2.6]{EP} is to use
the \emph{very flat} quasi-coherent sheaves (which were previously
defined in~\cite[Section~4.1]{Pcosh}) in lieu of the locally free
or flat ones.
 The aim of this paper is to show just how well this approach turns
out to work.

\subsection{{}} \label{very-flat-qcoh-introd}
 For every quasi-compact semi-separated scheme $X$, there is a full
subcategory of \emph{very flat quasi-coherent sheaves}
$X\qcoh_\vfl\subset X\qcoh$ in the abelian category $X\qcoh$ of
quasi-coherent sheaves on $X$ having the following properties:

\begin{enumerate}
\renewcommand{\theenumi}{VF\arabic{enumi}}
\item All very flat quasi-coherent sheaves are flat, that is
$X\qcoh_\vfl\subset X\qcoh_\fl$, where $X\qcoh_\fl\subset X\qcoh$
denotes the full subcategory of flat quasi-coherent sheaves on~$X$.
\item All locally projective quasi-coherent sheaves are very flat,
that is $X\qcoh_\lp\subset X\qcoh_\vfl$, where $X\qcoh_\lp\subset
X\qcoh$ denotes the full subcategory of locally projective
quasi-coherent sheaves on~$X$.
\item The tensor product of any two very flat quasi-coherent sheaves
is very flat, that is $X\qcoh_\vfl\.\ot_{\O_X}X\qcoh_\vfl\subset
X\qcoh_\vfl$.
\item The full subcategory $X\qcoh_\vfl\subset X\qcoh$ is closed under
extensions, infinite direct sums, direct summands, and the passages
to the kernels of surjective morphisms of quasi-coherent sheaves.
\item Every quasi-coherent sheaf on $X$ is a quotient sheaf of
a very flat quasi-coherent sheaf.
\end{enumerate}
 Here we say that a quasi-coherent sheaf $\F$ on $X$ is locally
projective if, for every affine open subscheme $U\subset X$,
the $\O_X(U)$\+module $\F(U)$ is projective.

 The properties (VF4\+-VF5) can be summarized by saying that
$X\qcoh_\vfl$ is a \emph{resolving subcategory} in $X\qcoh$ (closedness
under direct sums is an extra property not covered by the definition
of a resolving subcategory).
 In particular, it follows that the full subcategory $X\qcoh_\vfl$
inherits an exact category structure from the abelian category $X\qcoh$.
\begin{enumerate}
\renewcommand{\theenumi}{VF\arabic{enumi}}
\setcounter{enumi}{5}
\item The exact category $X\qcoh_\vfl$ has finite homological dimension,
not exceeding the number of affine open subschemes in an affine open
covering of a quasi-compact semi-separated scheme~$X$.
\item The embedding of exact categories $X\qcoh_\vfl\rarrow X\qcoh_\fl$
induces an equivalence of their derived categories,
$\sD(X\qcoh_\vfl)\simeq\sD(X\qcoh_\fl)$.
\end{enumerate}
 The properties (VF1\+-VF5) hold for the class of all flat
quasi-coherent sheaves in place of the class of very flat quasi-coherent
sheaves just as well, but the property~(VF6) does not.
 In particular, it means that any very flat module (\,$=$~very flat
quasi-coherent sheaf over an affine scheme) has projective dimension
not exceeding~$1$.
 
 But even flat modules of projective dimension~$1$ can be complicated.
 The following property shows the difference between a very flat module
and an arbitrary flat module of projective dimension~$1$.
 Given a scheme $X$ and a scheme point $x\in X$, let us view~$x$ as
a scheme (namely, the spectrum of the residue field $k_x(X)$ of
the point $x\in X$) and denote by $\iota_x\:x\rarrow X$ the natural
morphism.
 For any quasi-coherent sheaf $\F$ on $X$, the \emph{P\+support}
$\PSupp_X\F\subset X$ of $\F$ is the set of all points $x\in X$ such
that the quasi-coherent sheaf (or, which is the same in this case,
the $k_x(X)$\+vector space) $\iota_x^*\F$ does not vanish.
\begin{enumerate}
\renewcommand{\theenumi}{VF\arabic{enumi}}
\setcounter{enumi}{7}
\item For any very flat quasi-coherent sheaf $\F$ on $X$, the subset
$\PSupp_X\F\subset X$ is open in the Zariski topology of~$X$.
\end{enumerate}
 Notice that the property~(VF8) certainly does not hold for arbitrary
flat modules or quasi-coherent sheaves, as the example of the flat
$\boZ$\+module $\boQ$ (whose P\+support consists of the single
generic point of $\Spec\boZ$) already demonstrates.

 On the other hand, the properties (VF1\+-VF4) and~(VF8) hold
for the class of all locally projective quasi-coherent sheaves in place
of the class of all very flat ones.
 When the analogue of~(VF5) holds for locally projective quasi-coherent
sheaves, it follows that the analogues of~(VF6\+-VF7) also do.

\subsection{{}}
 Further important properties of the class of very flat quasi-coherent
sheaves are related to changing schemes.
 Let $f\:Y\rarrow X$ be a morphism of quasi-compact semi-separated
schemes, and let $X=\bigcup_\alpha U_\alpha$ be an open covering of
the scheme~$X$.
\begin{enumerate}
\renewcommand{\theenumi}{VF\arabic{enumi}}
\setcounter{enumi}{8}
\item The inverse image functor~$f^*$ takes very flat quasi-coherent
sheaves on $X$ to very flat quasi-coherent sheaves on $Y$, that is
$f^*(X\qcoh_\vfl)\subset Y\qcoh_\vfl$.
\item The very flatness property of a quasi-coherent sheaf on $X$
satisfies Zariski descent, that is a quasi-coherent sheaf $\F$ on $X$
is very flat if and only the quasi-coherent sheaf $\F|_{U_\alpha}$ is
very flat on $U_\alpha$ for every~$\alpha$.
\end{enumerate}
 Of course, the classes of flat quasi-coherent sheaves and locally
projective quasi-coherent sheaves also satisfy their similar
versions of~(VF9\+-VF10).

 All of the above having been mentioned, one remaining major property
one would like to have for the class of very flat quasi-coherent
sheaves is the preservation by \emph{direct images} with respect to
good enough affine morphisms.
 For example, the class of flat quasi-coherent sheaves is preserved
by direct images with respect to all \emph{flat} affine morphisms
$f\:Y\rarrow X$, where an affine morphism~$f$ is said to be flat if
the $\O_X(U)$\+module $\O_Y(f^{-1}(U))$ is flat for every affine open
subscheme $U\subset X$.
 Equivalently, a morphism~$f$ is flat if the local ring
$\O_y(Y)$ is a flat module over the local ring $\O_x(X)$ for every
pair of points $x\in X$ and $y\in Y$ such that $f(y)=x$.

 What can one say about the preservation of very flatness by the direct
images?
 What are the \emph{very flat} affine morphisms of schemes?
 Notice first of all that, according to~(VF10), the question reduces
to affine schemes, so it suffices to discuss the \emph{very flat
morphisms of commutative rings}.
 What are these?
 Before we begin to answer this question, let us define, at long last,
the \emph{very flat modules}.

\subsection{{}} \label{very-flat-definition-introd}
 Let $R$ be a commutative ring.
 For every element $r\in R$, we denote by $R[r^{-1}]$ the result of
inverting the element~$r$ in $R$, that is, the localization of $R$ with
respect to the multiplicative subset $\{1,r,r^2,r^3,\dots\}\subset R$.
 So $R[r^{-1}]$ is an $R$\+algebra, and in particular, an $R$\+module.
 Using the telescope construction for the inductive limit, one can
easily produce a two-term free resolution of the $R$\+module $R[r^{-1}]$,
so its projective dimension does not exceed~$1$.

 An $R$\+module $C$ is said to be \emph{$r$\+contraadjusted} if
$\Ext^1_R(R[r^{-1}],C)=0$.
 An $R$\+mod\-ule $C$ is \emph{contraadjusted} if it is
$r$\+contraadjusted for every $r\in R$.
 An $R$\+module $F$ is called \emph{very flat} if $\Ext^1_R(F,C)=0$
for every contraadjusted $R$\+module~$C$.

 A somewhat more explicit or more intuitively clear description of
the class of all very flat $R$\+modules can be given using the notion
of a smooth chain of submodules, or a \emph{transfinitely iterated
extension}.
 Let $F$ be an $R$\+module and $\gamma$~be an ordinal.
 Suppose that for every ordinal $\alpha\le\gamma$ we are given
a submodule $F_\alpha\subset F$ such that the following conditions
are satisfied:
\begin{itemize}
\item $F_0=0$ and $F_\gamma=F$; \par
\item one has $F_\alpha\subset F_\beta$ for all $\alpha\le\beta$;
\item and one has $F_\beta=\bigcup_{\alpha<\beta}F_\alpha$ for all
the limit ordinals $\beta\le\gamma$.
\end{itemize}
 Then the $R$\+module $F$ is said to be a \emph{transfinitely iterated
extension} (``in the sense of the inductive limit'') of
the $R$\+modules $F_{\alpha+1}/F_\alpha$, where $0\le\alpha<\gamma$.

 An $R$\+module $F$ is very flat if and only if it is a direct summand
of a transfinitely iterated extension of (some $R$\+modules isomorphic
to) the $R$\+modules $R[r^{-1}]$, where $r\in R$.
 This is a simple corollary of the main results of the paper~\cite{ET}
applied to the specific situation at hand (see
also~\cite[Corollary~1.1.4]{Pcosh}).

 The very flat modules over commutative rings were originally defined
in the preprint~\cite[Section~1.1]{Pcosh} and further studied in
the paper~\cite{ST}.
 Why are they relevant to algebraic geometry?
 One of the simplest answers to this question is this: let $U$ be
an affine scheme, and $V\subset U$ be an affine open subscheme.
 Then $\O(V)$ is a very flat $\O(U)$\+module~\cite[Lemma~1.2.4]{Pcosh}.

 Thus the class of very flat $R$\+modules, defined above in terms
of the rings of functions $R[r^{-1}]$ on the principal affine open
subschemes $\Spec R[r^{-1}]\subset\Spec R$, \ $r\in R$, can be
equivalently defined in terms of (the underlying $R$\+modules of)
the rings of functions on arbitrary affine open subschemes in $\Spec R$.

 A quasi-coherent sheaf $\F$ on a scheme $X$ is said to be \emph{very
flat} if, for every affine open subscheme $U\subset X$,
the $\O_X(U)$\+module $\F(U)$ is very flat.

\subsection{{}} \label{very-flat-morphisms-introd}
 Now let us discuss the direct images with respect to affine morphisms.
 Let $R\rarrow S$ be a morphism of commutative rings and $G$ be
a flat $S$\+module.
 Assume that $S$ is a flat $R$\+module.
 Then $G$ is a flat $R$\+module.

 What is the analogue of this assertion for very flat modules?
 Let $G$ be a very flat $S$\+module.
 What do we need to know about the ring homomorphism $R\rarrow S$
in order to conclude that $G$ is a very flat $R$\+module?
 The following definition provides the obvious answer: a commutative
ring homomorphism $R\rarrow S$ is called \emph{very flat} if, for
every element $s\in S$, the $R$\+module $S[s^{-1}]$ is very flat.

 The na\"\i ve condition that just the ring $S$ itself should be a very
flat $R$\+module is \emph{not} enough.
 In fact, in Example~\ref{very-flat-algebra-counterex} in this paper
we demonstrate an example of a commutative ring $S$ with an element
$s\in S$ such that $S$ is a free abelian group, but $S[s^{-1}]$ is
\emph{not} a very flat abelian group or $\boZ$\+module.

 This discussion implies that proving that a particular ring
homomorphism is very flat may be not an easy task, even in the most
elementary situations.
 Let $S=R[x]$ be the polynomial ring in one variable over~$R$.
 How does one check that the $R$\+module $S[s^{-1}]$ is very flat
for every $s\in S$?
 A geometric proof of this assertion in the case when $R$ contains
a field can be found in~\cite[Theorem~1.7.13]{Pcosh}.
 It is a complicated argument consisting of many steps.

 How does one prove that a finite flat morphism is very flat?
 Let $S=R[x]/(f)$ be the quotient ring of the polynomial ring $R[x]$
by the ideal generated by a unital polynomial $f(x)=x^n+f_{n-1}x^{n-1}
+\dotsb+f_0$, where $f_i\in R$.
 So $S$ is a finitely generated free $R$\+module of rank~$n$.
 How does one check that the $R$\+module $S[s^{-1}]$ is very flat
for every $s\in S$?
 Our attempts to look into specific simple instances of this question
led to the general theory developed in this paper.

 How does one prove that an \'etale morphism of algebraic varieties
is very flat?
 To all of these questions, we offer one and the same answer: by
applying the theorem formulated in the title of the present paper
(property~(VF15) below, or
Main Theorem~\ref{general-algebra-main-theorem} and
Corollary~\ref{very-flat-morphism-main-cor}).

\subsection{{}}  \label{VF-new-properties-introd}
 The definition of a very flat morphism of schemes in now in order.
 A morphism of schemes $f\:Y\rarrow X$ is called \emph{very flat} if,
for any affine open subschemes $U\subset X$ and $V\subset Y$ such
that $f(V)\subset U$, the $\O_X(U)$\+module $\O_Y(V)$ is very flat.
 An affine morphism~$f$ is very flat if and only if, for every
affine open subscheme $U\subset X$, the morphism of rings
$\O_X(U)\rarrow\O_Y(f^{-1}(U))$ is very flat.
\begin{enumerate}
\renewcommand{\theenumi}{VF\arabic{enumi}}
\setcounter{enumi}{10}
\item For any very flat affine morphism $f\:Y\rarrow X$, the direct
image functor~$f_*$ takes very flat quasi-coherent sheaves on $Y$
to very flat quasi-coherent sheaves on $X$, that is
$f_*(Y\qcoh_\vfl)\subset X\qcoh_\vfl$.
\item Let $f\:Y\rarrow X$ be a morphism of schemes and
$X=\bigcup_\alpha U_\alpha$ be an open covering of the scheme~$X$.
Then the morphism~$f$ is very flat if and only if the morphism
$U_\alpha\times_XY\rarrow U_\alpha$ is very flat for every~$\alpha$.
\item Let $f\:Y\rarrow X$ be a morphism of schemes and
$Y=\bigcup_\beta V_\beta$ be an open covering of the scheme~$Y$.
Then the morphism~$f$ is very flat if and only if the morphism
$V_\beta\rarrow X$ is very flat for every~$\beta$.
\item Any very flat morphism of schemes $f\:Y\rarrow X$ is an open
map between the underlying topological spaces of $Y$ and~$X$.
\end{enumerate}

 One can easily check from the definitions that the composition of
any two very flat morphisms is very flat.
 A morphism of rings $R\rarrow S$ is said to be \emph{universally
very flat} if the induced morphism of rings $T\rarrow S\ot_RT$
is very flat for every ring homomorphism $R\rarrow T$.
 Similarly, a morphism of schemes is said to be universally very flat
if it remains very flat after any base change.
\begin{enumerate}
\renewcommand{\theenumi}{VF\arabic{enumi}}
\setcounter{enumi}{14}
\item All flat morphisms of finite presentation between commutative
rings are (universally) very flat.
All flat morphisms (locally) of finite presentation between schemes
are (universally) very flat.
\item Any morphism from a field to a commutative ring is universally
very flat.
Any morphism from a scheme to the spectrum of a field is universally
very flat.
\end{enumerate}
 In particular, the property~(VF16) means that, for any commutative
algebras $S$ and $T$ over a field~$k$ and any element
$w\in S\ot_k T$, the $T$\+module $(S\ot_k T)[w^{-1}]$ is very flat.
 This is the result of our Corollary~\ref{universally-very-flat-cor}.

 To sum up much of the preceding discussion, it follows from
the properties~(VF11) and~(VF15) that direct images with respect to
flat affine morphisms of finite presentation take very flat
quasi-coherent sheaves to very flat quasi-coherent sheaves.
 The similar assertion, of course, holds for the flat quasi-coherent
sheaves; but the conditions for the direct image to preserve the class
of locally projective quasi-coherent sheaves are much more restrictive.

 In fact, local projectivity is not even preserved by direct images
with respect to affine open embeddings.
 This is the reason why the proof of property~(VF5) for flat or very
flat quasi-coherent sheaves does not apply to locally projective ones.

\subsection{{}}
 Let us now provide supporting references for
the assertions~(VF1\+-VF14).
 The properties~(VF1\+-VF2) hold by the definition, and (VF4)~is easy
to check (see the beginning of~\cite[Section~1.1]{Pcosh}).
 The property~(VF3) is~\cite[Lemma~1.2.1(a)]{Pcosh}.
 The property~(VF5) requires a proof; see~\cite[Lemma~4.1.1]{Pcosh}.

 The property~(VF6) is a reformulation of~\cite[Lemma~4.6.9(a)]{Pcosh}
(as the contraadjusted very flat quasi-coherent sheaves are
the injective objects of the exact category $X\qcoh_\vfl$, and
there are enough of them by~\cite[Corollary~4.1.4(b)]{Pcosh}).

 (VF7)~is a result of the paper~\cite{ES};
see~\cite[Corollary~6.1]{ES}.

 The property~(VF8) is~\cite[Theorem~1.7.6]{Pcosh}.
 The property~(VF9) is~\cite[Lemma~1.2.2(b)]{Pcosh}, and
(VF10)~is~\cite[Lemma~1.2.6(a)]{Pcosh}.
 The property~(VF11) is~\cite[Lemma~1.2.3(b)]{Pcosh}, (VF12)~follows
from~(VF10), and the property~(VF13) is~\cite[Lemma~1.2.7]{Pcosh}.
 The property~(VF14) follows from~(VF8)
(see~\cite[Corollary~1.7.8]{Pcosh}).

 The property~(VF15) is the main result of this paper.
 See Main Theorem~\ref{general-algebra-main-theorem} and
Corollary~\ref{very-flat-morphism-main-cor}.

\subsection{{}}
 Let us say a few more words about~(VF14) and~(VF15).
 It is a classical theorem in algebraic geometry that any flat
morphism of schemes that is locally of finite presentation is an open
map~\cite[Th\'eor\`eme~2.4.6]{Groth}.
 Our results interpret this theorem as the combination of two
implications: any flat morphism locally of finite presentation is
very flat, and any very flat morphism is open.

 Moreover, we have a module version of~(VF15)
(Main Theorem~\ref{general-module-main-theorem}),
providing a similar interpretation of the assertion in~\cite{Groth}
concerning \emph{quasi-flat} morphisms.

 Notice that an arbitrary flat morphism of schemes is, of course,
\emph{not} open: it suffices to consider the morphism
$\Spec\boQ\rarrow\Spec\boZ$.

\subsection{{}} \label{cosheaves-introd}
 Before we finish this introduction, let us explain the connection
with \emph{contraherent cosheaves}, which supplied the original
motivation for introducing the notion of a very flat module in
the preprint~\cite{Pcosh}.
 The concept of a contraherent cosheaf on a scheme is dual to that
of a quasi-coherent sheaf.

 More specifically, a \emph{contraherent cosheaf} $\P$ on a scheme $X$
is a rule assigning to every affine open subscheme $U\subset X$
an $\O_X(U)$\+module $\P[U]$ and to every pair of embedded affine open
subschemes $V\subset U\subset X$ a morphism of $\O_X(U)$\+modules
$\P[V]\rarrow\P[U]$ such that for every triple of embedded affine open
subschemes $W\subset V\subset U\subset X$ the triangle diagram
$\P[W]\rarrow\P[V]\rarrow\P[U]$ is commutative and the following two
conditions are satisfied:
\begin{enumerate}
\renewcommand{\theenumi}{\roman{enumi}}
\item for any pair of embedded affine open subschemes $V\subset U
\subset X$, the morphism of $\O_X(U)$\+modules $\P[V]\rarrow\P[U]$
induces an isomorphism of $\O_X(V)$\+modules
$$
 \Hom_{\O_X(U)}(\O_X(V),\P[U])\simeq\P[V];
$$
\item for any pair of embedded affine open subschemes $V\subset U
\subset X$, one has
$$
 \Ext^1_{\O_X(U)}(\O_X(V),\P[U])=0.
$$
\end{enumerate}

 The $\O_X(U)$\+module $\O_X(V)$ is very flat (see
Section~\ref{very-flat-definition-introd}), so its projective dimension
is at most~$1$; that is why the groups $\Ext^i$ with $i\ge2$ are not
mentioned in the condition~(ii).
 Moreover, it follows that the condition~(ii) is equivalent to
the condition that the $\P[U]$ should be a contraadjusted
$\O_X(U)$\+module.
 In fact, the category of contraherent cosheaves on an affine scheme
$U$ is equivalent to the category of contraadjusted modules over
the ring~$\O(U)$.

 One can also consider more narrow classes of contraherent cosheaves
defined by imposing stricter conditions on
the $\O_X(U)$\+modules~$\P[V]$.
 An $R$\+module $C$ is called (\emph{Enochs}) \emph{cotorsion}~\cite{En}
if $\Ext^1_R(F,C)=0$ for every flat $R$\+module~$C$.
 A contraherent cosheaf $\P$ is said to be \emph{locally cotorsion} if
the $\O_X(U)$\+module $\P[U]$ is cotorsion for every affine open
subscheme $U\subset X$.
 A contraherent cosheaf $\J$ on a scheme $X$ is said to be \emph{locally
injective} if the $\O_X(U)$\+module $\J[U]$ is injective for every
affine open subscheme $U\subset X$ \cite[Section~2.2]{Pcosh}.

\subsection{{}}
 While very flat morphisms of schemes play a role in the theory of
quasi-coherent sheaves in that the direct image of a very flat
quasi-coherent sheaf with respect to a very flat affine morphism
remains very flat, they play an even more important role in
the theory of contraherent cosheaves in that the very flatness of
a morphism is necessary for the inverse images of contraherent
cosheaves to be defined.
 To be more precise, there are two technical problems associated with
the construction of an inverse image of a contraherent cosheaf:
the nonexactness of colocalizations and the nonlocality of
contraherence.
 To avoid the second problem, which is not relevant to our present
discussion, let us assume that a morphism of schemes $f\:Y\rarrow X$
is \emph{coaffine}, that is, for every affine open subscheme
$V\subset Y$ there exists an affine open subscheme $U\subset X$
such that $f(V)\subset U$.

 In this case, given a contraherent cosheaf $\P$ on $X$, one defines
its inverse image $f^!\P$ on $Y$ as the rule assigning to an affine
open subscheme $V\subset Y$ the $\O_Y(V)$\+module
$\Hom_{\O_X(U)}(\O_Y(V),\P[U])$.
 For this definition to work and produce a contraherent cosheaf on $Y$,
the morphism~$f$ needs to be very flat.
 More narrow classes of contraherent cosheaves can be pulled back under
less restrictive conditions: e.~g., when $\P$ is a locally cotorsion
contraherent cosheaf on $X$, it suffices that $f$~be a flat morphism
of schemes; and the inverse image of locally injective contraherent
cosheaves can be taken with respect to an arbitrary scheme
morphism~\cite[Section~2.3]{Pcosh}.

 Locally injective contraherent cosheaves form a narrow class
dual-analogous to the classes of flat or locally projective
quasi-coherent sheaves.
 The class of locally cotorsion contraherent cosheaves is wider,
and one can restrict oneself to these when working over a Noetherian
or a locally Noetherian scheme, particularly if such a scheme has
finite Krull dimension.
 (In fact, the locally cotorsion contraherent cosheaves over locally
Noetherian schemes are particularly pleasant to work with, owing to
the classification of flat cotorsion modules over Noetherian
rings~\cite{En}.)
 But over a non-Noetherian scheme locally cotorsion contraherent
cosheaves are too few, and arbitrary (locally contraadjusted)
contraherent cosheaves are needed.
 Hence the crucial importance of very flat morphisms of schemes,
and consequently of the results of this paper, in the contraherent
cosheaf theory.

 Finally, the abundance of very flat morphisms of schemes opens
the perspective of developing the theory of locally contraadjusted
contraherent cosheaves on \emph{stacks} (e.~g., described in terms of
flat coverings of finite presentation); cf.~\cite[Appendix~B]{Pcosh}.

\subsection{{}}
 Some applications of the methods originally developed for the purposes
of this paper to the now-classical commutative algebra topic of
\emph{strongly flat modules}~\cite{BS,BS2} are worked out in
the companion paper~\cite{PSl}.
 In particular, arbitrary flat modules over a commutative
Noetherian ring with countable spectrum are described as direct summands
of transfinitely iterated extensions of localizations of the ring at
its countable multiplicative subsets, and a related  description of
the Enochs cotorsion modules is obtained.

\medskip
\textbf{Aknowledgements.}
 This paper grew out of the first author's visits to Prague in 2015--17,
and he wishes to thank to Jan Trlifaj for inviting him.
 Our discussions with Jan Trlifaj also played a particularly important
role in the development of this project.
 The first author is grateful to Silvana Bazzoni,
Jan \v St\!'ov\'\i\v cek, and Amnon Yekutieli for helpful conversations.
 The first author's research is supported by research plan
RVO:~67985840, by the Israel Science Foundation grant~\#\,446/15, and
by the Grant Agency of the Czech Republic under the grant P201/12/G028.
 The second author's research is supported by the Grant Agency of
the Czech Republic under the grant 17-23112S and by the SVV project
under the grant SVV-2017-260456.

\Section{Outline of the Arguments}

\subsection{Main theorem}
 Let $R$ be a commutative ring.
 A commutative $R$\+algebra $S$ is said to be \emph{finitely presented}
if it can be presented as the quotient algebra $R[x_1,\dotsc,x_m]/I$
of the algebra of polynomials $R[x_1,\dots,x_m]$ in a finite number of
variables with the coefficients in $R$ by a finitely generated ideal
$I\subset R[x_1,\dotsc,x_m]$.

 We refer to Section~\ref{very-flat-definition-introd} of
the Introduction above and~\cite[Section~1.1 and~C.2]{Pcosh} for
the definitions related to very flat and contraadjusted $R$\+modules
(see also~\cite[Sections~2 and~5]{ST}).
 The following theorem is the main result of this paper.
 Its particular case for the Noetherian rings (cf.\
Main Theorem~\ref{noetherian-module-main-theorem} below) was originally
proposed as a conjecture in~\cite[Conjecture~1.7.2]{Pcosh}.

\begin{mt} \label{general-algebra-main-theorem}
 Let $R$ be a commutative ring and $S$ be a finitely presented
commutative $R$\+algebra.
 Assume that $S$ is a flat $R$\+module.
 Then $S$ is a very flat $R$\+module.
\end{mt} 

 Given a ring $S$, an $S$\+module is called \emph{finitely presented}
if it is isomorphic to the cokernel of a morphism between two finitely
generated free $S$\+modules.
 The following generalization of
Main Theorem~\ref{general-algebra-main-theorem} is useful
in some applications.

\begin{mt} \label{general-module-main-theorem}
 Let $R$ be a commutative ring, $S$ be a finitely presented commutative
$R$\+algebra, and $F$ be a finitely presented $S$\+module.
 Assume that $F$ is a flat $R$\+module.
 Then $F$ is a very flat $R$\+module.
\end{mt} 

\subsection{Noetherian main lemma} \label{noetherian-outline}
 There are two proofs of Main Theorem~\ref{general-module-main-theorem}
given in this paper.
 One of them works for Noetherian rings only.
 Let us formulate this assertion before explaining the key ideas
behind its proof.

\begin{mt} \label{noetherian-module-main-theorem}
 Let $R$ be a Noetherian commutative ring, $S$ be a finitely generated
commutative $R$\+algebra, and $F$ be a finitely generated $S$\+module.
 Assume that $F$ is a flat $R$\+module.
 Then $F$ is a very flat $R$\+module.
\end{mt} 

 The proof of Main Theorem~\ref{noetherian-module-main-theorem} is
based on the following main lemma.
 For any element $r\in R$ and an $R$\+module $M$, we denote by
$M[r^{-1}]$ the $R$\+module $R[r^{-1}]\ot_RM$.

\begin{ml} \label{noetherian-main-lemma}
 Let $R$ be a Noetherian commutative ring, $r\in R$ be an element,
and $F$ be a flat $R$\+module.
 Then the $R$\+module $F$ is very flat if and only if
the $R/rR$\+module $F/rF$ is very flat and the $R[r^{-1}]$\+module
$F[r^{-1}]$ is very flat.
\end{ml}

 The ``only if'' assertion in Main Lemma~\ref{noetherian-main-lemma}
is easy (all extensions of scalars preserve very flatness).
 The ``if'' assertion is of key importance.

 The argument deducing Main Theorem~\ref{noetherian-module-main-theorem}
from Main Lemma~\ref{noetherian-main-lemma} is presented in
Section~\ref{implies-main-theorem-secn}.
 It uses Noetherian induction and Grothendieck's generic freeness lemma.

\subsection{Finitely very flat main lemma}  \label{fvf-outline}
 The other proof of Main Theorem~\ref{general-module-main-theorem}
not only works in full generality, but also provides a stronger result.
 In order to formulate it, we will need the following definition.

 Let $r_1$,~\dots, $r_m\in R$ be a finite collection of elements.
 We will denote it for brevity by a single letter~$\r$.
 Let us say that an $R$\+module $C$ is \emph{$\r$\+contraadjusted} if
it is $r_j$\+contraadjusted for every $j=1$,~\dots,~$m$.
 An $R$\+module $F$ is said to be \emph{$\r$\+very flat} if
$\Ext^1_R(F,C)=0$ for all $\r$\+contraadjusted $R$\+modules~$C$.
 An $R$\+module is $\r$\+very flat if and only if it is a direct summand
of a transfinitely iterated extension (in the sense of the inductive
limit) of $R$\+modules isomorphic to $R$ or $R[r_j^{-1}]$, \
$1\le j\le m$ (see~\cite{ET}, \cite[Corollary~6.14]{GT},
or~\cite[Corollary~1.1.4]{Pcosh}).

 Let us say that an $R$\+module $F$ is \emph{finitely very flat} if
there exists a finite collection of elements $\r=\{r_1,\dots,r_m\}$ in
$R$ such that $F$ is $\r$\+very flat.
 To repeat, it means that $F$ is a direct summand of a module admitting
a filtration indexed by an arbitrarily large ordinal with all
the successive quotients belonging, up to isomorphism, to the finite
set of $R$\+modules $R$ and $R[r_j^{-1}]$.

\begin{mt} \label{fvf-module-main-theorem}
 Let $R$ be a commutative ring, $S$ be a finitely presented
commutative $R$\+algebra, and $F$ be a finitely presented $S$\+module.
 Assume that $F$ is a flat $R$\+module.
 Then $F$ is a finitely very flat $R$\+module.
\end{mt}

 Actually, there is more to be said about $R$\+modules such as
in Main Theorem~\ref{fvf-module-main-theorem}.
 Any such $R$\+module $F$ can be presented by a countable set of
generators with a countable set of relations.
 Using the Hill Lemma~\cite[(proof of) Lemma 7.10(H4)]{GT}, one can
deduce that $F$ is a direct summand of an $R$\+module admitting
a filtration as above indexed by a \emph{countable} ordinal.

 Obviously, Main Theorem~\ref{fvf-module-main-theorem} implies
Main Theorem~\ref{general-module-main-theorem}
(and Main Theorem~\ref{general-module-main-theorem} implies
Main Theorem~\ref{general-algebra-main-theorem}).
 The proof of Main Theorem~\ref{fvf-module-main-theorem} is based on
the next main lemma.

\begin{ml} \label{fvf-main-lemma}
 Let $R$ be a commutative ring, $r\in R$ be an element, and $F$ be
a flat $R$\+module.
 Then the $R$\+module $F$ is finitely very flat if and only if
the $R/rR$\+module $F/rF$ is finitely very flat and the
$R[r^{-1}]$\+module $F[r^{-1}]$ is finitely very flat.
\end{ml}

 Once again, the ``only if'' assertion is easy (all extensions of
scalars preserve finite very flatness).
 The ``if'' assertion is important.

 The argument deducing Main Theorem~\ref{fvf-module-main-theorem}
from Main Lemma~\ref{fvf-main-lemma} is also explained in
Section~\ref{implies-main-theorem-secn}.
 It also uses Noetherian induction (even though the rings involved
are not Noetherian!)
 The point is that any finitely presented algebra over a commutative
ring can be obtained by an extension of scalars from a finitely
presented algebra over a ring finitely generated over the integers.

\subsection{Bounded torsion main lemma} \label{bounded-torsion-outline}
 Notice that Main Lemma~\ref{noetherian-main-lemma} does not follow
from Main Lemma~\ref{fvf-main-lemma}.
 One of the reasons why we present two proofs of the main theorem
in this paper is because we like the very flat version of the main
lemma more than the finitely very flat one.

 We do not know whether the Noetherianity condition can be dropped
in Main Lemma~\ref{noetherian-main-lemma}, but it can be weakened.
 Given an element $r$ in a commutative ring $R$ and an $R$\+module
$M$, we say that \emph{the $r$\+torsion in $M$ is bounded} if there
exists an integer $m\ge1$ such that $r^nx=0$ implies $r^mx=0$ for
all $x\in M$ and $n\ge1$.
 The following result is provable with our methods.

\begin{ml} \label{bounded-torsion-main-lemma}
 Let $R$ be a commutative ring and $r\in R$ be an element such that
the $r$\+torsion in $R$ is bounded.
 Let $F$ be a flat $R$\+module.
 Then the $R$\+module $F$ is very flat if and only if
the $R/rR$\+module $F/rF$ is very flat and the $R[r^{-1}]$\+module
$F[r^{-1}]$ is very flat.
\end{ml}

 The condition on the element $r\in R$ in
Main Lemma~\ref{bounded-torsion-main-lemma} can be relaxed somewhat
further.
 We refer to Remark~\ref{bounded-torsion-weakened-remark}
for the discussion.

\subsection{Obtaining contraadjusted modules} \label{obtaining-outline}
 Let us now say a few words about the proofs of the main lemmas.
 The trick is that these proofs happen on the contraadjusted rather
than on the very flat side.
 One does not really do anything with the given flat module $F$.
 Instead, one works with an arbitrary contraadjusted $R$\+module $C$
(in the case of Main Lemmas~\ref{noetherian-main-lemma}
and~\ref{bounded-torsion-main-lemma}) or with an arbitrary
$\r$\+contraadjusted $R$\+module $C$ (in the case of
Main Lemma~\ref{fvf-main-lemma}).

 The contraadjusted modules are described in terms of a certain
generation procedure.
 Given some class of $R$\+modules $\sE\subset R\modl$, we produce
a possibly larger class of $R$\+modules $\sC$ such that
whenever one has $\Ext_R^{>0}(F,E)=0$ for a given $R$\+module $F$ and
all $E\in\sE$, one also has $\Ext_R^{>0}(F,C)=0$ for all $C\in\sC$.

 In each of the main lemmas, in order to prove that the $R$\+module
$F$ is (finitely) very flat, one needs to show that
$\Ext^{>0}_R(F,C)=0$ for every contraadjusted (or $\r$\+contraadjusted)
$R$\+module~$C$.
 After all such $R$\+modules $C$ are known to be obtainable using
a generation procedure, it remains to check that
$\Ext^{>0}_R(F,E)=0$ for all the $R$\+modules $E$ belonging to
the ``seed class''~$\sE$.
 The latter is chosen in such a way that this follows pretty
straightforwardly from the conditions imposed on the $R$\+module~$F$
in the respective main lemma.

 The discussion of obtainable modules in this paper occupies two
sections.
 A simpler version of obtainability sufficient for the purposes of
Toy Main Lemma~\ref{toy-main-lemma} (see below) is introduced in
Section~\ref{obtainable-I-secn}; and a more sophisticated version
required for Main Lemmas~\ref{noetherian-main-lemma},
\ref{fvf-main-lemma}, and~\ref{bounded-torsion-main-lemma} is
subsequently defined in Section~\ref{obtainable-II-secn}.

\subsection{Toy main lemma}
 As the idea of the proof described in such terms as above may seem
to be too abstract and opaque, we start with a toy example.
 The proof of the following version of the main lemma is technically
much easier, still it demonstrates many of the same key features
that manifest themselves in the proofs of
Main Lemmas~\ref{noetherian-main-lemma},
\ref{bounded-torsion-main-lemma}, and~\ref{fvf-main-lemma}.

 Given a commutative ring $R$ with an element $r\in R$, let us say
that an $R$\+module $F$ is \emph{$r$\+very flat} if
$\Ext^1_R(F,C)=0$ for every $r$\+contraadjusted $R$\+module~$C$.
 This is the special case of the above definition of an $\r$\+very flat
module corresponding to the situation when the finite set of elements
$\r$ consists of a single element, $\r=\{r\}$.

 Since one has $\Ext_R^1(R,\>\bigoplus_{x\in X}R[r^{-1}])=0=
\Ext_R^1(R[r^{-1}],\>\bigoplus_{x\in X}R[r^{-1}])$ for any index set~$X$,
the description of $r$\+very flat $R$\+modules as the direct summands
of transfinitely iterated extensions reduces to the following
particularly simple form.
 An $R$\+module $F$ is $r$\+very flat if and only if it is a direct
summand of an $R$\+module $G$ such that there exists a short exact
sequence of $R$\+modules
$$
 0\lrarrow P\lrarrow G\lrarrow Q\lrarrow 0,
$$
where $P$ is a free $R$\+module and $Q$ is a free
$R[r^{-1}]$\+module~\cite[Corollary~6.13]{GT}.

\begin{tml} \label{toy-main-lemma}
 Let $R$ be a commutative ring, $r\in R$ be an element, and $F$ be
a flat $R$\+module.
 Then the $R$\+module $F$ is $r$\+very flat if and only if
the $R/rR$\+module $F/rF$ is projective and the $R[r^{-1}]$\+module
$F[r^{-1}]$ is projective.
\end{tml}

 The proof of Toy Main Lemma~\ref{toy-main-lemma} is given
in Section~\ref{toy-secn}.

\subsection{Contramodules} \label{contramodules-outline}
 Let $R$ be a commutative ring and $r\in R$ be an element.
 An $R$\+module $C$ is said to be an \emph{$r$\+contramodule}
\cite[Section~C.2]{Pcosh}, \cite{Pmgm}, \cite{Pcta} if
$\Hom_R(R[r^{-1}],C)=0=\Ext^1_R(R[r^{-1}],C)$.
 This is the key technical concept used in the proofs of all
the main lemmas.

 In the case when the $r$\+torsion in $R$ is bounded, $r$\+contramodule
$R$\+modules are also known under the name of \emph{cohomologically
$I$\+adically complete} modules~\cite{PSY,Yek}, where $I=(r)\subset R$
denotes the principal ideal generated by the element $r\in R$.
 For an arbitrary element~$r$ in a commutative ring $R$,
an $R$\+module $C$ is said to be \emph{$r$\+complete} if the natural
map from $C$ to its $r$\+adic completion $C\rarrow\varprojlim_n C/r^nC$
is surjective, and \emph{$r$\+separated} if this map is injective.

 All $r$\+contramodule $R$\+modules (and more generally, all
$r$\+contraadjusted $R$\+mod\-ules) are $r$\+complete, and all
$r$\+separated $r$\+complete $R$\+modules are $r$\+contramodules,
but the converse implications do not hold.
 The existence of nonseparated contramodules is an important technical
issue which we have to deal with throughout this paper (in particular,
in Sections~\ref{toy-secn}, \ref{contramodule-approx-secn}
and~\ref{fvf-mlemma-secn}).

 The category of $r$\+contramodule $R$\+modules is an abelian
subcategory in $R\modl$.
 So is the category of $R[r^{-1}]$\+modules.
 The idea of the proof of Toy Main Lemma~\ref{toy-main-lemma} is to
show that every $r$\+contraadjusted $R$\+module is obtainable
(in the sense discussed in Section~\ref{obtaining-outline}) from
$r$\+contramodule $R$\+modules and $R[r^{-1}]$\+modules.
 Then one shows that all the $r$\+contramodule $R$\+modules are
obtainable from $R/rR$\+modules.

 The proofs of Main Lemmas~\ref{noetherian-main-lemma}
and~\ref{bounded-torsion-main-lemma} are more difficult because one has
to perform the steps staying inside the class of contraadjusted
$R$\+modules.
 So first one shows that every contraadjusted $R$\+module is obtainable
from a contraadjusted $r$\+contramodule $R$\+module and a contraadjusted
$R[r^{-1}]$\+module, and then one needs to obtain a contraadjusted
$r$\+contramodule $R$\+module out of contraadjusted $R/rR$\+modules.
 It is the latter step that depends on the Noetherianity/bounded torsion
assumptions on~$R$.

 The homological techniques of working with contramodules that are
required for this step were originally developed
in~\cite[Sections~C.2 and~D.4]{Pcosh}.
 These were originally intended to prepare ground for developing
the theory of contraherent cosheaves of contramodules over formal
schemes and ind-schemes.
 We discuss these constructions for $r$\+contramodule $R$\+modules
in a longish Section~\ref{contramodule-approx-secn}.
 Then we proceed to prove Main Lemmas~\ref{noetherian-main-lemma}
and~\ref{bounded-torsion-main-lemma} in
Section~\ref{noetherian-mlemma-secn}.

\subsection{Good finite subsets in a ring} \label{good-subsets-outline}
 There is a caveat to the discussion of $\r$\+contraadjusted
$R$\+modules in Section~\ref{fvf-outline}, and particularly to
the mentioning of ``an arbitrary $\r$\+contraadjusted $R$\+module''
in Section~\ref{obtaining-outline}, that we should now explain.
 The $\r$\+contraadjusted $R$\+module, in the context of
the discussion in Section~\ref{obtaining-outline}, is arbitrary,
but the finite subset $\r\subset R$ is not.

 One can say that a finite subset $\r\subset R$ is \emph{good} if,
for every pair of elements $r'$ and $r''\in\r$, there exists
an element $r\in\r$ such that the $R$\+algebras $R[(r'r'')^{-1}]$
and $R[r^{-1}]$ are isomorphic to each other.
 In this paper, we essentially only work with good finite sets of
elements~$\r$ in commutative rings~$R$.

 In more down-to-earth terms, this is means that, given a finite
set of elements $r_1$,~\dots, $r_m\in R$, the first thing we want
to do with it is to embed it into a certain larger finite set
of elements.
 For every subset $J\subset\{1,\dots,m\}$, denote by $r_J$
the product $\prod_{j\in J}r_j$.
 The set of elements $\{r_J\mid J\subset\{1,\dotsc,m\}\,\}$ is good.

 In particular, we can now explain how to produce the specific finite
set of elements in the ring $R$ appearing in
Main Lemma~\ref{fvf-main-lemma}.
 Let $\s=\{s_1,\dots,s_p\}$ be a finite set of elements in
the quotient ring $R/rR$ such that the $R/rR$\+module $F/rF$ is
$\s$\+very flat, and let $\t=\{t_1,\dotsc,t_q\}$ be a finite set of
elements in the ring of fractions $R[r^{-1}]$ such that
the $R[r^{-1}]$\+module $F[r^{-1}]$ is $\t$\+very flat.

 For every $1\le i\le p$, choose a preimage $\tilde s_i\in R$ of
the element $s_i\in R/rR$.
 For every $1\le l\le q$, choose an element $\tilde t_l\in R$ and
an integer $n_l\ge0$ such that $t_l=\tilde t_l/r^{n_l}$ in $R[r^{-1}]$.
 Consider the collection of all elements $r$, $\tilde s_1$,~\dots,
$\tilde s_p$, $\tilde t_1$,~\dots, $\tilde t_q$ in $R$, and form
a good finite subset $\r\subset R$ containing all these elements.

 Then the $R$\+module $F$ is $\r$\+very flat.
 This is what (the proof of) Main Lemma~\ref{fvf-main-lemma}
actually tells.

\subsection{$\r^\times$-very flat theorem} \label{rtimes-vf-thm-outline}
 In order to prove Main Lemma~\ref{fvf-main-lemma}, we reformulate
it in the following form, which is a kind of multi-element version
of Toy Main Lemma~\ref{toy-main-lemma}.

 We will use the following notation.
 Let $r_1$,~\dots, $r_m$ be a finite set of elements in a commutative
ring~$R$.
 As above, for every subset $J\subset\{1,\dotsc,m\}$ we put
$r_J=\prod_{j\in J}r_j$.
 Furthermore, put $K=\{1,\dotsc,m\}\setminus J$ and denote by
$R_J=(R/\sum_{k\in K}r_kR)[r_J^{-1}]$ the ring obtained by annihilating
all the elements $r_k$, $k\in K$, and inverting all the elements
$r_j$, $j\in J$ in the ring~$R$.
 Let us denote the finite set of elements $\{r_J\mid J\subset
\{1,\dotsc,m\}\,\}$ in the ring $R$ by~$\r^\times$.

\begin{thm} \label{r-very-flat-theorem}
 Let $R$ be a commutative ring, $r_1$,~\dots, $r_m\in R$ be a finite
set of its elements, and $F$ be a flat $R$\+module.
 Then the $R$\+module $F$ is\/ $\r^\times$\+very flat if and only if
the $R_J$\+module $R_J\ot_RF$ is projective for every subset
$J\subset\{1,\dotsc,m\}$.
\end{thm}

 Deducing Main Lemma~\ref{fvf-main-lemma} from
Theorem~\ref{r-very-flat-theorem} is easy, but the proof
of Theorem~\ref{r-very-flat-theorem} is quite involved.
 We present both these arguments in Section~\ref{fvf-mlemma-secn}.

\subsection{Flat lemma}
 Looking into the argument outlined in Section~\ref{obtaining-outline}
suggests that the assumptions in the main lemmas can be weakened.
 Indeed, given that we assume both the $R/rR$\+module $F/rF$ and
the $R[r^{-1}]$\+module $F[r^{-1}]$ to be, at least, very flat, why
do we need to assume the $R$\+module $F$ to be flat, on top of that?

 In fact, the condition that the $R$\+module $F$ is flat can be
replaced in these assertions by the condition of vanishing of
certain specific $\Tor$ modules related to the element~$r$.
 The following lemma provides more information, however.

\begin{lem} \label{flat-lemma}
 Let $R$ be a commutative ring and $r\in R$ be an element.
 Let $F$ be an $R$\+module such that $\Tor^R_1(R/rR,F)=0=
\Tor^R_2(R/rR,F)$.
 Then the $R$\+module $F$ is flat if and only if the $R/rR$\+module
$F/rF$ is flat and the $R[r^{-1}]$\+module $F[r^{-1}]$ is flat.
\end{lem}

 Lemma~\ref{flat-lemma} is even easier to prove than
Toy Main Lemma~\ref{toy-main-lemma}, as there are no contramodules
in the proof of Lemma~\ref{flat-lemma}.
 Instead, $r$\+torsion modules are used.
 We prove both these ``toy'' lemmas in Section~\ref{toy-secn}.

\subsection{Surjective descent}
 Notice that for any element $r$ in a commutative ring $R$ the ring
$R/rR\oplus R[r^{-1}]$ is a finitely presented commutative $R$\+algebra.
 Furthermore, the ring homomorphism $R\rarrow R/rR\oplus R[r^{-1}]$
induces a bijective map (but, of course, not a homeomorphism) of
the spectra
$$
 \Spec R/rR\.\sqcup\.\Spec R[r^{-1}]\.=\.\Spec(R/rR\oplus R[r^{-1}])
 \lrarrow \Spec R.
$$
 The following two results generalize
Main Lemmas~\ref{noetherian-main-lemma}
and~\ref{fvf-main-lemma}, respectively.

\begin{thm} \label{noetherian-surj-descent-thm}
 Let $R$ be a Noetherian commutative ring and $S$ be a finitely
generated commutative $R$\+algebra such that the induced
map of the spectra\/ $\Spec S\rarrow\Spec R$ is surjective.
 Let $F$ be a flat $R$\+module.
 Then the $R$\+module $F$ is very flat if and only if
the $S$\+module $S\ot_RF$ is very flat.
\end{thm}

\begin{thm} \label{fvf-surj-descent-thm}
 Let $R$ be a commutative ring and $S$ be a finitely presented
commutative $R$\+algebra such that the induced map of the spectra\/
$\Spec S\rarrow\Spec R$ is surjective.
 Let $F$ be a flat $R$\+module.
 Then the $R$\+module $F$ is finitely very flat if and only if
the $S$\+module $S\ot_RF$ is finitely very flat.
\end{thm}

 The proofs of Theorems~\ref{noetherian-surj-descent-thm}
and~\ref{fvf-surj-descent-thm} are based on the main lemmas which
they generalize.
 Theorem~\ref{noetherian-surj-descent-thm} is deduced from
Main Lemma~\ref{noetherian-main-lemma}, and
Theorem~\ref{fvf-surj-descent-thm} is deduced from
Main Lemma~\ref{fvf-main-lemma}.
 This is done in Section~\ref{descent-secn}.

 The arguments use Noetherian induction very similar to the one that
we use in order to deduce the main theorems from the main lemmas, and
the same generic freeness lemma (cf.\ the discussion in
Sections~\ref{noetherian-outline}\+-\ref{fvf-outline}).

\Section{Main Lemma implies Main Theorem}
\label{implies-main-theorem-secn}

 The proofs of the main theorems are based on the following generic
freeness result.

\begin{lem} \label{generic-freeness-lemma}
 Let $R$ be a Noetherian commutative integral domain, $S$ be a finitely
generated commutative $R$\+algebra, and $F$ be a finitely generated
$S$\+module.
 Then there exists a nonzero element $a\in R$ such that
the $R[a^{-1}]$\+module $F[a^{-1}]$ is free.
\end{lem}

\begin{proof}
 This is~\cite[Lemme~6.9.2]{Groth}.
 For a generalization, see~\cite[Theorem~24.1]{Mats}.
\end{proof}

 For ease of reference, let us also formulate here the assertion
about preservation of (finite) very flatness by extensions of scalars.

\begin{lem} \label{very-flatness-extension-of-scalars}
\textup{(a)} Let $R\rarrow R'$ be a morphism of commutative rings
and $F$ be a very flat $R$\+module.
 Then the $R'$\+module $R'\ot_RF$ is also very flat. \par
\textup{(b)} Let $R\rarrow R'$ be a morphism of commutative rings
and $F$ be a finitely very flat $R$\+module.
 Then the $R'$\+module $R'\ot_RF$ is also finitely very flat.
\end{lem}

\begin{proof}
 Part~(a) is~\cite[Lemma~1.2.2(b)]{Pcosh}, and the proof of part~(b)
is similar.
 If $\r=\{r_1,\dotsc,r_m\}$ is a finite set of elements in $R$ such
that the $R$\+module $F$ is $\r$\+very flat, then
the $R'$\+module $R'\ot_RF$ is $\r'$\+very flat, where
$\r'=\{r'_1,\dotsc,r'_m\}$ and $r'_j\in R'$ denotes the image of
the element $r_j\in R$ under the ring homomorphism $R\rarrow R'$.
\end{proof}

 The proofs in this section are also based on the classical technique
of Noetherian induction, which can be formulated as follows.

\begin{nip}
 Let $R$ be a Noetherian commutative ring.
 Then there \emph{cannot} exist a sequence of rings $R_n$, $n\ge0$,
and nonzero elements $r_n\in R_n$ such that $R_0=R$ and
$R_{n+1}=R_n/r_nR_n$ for all $n\ge0$.
\end{nip}

\begin{proof}
 Denote by $\tilde r_n\in R$ some preimages of the elements
$r_n\in R_n$, and let $I_n$ be the ideal in $R$ generated
by the elements $\tilde r_0$,~\dots,~$\tilde r_{n-1}$.
 Then we have $R/I_n\simeq R_n$, and the ideals $I_n\subset R$
form an infinite ascending chain $0\varsubsetneq I_0\varsubsetneq I_1
\varsubsetneq I_2\varsubsetneq\dotsb\subset R$, contradicting
the assumption of Noetherianity of the ring~$R$.
\end{proof}

 We start with explaining the proof of
Main Theorem~\ref{noetherian-module-main-theorem}, which highlights
the key ideas of the argument.
 Then we proceed to prove the stronger
Main Theorem~\ref{fvf-module-main-theorem}.

 The following proposition is a particular case of Main
Theorem~\ref{noetherian-module-main-theorem} from which the general
case is deduced.

\begin{prop} \label{noetherian-main-theorem-proof-prop}
 Let $R$ be a Noetherian commutative ring, $S$ be a finitely generated
commutative $R$\+algebra, and $F$ be a finitely generated $S$\+module.
 Assume that $F$ is a flat $R$\+module, and that for every nonzero
element $r\in R$ the $R/rR$\+module $F/rF$ is very flat.
 Then $F$ is a very flat $R$\+module.
\end{prop}

\begin{proof}
 Case~I: suppose that the ring $R$ has zero-divisors, that is, there
exists a pair of nonzero elements $a$, $b\in R$ such that $ab=0$ in~$R$.
 By assumption, the $R/aR$\+module $F/aF$ is very flat and
the $R/bR$\+module $F/bF$ is very flat.
 Now the localization morphism $R\rarrow R[a^{-1}]$ factorizes as
$R\rarrow R/bR\rarrow R[a^{-1}]$, hence very flatness of
the $R/bR$\+module $F/bF$ implies very flatness of
the $R[a^{-1}]$\+module $F[a^{-1}]$ by
Lemma~\ref{very-flatness-extension-of-scalars}(a).
 Both the $R/aR$\+module $F/aF$ and the $R[a^{-1}]$\+module $F[a^{-1}]$ 
being very flat, very flatness of the $R$\+module $F$ follows by
Main Lemma~\ref{noetherian-main-lemma}.

 Case~II: now suppose that the ring $R$ is an integral domain.
 By Lemma~\ref{generic-freeness-lemma}, there exists a nonzero element
$a\in R$ such that the $R[a^{-1}]$\+module $F[a^{-1}]$ is free.
 By assumption, the $R/aR$\+module $F/aF$ is very flat.
 Once again, very flatness of the $R$\+module $F$ follows by virtue
of Main Lemma~\ref{noetherian-main-lemma}.
\end{proof}

\begin{proof}[Proof of
Main Theorem~\ref{noetherian-module-main-theorem}]
 Assume that the $R$\+module $F$ is \emph{not} very flat.
 By Proposition~\ref{noetherian-main-theorem-proof-prop}, it then
follows that there exists a nonzero element $r\in R$ such that
the $R/rR$\+module $F/rF$ is not very flat.

 Now $R/rR$ is a Noetherian ring, $S/rS$ is a finitely
generated $R/rR$\+algebra, and $F/rF$ is a finitely generated
$S/rS$\+module.
 Also, $F/rF$ is a flat $R/rR$\+module.
 Set $R_1=R/rR$, \ $S_1=S/rS$, and $F_1=F/rF$.
 Applying Proposition~\ref{noetherian-main-theorem-proof-prop} again,
we find a nonzero element $r_1\in R_1$ such that
the $R_1/r_1R_1$\+module $F_1/r_1F_1$ is not very flat.

 Proceeding in this fashion, we produce an infinite sequence of
nonzero elements $r_n\in R_n$ and quotient rings $R_{n+1}=R_n/r_nR_n$,
starting from $R_0=R$ and $r_0=r$.
 According to the Noetherian Induction Principle, this is impossible.
 The contradiction proves that the $R$\+module $F$ is very flat.
\end{proof}

 The proof of Main Theorem~\ref{fvf-module-main-theorem} is based on
the following lemma, which allows to apply Noetherian induction
to some problems involving non-Noetherian rings.

\begin{lem} \label{finitely-generated-ring-lemma}
 Let $R$ be a commutative ring, $S$ be a finitely presented
commutative $R$\+algebra, and $F$ be a finitely presented $S$\+module.
 Then there exist a commutative ring $\oR$ finitely generated over
the ring of integers\/ $\boZ$, a ring homomorphism $\oR\rarrow R$,
a finitely generated commutative $\oR$\+algebra $\oS$, and a finitely
generated $\oS$\+module $\oF$ such that $S=R\ot_\oR\oS$ and
$F=S\ot_\oS\oF=R\ot_\oR\oF$.
\end{lem}

\begin{proof}
 Let $I\subset R[x_1,\dotsc,x_m]$ be a finitely generated ideal such
that the $R$\+algebra $R[x_1,\dotsc,x_m]/I$ is isomorphic to~$S$,
and let $N\subset S^n$ be a finitely generated submodule such that
the $S$\+module $S^n/N$ is isomorphic to~$F$.
 Let $f_1(x_1,\dotsc,x_m)$,~\dots, $f_k(x_1,\dotsc,x_m)\in
R[x_1,\dotsc,x_m]$ be a finite set of generators of the ideal~$I$,
and let $(v_{1,1},\dotsc,v_{1,n})$,~\dots, $(v_{l,1},\dotsc,v_{l,n})\in
S^n$ be a finite set of generators of the $S$\+module~$N$.
 Finally, let $q_{i,j}\in R[x_1,\dotsc,x_m]$ be some preimages of
the elements $v_{i,j}\in S$.

 Let $\{r_\alpha\}$~denote the set of all coefficients of
the polynomials $f_h(x_1,\dotsc,x_m)$, \,$1\le h\le k$, and
$q_{i,j}(x_1,\dotsc,x_m)$, \,$1\le i\le l$, \,$1\le j\le n$.
 So $\{r_\alpha\}$~is a finite set of elements of the ring~$R$.
 Set $\oR$ to be the subring in $R$ generated by
the elements~$\{r_\alpha\}$.
 Then the polynomials $f_h$ and~$q_{ij}$ belong to the subring
$\oR[x_1,\dotsc,x_m]\subset R[x_1,\dotsc,x_m]$.

 Now we can set $\oI$ to be the ideal in $\oR[x_1,\dotsc,x_m]$
generated by the polynomials~$f_1$,~\dots, $f_k$, and put
$\overline S=\oR[x_1,\dotsc,x_m]/\oI$.
 Furthermore, let $\bar v_{ij}\in \oS$ denote the images of the elements
$q_{ij}\in\oR[x_1,\dotsc,x_m]$.
 Set $\oN$ to be the submodule in $\oS^n$ generated by the vectors
$(\bar v_{1,1},\dotsc,\bar v_{1,n})$,~\dots,
$(\bar v_{l,1},\dotsc,\bar v_{l,n})$, and put $\oF=\oS^n/\oN$.

 Alternatively, one could set $\oR$ to be the polynomial ring
$\boZ[z_\alpha]$ in the free variables~$z_\alpha$ corresponding
bijectively to the elements $r_\alpha\in R$.
 Then one would have a commutative ring homomorphism $\oR\rarrow R$
taking $z_\alpha$ to~$r_\alpha$.
 It would remain to lift the polynomials $f_h$ and $q_{ij}\in
R[x_1,\dotsc,x_m]$ to some polynomials $\bar f_h$ and $\bar q_{ij}\in
\oR[x_1,\dotsc,x_m]$, and proceed with the constructions of
an $\oR$\+algebra $\oS$ and an $\oS$\+module $\oF$ in the way
similar to the above.
\end{proof}

\begin{rem}
 Notice that, when the $R$\+module $F$ is flat, there is \emph{no}
claim in Lemma~\ref{finitely-generated-ring-lemma}
that the $\oR$\+module $\oF$ is flat.
\end{rem}

 Similarly to the above, we now proceed to formulate a particular
case of Main Theorem~\ref{fvf-module-main-theorem} from which
the general case will be deduced.

\begin{prop} \label{fvf-main-theorem-proof-prop}
 Let $\oR$ be a Noetherian commutative ring, $\oS$ be a finitely
generated commutative $\oR$\+algebra, and $\oF$ be a finitely generated
$\oS$\+module.
 Let $R$ be a commutative $\oR$\+algebra.
 Assume that $F=R\ot_\oR\oF$ is a flat $R$\+module, and that for every
nonzero element $r\in\oR$ the $R/rR$\+module $F/rF$ is finitely
very flat.
 Then $F$ is a finitely very flat $R$\+module.
\end{prop}

\begin{proof}
 Case~I: suppose that the ring $\oR$ has zero-divisors, that is,
there exists a pair of elements $a\ne0\ne b$ in $\oR$ such that
$ab=0$ in~$\oR$.
 By assumption, the $R/aR$\+module $F/aF$ is finitely very flat and
the $R/bR$\+module $F/bF$ is finitely very flat.
 Arguing as in the proof of
Proposition~\ref{noetherian-main-theorem-proof-prop} and using
Lemma~\ref{very-flatness-extension-of-scalars}(b), we see that
the $R[a^{-1}]$\+module $F[a^{-1}]$ is finitely very flat.
 Both the $R/aR$\+module $F/aF$ and the $R[a^{-1}]$\+module
$F[a^{-1}]$ being finitely very flat, finite very flatness of
the $R$\+module $F$ follows by Main Lemma~\ref{fvf-main-lemma}.

 Case~II: now suppose that the ring $\oR$ is an integral domain.
 By Lemma~\ref{generic-freeness-lemma}, there exists a nonzero element
$a\in\oR$ such that the $\oR[a^{-1}]$\+module $\oF[a^{-1}]$ is free.
 It follows that the $R[a^{-1}]$\+module $F[a^{-1}]\simeq
R\ot_\oR\oF[a^{-1}]$ is free as well.
 By assumption, the $R/aR$\+module $F/aF$ is finitely very flat.
 As above, finite very flatness of the $R$\+module $F$ follows
by means of Main Lemma~\ref{fvf-main-lemma}.
\end{proof}

\begin{proof}[Proof of Main Theorem~\ref{fvf-module-main-theorem}]
 Starting from a commutative ring $R$, a finitely presented commutative
$R$\+algebra $S$, and a finitely presented $S$\+module $F$, we first
of all apply Lemma~\ref{finitely-generated-ring-lemma}.
 This produces a Noetherian commutative ring $\oR$ with a ring
homomorphism $\oR\rarrow R$, a finitely generated commutative
$\oR$\+algebra $\oS$, and a finitely generated $\oS$\+module~$\oF$
such that $F\simeq R\ot_\oR\oF$.

 By the assumptions of the theorem, the $R$\+module $F$ is flat.
 Assume that it is \emph{not} finitely very flat.
 By Proposition~\ref{fvf-main-theorem-proof-prop}, it then follows
that there exists a nonzero element $r\in\oR$ such that
the $R/rR$\+module $F/rF$ is not finitely very flat.

 Now $\oR/r\oR$ is a Noetherian ring, $\oS/r\oS$ is a finitely
generated $\oR/r\oR$\+algebra, and $\oF/r\oF$ is a finitely generated
$\oS/r\oS$\+module.
 Furthermore, $R/rR$ is an $\oR/r\oR$\+algebra and $F/rF\simeq
R/rR\ot_{\oR/r\oR}\oF/r\oF$ is a flat $R/rR$\+module.
 Set $\oR_1=\oR/r\oR$, \ $\oS_1=\oS/r\oS$, \ $\oF_1=\oF/r\oF$, \
$R_1=R/rR$, and $F_1=F/rF$.
 Applying Proposition~\ref{fvf-main-theorem-proof-prop}, we find
a nonzero element $r_1\in\oR_1$ such that the $R_1/rR_1$\+module
$F_1/rF_1$ is not very flat.

 Proceeding in this way, we produce an infinite sequence of nonzero
elements $r_n\in\oR_n$ and quotient rings $\oR_{n+1}=\oR_n/r_n\oR_n$,
starting from $\oR_0=\oR$ and $r_0=r$.
 The argument by contradiction with Noetherian Induction Principle
(applied to the Noetherian ring~$\oR$) finishes in the same way as
in the above proof of Main Theorem~\ref{noetherian-module-main-theorem}.
\end{proof}

\Section{Obtainable Modules I}  \label{obtainable-I-secn}

 In this section, we discuss the ``simple obtainability''.
 The discussion of the more complicated ``two-sorted obtainability''
is postponed to Section~\ref{obtainable-II-secn}.

 In this section we work over an associative (not necessarily
commutative) ring~$R$.
 We denote by $R\modl$ the abelian category of left $R$\+modules
and by $\modr R$ the abelian category of right $R$\+modules.

 Let $\sE$ and $\sF\subset R\modl$ be two classes of left $R$\+modules.
 Denote by ${}^{\perp_{\ge1}}\.\sE\subset R\modl$ the class of all
left $R$\+modules $F$ such that $\Ext_R^i(F,E)=0$ for all $E\in\sE$
and $i\ge1$, and by $\sF^{\perp_{\ge1}}\subset R\modl$ the class of all
left $R$\+modules $C$ such that $\Ext_R^i(F,C)=0$ for all $F\in\sF$
and $i\ge1$.

 Let $\sE\subset R\modl$ be a fixed class of left $R$\+modules.
 Set $\sF={}^{\perp_{\ge1}}\.\sE$ and $\sC=\sF^{\perp_{\ge1}}$.
 Our aim is to describe (to the extent possible, which may depend on
the situation at hand) the class $\sC$ in terms of the class~$\sE$.

 The following definition of a transfinitely iterated extension
is dual to that from Section~\ref{very-flat-definition-introd}.
 Let $C$ be a left $R$\+module and $\delta$~be an ordinal.
 Suppose that for every ordinal $\alpha\le\delta$ we are given
a left $R$\+module $C_\alpha$ and for every pair of ordinals
$\alpha<\beta\le\delta$ we are given an $R$\+module morphism
$C_\beta\rarrow C_\alpha$ such that the following conditions
are satisfied:
\begin{itemize}
\item $C_0=0$ and $C_\delta=C$;
\item the triangle diagrams $C_\gamma\rarrow C_\beta\rarrow C_\alpha$
are commutative for all triples of ordinals
$\alpha<\beta<\gamma\le\delta$;
\item the induced morphism into the projective limit
$C_\beta\rarrow \varprojlim_{\alpha<\beta}C_\alpha$ is an isomorphism
for all limit ordinals $\beta\le\delta$;
\item the morphism $C_{\alpha+1}\rarrow C_\alpha$ is surjective for
all ordinals $\alpha<\delta$.
\end{itemize}
 Let $D_\alpha$ denote the kernel of the morphism $C_{\alpha+1}\rarrow
C_\alpha$.
 Then the $R$\+module $C$ is said to be a \emph{transfinitely
iterated extension} (\emph{in the sense of the projective limit}) of
the $R$\+modules $D_\alpha$, where $0\le\alpha<\delta$.

\begin{lem} \label{dual-eklof-lemma}
 Let $F$ and $C$ be left $R$\+modules such that $C$ is a transfinitely
iterated extension (in the sense of the projective limit) of left
$R$\+modules~$D_\alpha$.
 Assume that\/ $\Ext_R^1(F,D_\alpha)=0$ for all~$\alpha$.
 Then\/ $\Ext^1_R(F,C)=0$.
\end{lem}

\begin{proof}
 This is the dual version of the Eklof Lemma~\cite[Lemma~1]{ET}.
 See~\cite[Proposition~18]{ET} (or, for a generalization to arbitrary
abelian categories, \cite[Lemma~4.5]{PR}).
\end{proof}

\begin{lem} \label{positive-ext-right-orthogonal-closedness}
 For any class of left $R$\+modules\/ $\sF\subset R\modl$, the class
of left $R$\+modules\/ $\sC=\sF^{\perp_{\ge1}}$ is closed under
the passages to direct summands, extensions, cokernels of injective
morphisms, infinite products, and transfinitely iterated extensions
in the sense of the projective limit.
\end{lem}

\begin{proof}
 Notice that, for any class of left $R$\+modules $\sG\subset R\modl$,
the class $\sG^{\perp_1}\subset R\modl$ consisting of all the left
$R$\+modules $C$ such that $\Ext^1_R(G,C)=0$ for all $G\in\sG$
is closed under all the mentioned operations with the possible
exception of the cokernels of injective morphisms.
 This follows from Lemma~\ref{dual-eklof-lemma} (notice also that
extensions and infinite products are particular cases of
the transfinitely iterated extensions in the sense of the projective
limit).
 Now let $\sG$ be the closure of $\sF$ with respect to the operation of
the passage to a syzygy module.
 Then $\sF^{\perp_{\ge1}}=\sG^{\perp_1}$.
 Checking that $\sF^{\perp_{\ge1}}$ is closed under the cokernels of
injective morphisms is easy.
\end{proof}

 The following definition plays a key role.

\begin{defn} \label{simply-right-obtainable-def}
 The class of all left $R$\+modules \emph{simply right obtainable} from
a given class $\sE\subset R\modl$ is defined as the (obviously, unique)
minimal class of left $R$\+modules containing $\sE$ and closed under
the operations of the passage to a direct summand, an extension,
the cokernel of an injective morphism, an infinite product, or
a transfinitely iterated extension in the sense of the projective limit.
\end{defn}

\begin{lem} \label{simply-right-obtainable-orthogonal-lem}
 For any class of left $R$\+modules\/ $\sE\subset R\modl$, all the left
$R$\+modules simply right obtainable from\/ $\sE$ belong to the class\/
$\sC=({}^{\perp_{\ge1}}\.\sE)^{\perp_{\ge1}}\subset R\modl$.
\end{lem}

\begin{proof}
 Follows from Lemma~\ref{positive-ext-right-orthogonal-closedness}.
\end{proof}

 Now let us formulate the dual definition.
 For any class of left $R$\+modules $\sM$, we denote by
${}^{\sT_{\ge1}}\sM\subset\modr R$ the class of all right $R$\+modules
$N$ such that $\Tor_i^R(N,M)=0$ for all $M\in\sM$ and $i\ge1$.
 Similarly, for any class of right $R$\+modules $\sN$, we denote by
$\sN^{\sT_{\ge1}}\subset R\modl$ the class of all left $R$\+modules
$M$ such that $\Tor_i^R(N,M)=0$ for all $N\in\sN$ and $i\ge1$.

\begin{defn} \label{simply-left-obtainable-def}
 The class of all left $R$\+modules \emph{simply left obtainable} from
a given class $\sM\subset R\modl$ is defined as the (obviously, unique)
minimal class of left $R$\+modules containing $\sM$ and closed under
the operations of the passage to a direct summand, an extension,
the kernel of a surjective morphism, an infinite direct sum, or
a transfinitely iterated extension in the sense of
the inductive limit.
\end{defn}

\begin{lem}
\textup{(a)} For any class of left $R$\+modules\/ $\sM\subset R\modl$,
all the left $R$\+modules simply left obtainable from\/ $\sM$ belong
to the class\/ ${}^{\perp_{\ge1}}(\sM^{\perp_{\ge1}})\subset R\modl$. \par
\textup{(b)} For any class of left $R$\+modules\/ $\sM\subset R\modl$,
all the left $R$\+modules simply left obtainable from\/ $\sM$ belong
to the class $({}^{\sT_{\ge1}}\.\sM)^{\sT_{\ge1}}\subset R\modl$.
\end{lem}

\begin{proof}
 Part~(a) is provable in the same way as
Lemma~\ref{simply-right-obtainable-orthogonal-lem}
(using the classical Eklof Lemma~\cite[Lemma~1]{ET} in place of its
dual version which we formulated as Lemma~\ref{dual-eklof-lemma}).
 Part~(b) can be deduced from part~(a) using the functor
$\Hom_\boZ({-},\boQ/\boZ)$, or proved directly using the fact that
the functor $\Tor$ preserves filtered inductive limits.
\end{proof}

\begin{rem} \label{simple-obtainability-limitations-of-use-remark}
 In the context of the arguments in this paper (see the discussion
in Section~\ref{obtaining-outline}), we will not nearly use
the full strength of Definitions~\ref{simply-right-obtainable-def}
and~\ref{simply-left-obtainable-def}.
 In particular, the passage to direct summands will not be used, and
we will only use countable products/sums, and the transfinitely iterated
extensions indexed by the ordinal $\omega=\boZ_{\ge0}$ of nonnegative
integers only.
 We will call the latter ``infinitely iterated extensions''.
\end{rem}

\Section{Toy Examples of a Main Lemma}  \label{toy-secn}

 Part~(a) of the following lemma will be used in the proofs of all
the main lemmas.

\begin{lem} \label{change-of-scalars-vanishing-tor-lemma}
 Let $R\rarrow R'$ be a morphism of commutative rings and $F$ be
an $R$\+module.
 Assume that $\Tor^R_i(R',F)=0$ for all $i\ge1$.
 Then \par
\textup{(a)} for any $R'$\+module $C$ and all $i\ge0$ there is a natural
isomorphism\/ $\Ext^i_R(F,C)\simeq\Ext^i_{R'}(R'\ot_RF,\>C)$; \par
\textup{(b)} for any $R'$\+module $N$ and all $i\ge0$ there is a natural
isomorphism\/ $\Tor_i^R(N,F)\simeq\Tor_i^{R'}(N,\>R'\ot_RF)$.
\end{lem}

\begin{proof} 
 It suffices to notice that, in the assumptions of the lemma, for
any projective resolution $P_\bu\rarrow F$ of the $R$\+module $F$,
the complex $R'\ot_RP_\bu$ is a projective resolution of
the $R'$\+module $R'\ot_RF$.
\end{proof}

 Part~(b) of the following generalization of
Lemma~\ref{change-of-scalars-vanishing-tor-lemma} will be used in
the proof of Lemma~\ref{flat-lemma} in this section.

\begin{lem} \label{change-of-scalars-partly-vanishing-tor}
 Let $R\rarrow R'$ be a morphism of commutative rings, $F$ be
an $R$\+module, and $n\ge1$ be an integer.
 Assume that $\Tor^R_i(R',F)=0$ for all\/ $1\le i\le n$.
 Then \par
\textup{(a)} for any $R'$\+module $C$ and all\/ $0\le i\le n$ there is
a natural isomorphism\/ $\Ext^i_R(F,C)\simeq\Ext^i_{R'}(R'\ot_RF,\>C)$;
\par
\textup{(b)} for any $R'$\+module $N$ and all\/ $0\le i\le n$ there is
a natural isomorphism\/ $\Tor_i^R(N,F)\simeq\Tor_i^{R'}(N,\>R'\ot_RF)$.
\end{lem}

\begin{proof}
 Similarly to the previous proof, let $P_{n+1}\rarrow P_n\rarrow\dotsb
\rarrow P_1\rarrow P_0$ be an initial fragment of a projective
resolution of the $R$\+module~$F$.
 Then $R'\ot_R P_{n+1}\rarrow R'\ot_R P_n\rarrow\dotsb\rarrow
R'\ot_R P_1\rarrow R'\ot_R P_0$ is an initial fragment of a projective
resolution of the $R'$\+module $R'\ot_RF$.
\end{proof}

 We start with proving Lemma~\ref{flat-lemma} before proceeding to
prove the slightly more complicated Toy Main Lemma~\ref{toy-main-lemma}.

 Let $R$ be a commutative ring and $r\in R$ be an element.
 An $R$\+module $M$ is said to be \emph{$r$\+torsion} if for every
$x\in M$ there exists $n\ge1$ such that $r^nx=0$ in~$M$.
 We denote the full subcategory of $r$\+torsion submodules
by $R\modl_{r\tors}\subset R\modl$.

 For any $R$\+module $M$ and an element $t\in R$, denote by ${}_tM
\subset M$ the submodule formed by all the elements $x\in M$
such that $tx=0$ in~$M$.
 Furthermore, denote by $\Gamma_r(M)\subset M$ the maximal $r$\+torsion
submodule in $M$; so $\Gamma_r(M)=\bigcup_{n\ge1}\.{}_{r^n}M$.
 The functor $\Gamma_r\:R\modl\rarrow R\modl_{r\tors}$ is right adjoint
to the embedding functor $R\modl_{r\tors}\rarrow R\modl$.

\begin{lem} \label{torsion-modules-obtainable}
 Any $r$\+torsion $R$\+module can be obtained as an infinitely iterated
extension, in the sense of the inductive limit, of $R/rR$\+modules
(viewed as $R$\+modules via the restriction of scalars).
\end{lem}

\begin{proof}
 Let $M$ be an $r$\+torsion $R$\+module.
 Then we have $M=\bigcup_{n\ge0}\.{}_{r^n}M$, and the successive quotients
${}_{r^{n+1}}M/{}_{r^n}M$ are $R/rR$\+modules.
\end{proof}

\begin{proof}[Proof of Lemma~\ref{flat-lemma}]
 Let $N$ be an $R$\+module.
 We have to show that $\Tor_1^R(N,F)=0$.

 Denote by~$l_N$ the $R$\+module morphism $N\rarrow N[r^{-1}]$.
 There are two short exact sequences of $R$\+modules
\begin{gather}
 0\lrarrow\ker(l_N)\lrarrow N\lrarrow\im(l_N)\lrarrow0,
 \label{torsion-first-sequence} \\
 0\lrarrow\im(l_N)\lrarrow N[r^{-1}]\lrarrow\coker(l_N)\lrarrow0,
 \label{torsion-second-sequence}
\end{gather}
where $\ker(l_N)=\Gamma_r(N)$ and $\coker(l_N)=\coker(l_R)\ot_RN$.
 From the related long exact sequences of $\Tor_*^R({-},F)$
\begin{gather*}
 \dotsb\lrarrow\Tor_1^R(\ker(l_N),F)\lrarrow\Tor_1^R(N,F)\lrarrow
 \Tor_1^R(\im(l_N),F)\lrarrow\dotsb, \\
 \dotsb\lrarrow\Tor_2^R(\coker(l_N),F)\lrarrow\Tor_1^R(\im(l_N),F)
 \lrarrow\Tor_1^R(N[r^{-1}],F)\lrarrow\dotsb
\end{gather*}
we see that, in order to prove that $\Tor_1^R(N,F)=0$, it suffices to
check that $\Tor_1^R(\ker(l_N),F)=0$, \ $\Tor_1^R(N[r^{-1}],F)=0$, and
$\Tor_2^R(\coker(l_N),F)=0$.

 Now $N[r^{-1}]$ is an $R[r^{-1}]$\+module and the $R[r^{-1}]$\+module
$F[r^{-1}]$ is flat by assumption, so $\Tor_1^R(N[r^{-1}],F)=0$ by
Lemma~\ref{change-of-scalars-vanishing-tor-lemma}(b).
 Furthermore, $\ker(l_N)$ and $\coker(l_N)$ are $r$\+torsion
$R$\+modules.

 For any fixed $i\ge1$ and an $r$\+torsion $R$\+module $M$, proving
that $\Tor^R_i(M,F)=0$ reduces to showing that $\Tor^R_i(D_n,F)=0$
for some $R/rR$\+modules $D_n$, $n\ge0$ (by
Lemma~\ref{torsion-modules-obtainable}).
 Finally, in view of the assumptions of Lemma~\ref{flat-lemma}, for
any $R/rR$\+module $D$ and $i=1$ or~$2$ one has $\Tor^R_i(D,F)=0$
by Lemma~\ref{change-of-scalars-partly-vanishing-tor}(b).
\end{proof}

 Along the way, we have essentially proved the following result.

\begin{prop} \label{flat-obtainable-proposition}
 Let $R$ be a commutative ring and $r\in R$ be an element.
 Then all $R$\+modules are simply left obtainable from $R/rR$\+modules
and $R[r^{-1}]$\+modules.
\end{prop}

\begin{proof}
 All $r$\+torsion $R$\+modules are simply left obtainable from
$R/rR$\+modules by Lemma~\ref{torsion-modules-obtainable}.
 Furthermore, according to the exact
sequences~(\ref{torsion-first-sequence}\+-%
\ref{torsion-second-sequence}) from the above proof of
Lemma~\ref{flat-lemma}, every $R$\+module $N$ is an extension
of a torsion $R$\+module $\ker(l_N)$ and an $R$\+module $\im(l_N)$
obtainable as the kernel of a surjective morphism from
an $R[r^{-1}]$\+module $N[r^{-1}]$ onto an $r$\+torsion $R$\+module
$\coker(l_N)$.
\end{proof}

 The proof of Toy Main Lemma~\ref{toy-main-lemma} is based on
a dual version of Lemma~\ref{torsion-modules-obtainable}, which
we will now formulate and prove.

 We recall (cf.\ Sections~\ref{very-flat-definition-introd}
and~\ref{contramodules-outline}) that an $R$\+module $C$
is said to be an \emph{$r$\+contramodule} if
$\Hom_R(R[r^{-1}],C)=0=\Ext^1_R(R[r^{-1}],C)$.
 The full subcategory of $r$\+contramodule $R$\+modules
$R\modl_{r\ctra}$ is closed under the kernels, cokernels, extensions,
and infinite products in $R\modl$ \cite[Proposition~1.1]{GL},
\cite[Theorem~1.2(a)]{Pcta}.
 The embedding functor $R\modl_{r\ctra}\rarrow R\modl$ has a left
adjoint functor $\Delta_r\:R\modl\allowbreak\rarrow R\modl_{r\ctra}$,
which can be constructed as follows.

 For any complex of $R$\+modules $L^\bu$ and an $R$\+module $C$,
we will use the simplified notation $\Ext_R^i(L^\bu,C)$ for
the $R$\+modules $\Hom_{\sD(R\modl)}(L^\bu,C[i])$, \ $i\in\boZ$
of morphisms in the derived category $\sD(R\modl)$.
 Denote by $K^\bu=K^\bu(R;r)$ the two-term complex of $R$\+modules
$R\overset{l_R}\rarrow R[r^{-1}]$, where the term $R$ is placed in
the cohomological degree~$-1$ and the term $R[r^{-1}]$ is placed
in the cohomological degree~$0$.

 According to~\cite[Theorem~6.4 and Remark~6.5]{Pcta}, we have
a natural isomorphism $\Delta_r(C)=\Ext_R^1(K^\bu,C)$ for any
$R$\+module~$C$.

\begin{lem} \label{r-contramodule-obtainable} \hbadness=1900
 Any $r$\+contramodule $R$\+module is simply right obtainable from
$R/rR$\+modules (viewed as $R$\+modules via the restriction of scalars).
\end{lem}

\begin{proof}
 This result goes back to~\cite[proof of Theorem~9.5]{Pcta}, where
a similar assertion was essentially proved for a finitely generated
ideal (rather than just a single element) in a commutative ring.

 As one has $\Delta_r(C)=C$ for any $r$\+contramodule $R$\+module $C$,
it suffices to show that the $R$\+module $\Delta_r(A)$ is simply
right obtainable from $R/rR$\+modules for any $R$\+module~$A$.
 For any $R$\+module $A$, consider two projective systems of
$R$\+modules (indexed by the positive integers):
$$
 A/rA\llarrow A/r^2A\llarrow A/r^3A\llarrow\dotsb,
$$
where the maps $A/r^{n+1}A\rarrow A/r^nA$ are the natural surjections,
and
$$
 {}_rA\llarrow {}_{r^2}A\llarrow {}_{r^3}A\llarrow\dotsb,
$$
where the maps ${}_{r^{n+1}}A\rarrow {}_{r^n}A$ are provided by
the operator of multiplication with~$r$.

\begin{subl} \label{delta-lambda-r-sequence}
 Let $R$ be a commutative ring and $r\in R$ be an element.
 Then for any $R$\+module $A$ there is a natural short exact sequence
of $R$\+modules
$$
 0\lrarrow\varprojlim\nolimits_{n\ge1}^1\.{}_{r^n}A\lrarrow\Delta_r(A)
 \lrarrow\varprojlim\nolimits_{n\ge1}A/r^nA\lrarrow0.
$$
\end{subl}

\begin{proof}
 This is~\cite[Lemma~6.7]{Pcta}.
 Essentially, the reason is that the complex $K^\bu$ is the inductive
limit of the complexes $R\overset{r^n}\rarrow R$ over $n\ge1$, so
the complex $\boR\Hom_R(K^\bu,A)$ is the homotopy projective limit of
the complexes $A\overset{r^n}\rarrow A$.
\end{proof}

 We denote the $r$\+adic completion $\varprojlim_{n\ge1}A/r^nA$ of
an $R$\+module $A$ by~$\Lambda_r(A)$.
 Sublemma~\ref{delta-lambda-r-sequence} describes the kernel of
the natural surjective morphism $\Delta_r(A)\rarrow\Lambda_r(A)$.

\begin{subl} \label{derived-projlim-obtainable}
 Let $R$ be a commutative ring and $r\in R$ be an element.
 Let $D_1\larrow D_2\larrow D_3\larrow\dotsb$ be a projective system of
$R$\+modules such that $D_n$ is an $R/r^nR$\+module for every $n\ge1$.
 Then the $R$\+modules \textup{(a)}~$\varprojlim_nD_n$ and
\textup{(b)}~$\varprojlim_n^1 D_n$ are simply right obtainable
from $R/rR$\+modules.
\end{subl}

\begin{proof}
 First of all we notice that, for any $n\ge2$, any $R/r^nR$\+module is
simply right obtainable (as a finitely iterated extension) from
$R/rR$\+modules.

 Part~(a): denote by $D'_n\subset D_n$ the image of the projection map
$\varprojlim_m D_m\rarrow D_n$.
 Then we have $\varprojlim_n D_n=\varprojlim_n D'_n$, and the maps
$D'_{n+1}\rarrow D'_n$ are surjective.
 Hence the $R$\+module $\varprojlim_n D_n$ is a infinitely iterated
extension of the $R$\+modules $D'_1$ and $\ker(D'_{n+1}\to D'_n)$, \
$n\ge1$.
 The former is, obviously, an $R/rR$\+module, and the latter
are $(R/r^{n+1}R)$\+modules.

 Part~(b): by the definition of $\varprojlim_{n\ge1}^1 D_n$, we have
an exact sequence
$$
 0\lrarrow\varprojlim\nolimits_{n\ge1}D_n\lrarrow
 \prod\nolimits_{n\ge1} D_n\lrarrow\prod\nolimits_{n\ge1} D_n
 \lrarrow\varprojlim\nolimits_{n\ge1}^1 D_n\lrarrow0.
$$
 Hence the $R$\+module $\varprojlim_n^1D_n$ is obtainable from
the $R$\+modules $\varprojlim_n D_n$ and $\prod_nD_n$ by two
passages to the cokernel of an injective morphism.
\end{proof}

 The assertion of Lemma~\ref{r-contramodule-obtainable} follows from
Sublemmas~\ref{delta-lambda-r-sequence}
and~\ref{derived-projlim-obtainable}(a\+b).
\end{proof}

 We recall (cf.\ Section~\ref{very-flat-definition-introd}) that
an $R$\+module $C$ is said to be \emph{$r$\+contraadjusted} if
$\Ext^1_R(R[r^{-1}],C)=0$.
 The following result is a dual version of
Proposition~\ref{flat-obtainable-proposition}.

\begin{tmp} \label{toy-main-proposition}
 Let $R$ be a commutative ring and $r\in R$ be an element.
 Then an $R$\+module $C$ is $r$\+contraadjusted if and only if it is
simply right obtainable from $R/rR$\+modules and $R[r^{-1}]$\+modules.
\end{tmp}
 
\begin{proof}
 ``If'': denote by $\sE\subset R\modl$ the class of all $R/rR$\+modules
and $R[r^{-1}]$\+modules (viewed as $R$\+modules via the restriction of
scalars).
 All the $R$\+modules from $\sE$ are $r$\+contraadjusted, so one has
$R[r^{-1}]\in\sF={}^{\perp_{\ge1}}\.\sE$.
 By Lemma~\ref{simply-right-obtainable-orthogonal-lem}, it follows that
all the $R$\+modules simply right obtainable from $\sE$ belong
to $\sF^{\perp_{\ge1}}\subset\{R[r^{-1}]\}^{\perp_{\ge1}}$, i.~e., they
are $r$\+contraadjusted.

 ``Only if'': there is a distinguished triangle
$$
 R\lrarrow R[r^{-1}]\lrarrow K^\bu\lrarrow R[1]
$$
in the derived category $\sD(R\modl)$.
 Applying the functor $\Hom_{\sD(R\modl)}({-},C[*])$, we obtain
a $5$\+term exact sequence
\begin{setlength}{\multlinegap}{0pt}
\begin{multline} \label{r-matlis-sequence}
 0\lrarrow\Hom_R(R[r^{-1}]/R,C)\lrarrow\Hom_R(R[r^{-1}],C)\lrarrow C \\
 \lrarrow\Delta_r(C)\lrarrow\Ext^1_R(R[r^{-1}],C)\lrarrow0,
\end{multline}
where $R[r^{-1}]/R$ is a shorthand notation for the cokernel of
the map $l_R\:R\rarrow R[r^{-1}]$, the derived category Hom module
$\Ext^0_R(K^\bu,C)=\Hom_{\sD(R\modl)}(K^\bu,C)$ is isomorphic to
$\Hom_R(R[r^{-1}]/R,C)$, and $\Ext^1_R(K^\bu,C)=\Delta_r(C)$.

 When the $R$\+module $C$ is $r$\+contraadjusted, the rightmost term
vanishes, so we are reduced to a $4$\+term exact sequence
\begin{equation} \label{r-contraadjusted-sequence}
 0\lrarrow\Hom_R(R[r^{-1}]/R,C)\lrarrow\Hom_R(R[r^{-1}],C)\lrarrow C \\
 \lrarrow\Delta_r(C)\lrarrow0.
\end{equation}
 Now $\Hom_R(R[r^{-1}],C)$ is an $R[r^{-1}]$\+module, $\Delta_r(C)$
is an $r$\+contramodule $R$\+module (computed in
Sublemma~\ref{delta-lambda-r-sequence}), and $\Hom_R(R[r^{-1}]/R,C)$ is
also an $r$\+contramodule (by~\cite[Lemma~6.1(a)]{Pcta}).
 In fact, $M=R[r^{-1}]/R$ is an $r$\+torsion $R$\+module, so
$\Hom_R(M,C)$ is the projective limit of the sequence of
$R/r^nR$\+modules $\Hom_R({}_{r^n}M,C)$ (cf.\
Sublemma~\ref{derived-projlim-obtainable}(a)).

 To sum up, the $R$\+module $C$ is obtainable from two
$r$\+contramodule $R$\+modules and one $R[r^{-1}]$\+module using
one passage to the cokernel of an injective morphism and one
passage to an extension of two modules.
 In view of Lemma~\ref{r-contramodule-obtainable}, the assertion
of the toy main proposition follows.
\end{setlength}
\end{proof}

\begin{proof}[Proof of Toy Main Lemma~\ref{toy-main-lemma}]
 The ``only if'' assertion holds, because the functors $F\longmapsto
F[r^{-1}]$ and $F\longmapsto F/rF$ preserve transfinitely iterated
extensions, in the sense of the inductive limit, of flat $R$\+modules,
and both of them take the $R$\+module $R[r^{-1}]$ to a free module
(with $1$ or~$0$ generators) over the respective ring.
 The ``if'' assertion is the nontrivial part.

 To show that the $R$\+module $F$ is $r$\+very flat, we need to check
that $\Ext^1_R(F,C)=0$ for all $r$\+contraadjusted $R$\+modules~$C$.
 As in the proof of Toy Main Proposition~\ref{toy-main-proposition},
we denote by $\sE\subset R\modl$ the class of all $R/rR$\+modules
and $R[r^{-1}]$\+modules.

 By Lemma~\ref{change-of-scalars-vanishing-tor-lemma}(a), it follows
from the assumptions of flatness of the $R$\+module $F$ and
projectivity of the $R/rR$\+module $F/rF$ and the $R[r^{-1}]$\+module
$F[r^{-1}]$ that $\Ext^i_R(F,E)=0$ for all $E\in\sE$ and $i\ge1$.
 So $F\in{}^{\perp_{\ge1}}\.\sE$.
 By Toy Main Proposition~\ref{toy-main-proposition}, the $R$\+module
$C$ is simply right obtainable from~$\sE$.
 By Lemma~\ref{simply-right-obtainable-orthogonal-lem}, it follows
that $C\in({}^{\perp_{\ge1}}\.\sE)^{\perp_{\ge1}}\subset
\{F\}^{\perp_{\ge1}}$, hence $\Ext_R^1(F,C)=0$.
\end{proof}

\Section{Contramodule Approximation Sequences}
\label{contramodule-approx-secn}

 The proofs of Main Lemmas~\ref{noetherian-main-lemma}
and~\ref{bounded-torsion-main-lemma} are based on the constructions
of certain short exact sequences of $r$\+contramodule $R$\+modules,
which are called the \emph{approximation sequences}.
 In Section~\ref{cotorsion-theories-abelian-categories-subsecn} below,
we present the related general background material about cotorsion
theories and approximation sequences in abelian categories.
 Then we proceed to construct the needed short exact sequences
in Sections~\ref{noetherian-flat-cotorsion-theory-subsecn},
\ref{bounded-torsion-veryflat-cotorsion-theory-subsecn},
and~\ref{veryflat-cotorsion-theory-quotseparated-subsecn}.
 The argument deducing the main lemmas from the existence of
approximation sequences is based on the discussion of
\emph{separated $r$\+contramodule $R$\+modules} in
Section~\ref{separated-contramodules-subsecn}; and the important
category of \emph{quotseparated $r$\+contramodule
$R$\+modules} is introduced in Section~\ref{quotseparated-subsecn}.

\subsection{Cotorsion theories in abelian categories}
\label{cotorsion-theories-abelian-categories-subsecn}
 Let $\sA$ be an abelian category and $\sF$, $\sC\subset\sA$ be two
classes of objects in~$\sA$.
 We will denote by $\sF^{\perp_1}\subset\sA$ the class of all objects
$C\in\sA$ such that $\Ext_\sA^1(F,C)=0$ for all $F\in\sF$, and by
${}^{\perp_1}\sC\subset\sA$ the class of all objects $F\in\sA$ such that
$\Ext_\sA^1(F,C)=0$ for all $C\in\sC$.

 A pair of classes of objects $(\sF,\sC)$ in an abelian category $\sA$
is called a \emph{cotorsion theory} (or a~\emph{cotorsion pair}) if
$\sC=\sF^{\perp_1}$ and $\sF={}^{\perp_1}\sC$ (see, e.~g.,
\cite[Section~5]{St} or~\cite[Section~8]{Pcta}).
 A cotorsion theory $(\sF,\sC)$ is called \emph{complete} if for every
object $X\in\sA$ there exist short exact sequences in the category~$\sA$
\begin{gather}
 0\lrarrow C'\lrarrow F\lrarrow X\lrarrow 0
 \label{special-precover-sequence} \\
 0\lrarrow X\lrarrow C\lrarrow F'\lrarrow 0
 \label{special-preenvelope-sequence}
\end{gather}
with some objects $F$, $F'\in\sF$ and $C$, $C'\in\sC$.
 The short exact sequences~(\ref{special-precover-sequence}\+-%
\ref{special-preenvelope-sequence}) are called \emph{approximation
sequences} for an object $X\in\sA$ with respect to the cotorsion
theory~$(\sF,\sC)$.
 The morphism $F\rarrow X$ in a short exact
sequence~\eqref{special-precover-sequence} is called a \emph{special\/
$\sF$\+precover} of an object $X$, while the morphism $X\rarrow C$
in a short exact sequence~\eqref{special-preenvelope-sequence} is called
a \emph{special\/ $\sC$\+preenvelope} of~$X$.

 Conversely, let $\sF$ and $\sC\subset\sA$ be two classes of objects
such that $\Ext^1_\sA(F,C)=0$ for all $F\in\sF$ and $C\in\sC$, and
short exact sequences~(\ref{special-precover-sequence}\+-%
\ref{special-preenvelope-sequence}) exist for all objects $X\in\sA$.
 Assume that the classes of objects $\sF$ and $\sC$ are closed with
respect to the passages to direct summands in~$\sA$.
 Then the pair of classes of objects $(\sF,\sC)$ is a (complete)
cotorsion theory in~$\sA$.

 A complete cotorsion theory $(\sF,\sC)$ in $\sA$ is said to be
\emph{hereditary} if one of the following equivalent conditions
holds~\cite[Lemma~6.17]{St}:
\begin{enumerate}
\renewcommand{\theenumi}{\roman{enumi}}
\item $\Ext_\sA^i(F,C)=0$ for all $F\in\sF$, \ $C\in\sC$, and
$i\ge1$;
\item $\Ext_\sA^2(F,C)=0$ for all $F\in\sF$ and $C\in\sC$;
\item the full subcategory $\sF\subset\sA$ is closed under
the passages to the kernels of epimorphisms;
\item the full subcategory $\sC\subset\sA$ is closed under
the passages to the cokernels of monomorphisms.
\end{enumerate}
 So, in a hereditary complete cotorsion theory, one has
$\sC=\sF^{\perp_{\ge1}}$ and $\sF={}^{\perp_{\ge1}}\sC$ (in the notation
of Section~\ref{obtainable-I-secn}).

 Let $\sS\subset\sA$ be a class of objects.
 Clearly, the pair of classes $\sF=\sS^{\perp_1}$ and
$\sC={}^{\perp_1}\sF$ is a cotorsion theory in~$\sA$.
 One says that the cotorsion theory $(\sF,\sC)$ is \emph{generated}
by the class of objects~$\sS\subset\sA$.

 Any complete cotorsion theory generated by a class of objects of
projective dimension not exceeding~$1$ in an abelian category~$\sA$
is hereditary (as the above condition~(iv) is clearly satisfied).
 Moreover, in such a cotorsion theory the class $\sC$ is closed
under the passages to arbitrary quotient objects, while the class
$\sF$ consists of objects of projective dimension not exceeding~$1$.

 A classical result of the paper~\cite[Theorem~10]{ET} claims that,
in the abelian category $\sA=R\modl$ of modules over an associative
ring $R$, any cotorsion theory generated by a \emph{set} of objects
is complete.
 Among many applications of this result, it allows to prove completeness
of the following cotorsion theories in $R\modl$, which are of relevance
to us.

 A left module $C$ over an associative ring $R$ is said to be
\emph{cotorsion}~\cite{En} (cf.\ Section~\ref{cosheaves-introd})
if $\Ext^1_R(F,C)=0$ for all flat left $R$\+modules~$C$.
 The pair of full subcategories (flat left $R$\+modules, cotorsion
left $R$\+modules) forms a complete cotorsion theory in $R\modl$
\cite[Proposition~2]{BBE}, which is called the \emph{flat} cotorsion
theory.
 The flat cotorsion theory in $R\modl$ is also hereditary
(as the condition~(iii) is satisfied).

 The definitions of very flat and contraadjusted modules over
a commutative ring $R$ were given in
Section~\ref{very-flat-definition-introd}.
 The pair of full subcategories (very flat $R$\+modules, contraadjusted
$R$\+modules) forms a complete cotorsion theory in $R\modl$
\cite[Section~1.1]{Pcosh}, which is called the \emph{very flat}
cotorsion theory.
 The very flat cotorsion theory is hereditary, because it is generated
by a set of modules of projective dimension not exceeding~$1$
(namely, the $R$\+modules $R[r^{-1}]$, \,$r\in R$).

 In this section, we are interested in certain (hereditary complete)
cotorsion theories in the abelian category $\sA=R\modl_{r\ctra}$ of
$r$\+contramodule $R$\+modules, where $R$ is a commutative ring
and $r\in R$ is an element.
 One thing one can say about the abelian category $R\modl_{r\ctra}$ is
that it is \emph{locally presentable}, or more precisely,
locally $\aleph_1$\+presentable~\cite[Examples~4.1]{PR}.
 The following result from the paper~\cite{PR} provides a generalization
of~\cite[Theorem~10]{ET} to locally presentable abelian categories.

\begin{thm}
 Let\/ $\sA$ be a locally presentable abelian category, $\sS\subset\sA$
be a set of objects, and\/ $(\sF,\sC)$ be the cotorsion theory
generated by\/ $\sS$ in\/~$\sA$.
 Assume that every object of\/~$\sA$ is a quotient object of an object
from\/~$\sF$ and a subobject of an object from\/~$\sC$.
 Then the cotorsion theory $(\sF,\sC)$ in the category\/~$\sA$ is
complete.
\end{thm}

\begin{proof}
 This is~\cite[Corollary~3.6]{PR} (see~\cite[Theorem~4.8]{PR} for
further details).
\end{proof}

\subsection{Flat cotorsion theory in $r$-contramodule $R$-modules
for a Noetherian commutative ring~$R$}
\label{noetherian-flat-cotorsion-theory-subsecn}
 In this section we will consider $r$\+contramodule $R$\+modules that
are flat, cotorsion, or contraadjusted \emph{as $R$\+modules}.
 To emphasize this aspect, we will call such $r$\+contramodule
$R$\+modules \emph{$R$\+flat}, \emph{$R$\+cotorsion}, or
\emph{$R$\+contraadjusted}.

 The aim of this section is to present a sketch of proof of
the following theorem (which is a particular case of the results
of~\cite[Section~C.2]{Pcosh}).

\begin{thm} \label{noetherian-r-contra-flat-cotorsion-theory-thm}
 Let $R$ be a Noetherian ring and $r\in R$ be an element.
 Then the pair of full subcategories $(R$\+flat $r$\+contramodule
$R$\+modules, $R$\+cotorsion $r$\+contramodule $R$\+modules) is
a hereditary complete cotorsion theory in the abelian category
$R\modl_{r\ctra}$.  \hbadness=1275
\end{thm}

 The idea of the proof is to produce approximation sequences for
an $r$\+contramodule $R$\+module from its approximation sequences
in the flat cotorsion theory in the category of arbitrary
$R$\+modules.

 Let $R$ be a commutative ring and $r\in R$ be an element.
 For any $R$\+module $M$, we denote by $b_M$ the $R$\+module morphism
$\Hom_R(R[r^{-1}],M)\rarrow M$ obtained by applying the functor
$\Hom_R({-},M)$ to the $R$\+module morphism $R\rarrow R[r^{-1}]$.
 According to the exact sequence~\eqref{r-matlis-sequence} from
the proof of Toy Main Proposition~\ref{toy-main-proposition},
one has $\ker(b_M)=\Hom_R(R[r^{-1}]/R,M)$.

 Let $h_r(M)\subset M$ denote the image of the morphism~$b_M$.
 Equivalently, $h_r(M)$ is the maximal $r$\+divisible $R$\+submodule
in~$M$.
 According to the exact sequence~\eqref{r-contraadjusted-sequence},
for any $r$\+contraadjusted $R$\+module $K$ the adjunction morphism
$K\rarrow\Delta_r(K)$ is surjective with the kernel~$h_r(K)$; so
we have $\Delta_r(K)=K/h_r(K)$.

 We recall that the class of $r$\+contraadjusted $R$\+modules is
closed under the passages to arbitrary quotient objects, while
the class of cotorsion $R$\+modules is closed under the cokernels
of injective morphisms.

\begin{lem} \label{divisible-part-contraadjusted-cotorsion}
 Let $R$ be a commutative ring and $r$, $s\in R$ be two elements.
 Then \par
\textup{(a)} for any $r$\+contraadjusted $R$\+module $K$,
the $R$\+module $h_r(K)$ is $r$\+contraadjusted; \par
\textup{(b)} for any $(rs)$\+contraadjusted $R$\+module $K$,
the $R$\+module $h_r(K)$ is $s$\+contraad\-justed; \par
\textup{(c)} for any cotorsion $R$\+module $K$, the $R$\+module\/
$\Hom_R(R[r^{-1}],K)$ is cotorsion.
\end{lem}

\begin{proof}
 Part~(b): the $R$\+module $\Hom_R(R[r^{-1}],K)$ is $s$\+contraadjusted
for any $(rs)$\+con\-traadjusted $R$\+module $K$ by
Lemma~\ref{hom-s-contraadjusted} below
(cf.~\cite[proof of Lemma~1.2.1]{Pcosh}); hence its quotient $R$\+module
$h_r(K)$ is also $s$\+contraadjusted.
 Part~(a) is part~(b) for $r=\nobreak s$.
 Part~(c): the $R$\+module $\Hom_R(F,K)$ is cotorsion for any flat
$R$\+module $F$ and cotorsion $R$\+module $K$
\cite[Lemma~1.3.2(a)]{Pcosh}.
\end{proof}

 We refer to Section~\ref{bounded-torsion-outline} for the definition
of what it means that an $R$\+module has \emph{bounded $r$\+torsion}.

\begin{lem} \label{noetherian-flat-bounded-torsion}
 Let $R$ be a Noetherian commutative ring and $r\in R$ be an element.
 Then any flat $R$\+module has bounded $r$\+torsion.
\end{lem}

\begin{proof}
 This is a particular case of~\cite[Lemma~C.2.1]{Pcosh}, which can be
also obtained as a particular case of
Lemma~\ref{flat-modules-bounded-torsion} below.
\end{proof}

 The following lemma plays a key role.

\begin{lem} \label{flat-r-contraadjusted-remains-flat}
 Let $R$ be a Noetherian commutative ring and $r\in R$ be an element.
 Let $F$ be a flat $r$\+contraadjusted $R$\+module $F$.
 Then the $R$\+module $\Delta_r(F)=F/h_r(F)$ is also flat.
\end{lem}

\begin{proof}
 There are many ways to prove this basic result.
 In particular, one can deduce it from~\cite[Lemmas~1.6.2
and~C.2.1]{Pcosh}, as explained in~\cite[proof of
Corollary~C.2.3(a)]{Pcosh}.
 Alternatively, one can observe that the $R$\+module $\Delta_r(F)$ is
flat for any flat $R$\+module $F$, which is a particular case
of~\cite[Corollary~10.4(b)]{Pcta}.
\end{proof}

 Now we can produce the promised approximation sequences in
$R\modl_{r\ctra}$.

\begin{lem} \label{noetherian-r-contra-special-flat-precover}
 Let $R$ be a Noetherian commutative ring and $r\in R$ be an element.
 Let $C$ be an $r$\+contramodule $R$\+module, and let\/
$0\rarrow K\rarrow F\rarrow C\rarrow0$ be a short exact sequence of
$R$\+modules with a flat $R$\+module $F$ and an $r$\+contraadjusted
$R$\+module~$K$.
 Then\/ $0\rarrow K/h_r(K)\rarrow F/h_r(F)\rarrow C\rarrow0$ is
a short exact sequence of $r$\+contramodule $R$\+modules with
a flat $R$\+module $F/h_r(F)$.
 If $K$ is a cotorsion $R$\+module, then so are the $R$\+modules
$h_r(K)$ and $K/h_r(K)$.
\end{lem}

\begin{proof}
 This is~\cite[Lemma~C.2.6]{Pcosh}.
 By Lemma~\ref{noetherian-flat-bounded-torsion}, we have
$\Hom_R(R[r^{-1}]/R,F)=0$, and consequently $\Hom_R(R[r^{-1}]/R,K)=0$.
 Therefore, $h_r(K)=\Hom_R(R[r^{-1}],K)$ and $h_r(F)=\Hom_R(R[r^{-1}],F)$.
 Since $C$ is an $r$\+contramodule $R$\+module, it follows that
the morphism $K\rarrow F$ induces an isomorphism $h_r(K)\simeq h_r(F)$.
 Thus the short sequence $0\rarrow K/h_r(K)\rarrow F/h_r(F)\rarrow C
\rarrow0$ is exact.

 The $R$\+module $F$ is $r$\+contraadjusted as an extension of two
$r$\+contraadjusted $R$\+modules $C$ and~$K$.
 The $R$\+modules $K$ and $F$ being $r$\+contraadjusted, it follows
that the $R$\+modules $K/h_r(K)$ and $F/h_r(F)$ are $r$\+contramodules.
 The $R$\+module $F/h_r(F)$ is flat by
Lemma~\ref{flat-r-contraadjusted-remains-flat}.
 If the $R$\+module $K$ is cotorsion, then the $R$\+module $h_r(K)$ is
cotorsion by Lemma~\ref{divisible-part-contraadjusted-cotorsion}(c),
hence the $R$\+module $K/h_r(K)$ is cotorsion as the cokernel of
an injective morphism of cotorsion $R$\+modules.
\end{proof}

\begin{lem} \label{noetherian-r-contra-special-cotorsion-preenvelope}
 Let $R$ be a Noetherian commutative ring and $r\in R$ be an element.
 Let $C$ be an $r$\+contramodule $R$\+module, and let\/
$0\rarrow C\rarrow K\rarrow F\rarrow0$ be a short exact sequence of
$R$\+modules with a flat $R$\+module $F$ and an $r$\+contraadjusted
$R$\+module~$K$.
 Then\/ $0\rarrow C\rarrow K/h_r(K)\rarrow F/h_r(F)\rarrow0$ is
a short exact sequence of $r$\+contramodule $R$\+modules with
a flat $R$\+module $F/h_r(F)$.
 If $K$ is a cotorsion $R$\+module, then so are the $R$\+modules
$h_r(K)$ and $K/h_r(K)$.
\end{lem}

\begin{proof}
 This is~\cite[Lemma~C.2.4]{Pcosh}.
 By Lemma~\ref{noetherian-flat-bounded-torsion}, we have
$\Hom_R(R[r^{-1}]/R,F)=0$.
 Since $\Hom_R(R[r^{-1}]/R,C)\subset\Hom_R(R[r^{-1}],C)=0$,
it follows that $\Hom_R(R[r^{-1}]/R,\allowbreak K)=0$.
 The proof of exactness of the short sequence
$0\rarrow C\rarrow K/h_r(K)\rarrow F/h_r(F)\rarrow0$ finishes
as in the proof of
Lemma~\ref{noetherian-r-contra-special-flat-precover}.

 The $R$\+module $F$ is $r$\+contraadjusted as a quotient module
of an $r$\+contraadjusted $R$\+module~$K$.
 The rest of the argument is the same as in
Lemma~\ref{noetherian-r-contra-special-flat-precover}.
\end{proof}

\begin{proof}[Proof of
Theorem~\ref{noetherian-r-contra-flat-cotorsion-theory-thm}]
 One has $\Ext_{R\modl_{r\ctra}}^1(F,C)=0$ for any $R$\+flat
$r$\+contramodule $R$\+module $F$ and any $R$\+cotorsion
$r$\+contramodule $R$\+module $C$, because $\Ext_R^1(F,C)=0$ and
$R\modl_{r\ctra}\subset R\modl$ is a full subcategory closed under
kernels, cokernels, and extensions.
 The approximation sequences~(\ref{special-precover-sequence}\+-%
\ref{special-preenvelope-sequence}) in $R\modl_{r\ctra}$ can be
produced from the similar sequences in $R\modl$ using
Lemmas~\ref{noetherian-r-contra-special-flat-precover}\+-%
\ref{noetherian-r-contra-special-cotorsion-preenvelope}.
 It remains to observe that the classes of $R$\+flat $r$\+contramodule
$R$\+modules and $R$\+cotorsion $r$\+contramodule $R$\+modules are
closed under direct summands.
 The flat cotorsion theory in $R\modl_{r\ctra}$ is hereditary,
because the conditions~(iii) and/or~(iv) from
Section~\ref{cotorsion-theories-abelian-categories-subsecn} are
satisfied (as the flat cotorsion theory in $R\modl$ is hereditary).
\end{proof}

\subsection{Very flat cotorsion theory in the bounded torsion case}
\label{bounded-torsion-veryflat-cotorsion-theory-subsecn}
 The notion of a left contramodule over a complete, separated
topological associative ring $\fR$ with a base of neighborhoods of
zero formed by open right ideals was introduced
in~\cite[Remark~A.3]{Psemi}, \cite[Section~1.2]{Pweak}, and
studied in~\cite[Appendix~D]{Pcosh}, \cite[Sections~5\+-7]{PR}.
 In the context of this paper, we set $\fR=\varprojlim_{n\ge1}R/r^nR$
to be the $r$\+adic completion of a commutative ring $R$, endowed with
the projective limit (\,$=$~$r$\+adic) topology.
 We denote the abelian category of $\fR$\+contramodules by
$\fR\contra$.

 In this section we are interested in commutative rings $R$ with
an element $r\in R$ such that the $r$\+torsion in $R$ is
bounded (cf.\ Section~\ref{bounded-torsion-outline}).

\begin{lem} \label{flat-modules-bounded-torsion}
 Let $R$ be a commutative ring and $r\in R$ be an element such that
the $r$\+torsion in $R$ is bounded.
 Then the $R$\+torsion in any flat $R$\+module is bounded, too.
\end{lem}

\begin{proof}
 Let $m\ge1$ be an integer such that one has $r^mx=0$ for any
$r$\+torsion element $x\in R$.
 Using the Govorov--Lazard description of flat $R$\+modules
as filtered inductive limits of projective ones, one easily shows
that for any flat $R$\+module $F$ and an $r$\+torsion element
$y\in F$ one also has $r^my=0$.
\end{proof}

 The Noetherian particular case of the following result goes back
to~\cite[Theorem~B.1.1]{Pweak}.
 For a generalization to arbitrary commutative rings $R$ with
an element~$r$, see
Theorem~\ref{quotseparated-contramodule-category-equivalence} below.

\begin{thm} \label{bounded-torsion-contramodule-category-equivalence}
 Let $R$ be a commutative ring and $r\in R$ be an element such that
the $r$\+torsion in $R$ is bounded.
 Then the forgetful functor\/ $\fR\contra\rarrow R\modl$ induces
an equivalence of abelian categories\/ $\fR\contra\simeq
R\modl_{r\ctra}$.
\end{thm}

\begin{proof}
 This is~\cite[Example~2.2(5)]{Pper}.
 Essentially, both $\fR\contra$ and $R\modl_{r\ctra}$ are isomorphic to
the category of algebras/modules over the additive monad $X\longmapsto
\Delta_r(R[X])=\Lambda_r(R[X])$ on the category of sets, where $R[X]$
denotes the free $R$\+module generated by a set~$X$
(cf.\ Sublemma~\ref{delta-lambda-r-sequence}).
\end{proof}

 So, in the bounded torsion case, ``contramodules over the topological
ring~$\fR$'' is just another name for $r$\+contramodule $R$\+modules.
 The point is that this interpretation of $r$\+contramodule $R$\+modules
allows to apply the results of~\cite[Section~7]{PR} about cotorsion
theories in $\fR\contra$.

 An $\fR$\+contramodule $F$ is called \emph{flat} if
the $R/r^nR$\+module $F/r^nF$ is flat for every $n\ge1$
\cite[Section~D.1]{Pcosh}, \cite[Sections~5\+-6]{PR}.
 Unlike in the Noetherian case of
Section~\ref{noetherian-flat-cotorsion-theory-subsecn}, there is
\emph{no} claim that flat $\fR$\+contramodules are flat
$R$\+modules here.

 In fact, when the $r$\+torsion in $R$ is bounded,
an $\fR$\+contramodule or an $r$\+contramod\-ule $R$\+module is the same
thing as a cohomologically $I$\+adically complete $R$\+module in
the sense of~\cite{PSY,Yek}, where $I=(r)\subset R$ is the principal
ideal generated by the element $r\in R$ \cite[Section~0.9]{Pmgm}.
 Furthermore, in the same assumption, an $\fR$\+contramodule is flat
if and only if it is an $I$\+adically flat $R$\+module in
the sense of the paper~\cite{Yek2} (see~\cite[Theorems~4.3
and~6.9]{Yek2} and~\cite[Lemma~10.1]{Pcta}).
 The counterexample in~\cite[Theorem~7.2]{Yek2} shows that a flat
$\fR$\+contramodule does not need to be a flat $R$\+module.

\begin{thm} \label{generated-by-set-of-flat-contra-thm}
 The cotorsion theory generated by any set of flat\/
$\fR$\+contramodules is complete in\/ $\fR\contra$.
\end{thm}

\begin{proof}
 This is a particular case of~\cite[Corollary~7.11]{PR}.
\end{proof}

\begin{lem} \label{delta-modulo-r-exp-n}
 For any $R$\+module $A$ and any integer $n\ge1$, the adjunction
morphism $A\rarrow\Delta_r(A)$ induces an isomorphism
$A/r^nA\simeq\Delta_r(A)/r^n\Delta_r(A)$.
\end{lem}

\begin{proof}
 All $R/r^nR$\+modules are $r$\+contramodule $R$\+modules.
 Since the functor $\Delta_r$ is left adjoint to the embedding
$R\modl_{r\ctra}\rarrow R\modl$, for any $R/r^nR$\+module $D$ we have
\begin{multline*}
\Hom_{R/r^nR}(A/r^nA,\.D)=\Hom_R(A,D) \\ =\Hom_R(\Delta_r(A),D)=
\Hom_{R/r^nR}(\Delta_r(A)/r^n\Delta_r(A),\.D),
\end{multline*}
implying the desired isomorphism.
\end{proof}

 Consider the set of all the $r$\+contramodule $R$\+modules
$\Delta_r(R[s^{-1}])$, where $s\in R$.
 One has $\Delta_r(R[s^{-1}])/r^n\Delta_r(R[s^{-1}])=
R[s^{-1}]/r^nR[s^{-1}]=(R/r^nR)[s^{-1}]$ by
Lemma~\ref{delta-modulo-r-exp-n}, so $\Delta_r(R[s^{-1}])$
is a flat $\fR$\+contramodule. {\hbadness=1750\par}

 The cotorsion theory $(\sF,\sC)$ generated by the set of all
$r$\+contramodule $R$\+modules $\Delta_r(R[s^{-1}])$ is called
the \emph{very flat} cotorsion theory in $R\modl_{r\ctra}$.
 The objects from the class $\sF$ are called the \emph{very flat}
$r$\+contramodule $R$\+modules, and the objects of the class $\sC$
are called the \emph{contraadjusted} $r$\+contramodule $R$\+modules.

 Once again, there is \emph{no} claim that very flat $r$\+contramodule
$R$\+modules are very flat $R$\+modules (even for a Noetherian ring $R$;
cf.~\cite[Section~C.3]{Pcosh}).
 On the other hand, the following results hold.

\begin{prop} \label{very-flat-contramodules-are-flat}
 Let $R$ be a commutative ring and $r\in R$ be an element such that
the $r$\+torsion in $R$ is bounded.
 Then all very flat $r$\+contramodule $R$\+modules are flat\/
$\fR$\+contramodules.
\end{prop}

\begin{proof}
 This follows from the existence of the \emph{flat} cotorsion theory
in $\fR\contra$; see~\cite[Corollary~7.8]{PR}.
 Alternatively, one can apply~\cite[Theorem~4.8(a,d)]{PR} to show
that an $r$\+contramodule $R$\+module is very flat if and only if
it is a direct summand of a transfinitely iterated extension of
the objects $\Delta_r(R[s^{-1}])$ in the category $R\modl_{r\ctra}$
(in the sense of~\cite[Definition~4.3]{PR}).
 All such transfinitely iterated extensions are flat
$\fR$\+contramodules by~\cite[Lemma~5.6 and Corollary~7.1(b)]{PR}.
\end{proof}

\begin{thm} \label{contraadjusted-r-contramodules-described}
 Let $R$ be a commutative ring and $r\in R$ be an element such that
the $r$\+torsion in $R$ is bounded.
 Then an $r$\+contramodule $R$\+module is contraadjusted if and only
if it is contraadjusted as an $R$\+module.
\end{thm}

\begin{proof}
 Recall that the functor $\Delta_r$ can be computed as
$\Delta_r(A)=\Ext^1_R(K^\bu,A)$, where $K^\bu=K^\bu(R;r)$ denotes
the two-term complex $R\rarrow R[r^{-1}]$ (see Section~\ref{toy-secn}).
 Now let $0\rarrow C\rarrow B\rarrow A\rarrow 0$ be a short exact
sequence of $R$\+modules.
 Applying the functor $\Ext_R^*(K^\bu,{-})$, we obtain a long exact
sequence
\begin{equation} \label{derived-delta-sequence}
 \dotsb\lrarrow\Hom_R(R[r^{-1}]/R,A)\lrarrow\Delta_r(C)\lrarrow
 \Delta_r(B)\lrarrow\Delta_r(A)\lrarrow0.
\end{equation}

 Let $C$ be an $r$\+contramodule $R$\+module.
 Suppose that we have a short exact sequence of $R$\+modules
\begin{equation} \label{contraadjusted-over-R-extension-sequence}
 0\lrarrow C\lrarrow B\lrarrow R[s^{-1}]\lrarrow0.
\end{equation}
 Applying the functor $\Delta_r$, we get a short exact sequence
\begin{equation} \label{induced-R-mod-r-ctra-ext-sequence}
 0\lrarrow C\lrarrow \Delta_r(B)\lrarrow \Delta_r(R[s^{-1}])\lrarrow0,
\end{equation}
because $\Delta_r(C)=C$ and $\Hom_R(R[r^{-1}]/R,\.R[s^{-1}])=0$ by
Lemma~\ref{flat-modules-bounded-torsion}.
 There is a natural (adjunction) morphism from the short exact
sequence~\eqref{contraadjusted-over-R-extension-sequence} into
the short exact sequence~\eqref{induced-R-mod-r-ctra-ext-sequence},
hence it follows that
the sequence~\eqref{contraadjusted-over-R-extension-sequence} is
the pullback of the sequence~\eqref{induced-R-mod-r-ctra-ext-sequence}
with respect to the morphism $R[s^{-1}]\rarrow\Delta_r(R[s^{-1}])$.

 Conversely, given a short exact sequence of $r$\+contramodule
$R$\+modules
\begin{equation} \label{contraadjusted-r-contramod-ext-sequence}
 0\lrarrow C\lrarrow B'\lrarrow \Delta_r(R[s^{-1}])\lrarrow0,
\end{equation}
one can take the pullback with respect to the morphism
$R[s^{-1}]\rarrow\Delta_r(R[s^{-1}])$ in order to produce a short
exact sequence~\eqref{contraadjusted-over-R-extension-sequence}.
 Then the adjunction provides an isomorphism
from the exact sequence~\eqref{induced-R-mod-r-ctra-ext-sequence}
into the exact sequence~\eqref{contraadjusted-r-contramod-ext-sequence}.

 We have constructed an isomorphism of the Ext modules
$$
 \Ext_R^1(R[s^{-1}],C)\simeq\Ext_{R\modl_{r\ctra}}^1(\Delta_r(R[s^{-1}]),C)
$$
for any $r$\+contramodule $R$\+module $C$, implying the assertion
of the theorem.
\end{proof}

 Using the sequence~\eqref{derived-delta-sequence}, one can also show
that the objects $\Delta_r(R[s^{-1}])$ have projective dimension at
most~$1$ in $R\modl_{r\ctra}$.
 Indeed, the functor $\Delta_r\:R\modl\rarrow R\modl_{r\ctra}$ is left
adjoint to an exact functor, so it takes projectives to projectives.
 Applying it to a projective resolution $0\rarrow Q\rarrow P\rarrow
R[s^{-1}]\rarrow0$ of the $R$\+module $R[s^{-1}]$, one obtains
a projective resolution $0\rarrow\Delta_r(Q)\rarrow\Delta_r(P)\rarrow
\Delta_r(R[s^{-1}])\rarrow0$ of the $r$\+contramodule $R$\+module
$\Delta_r(R[s^{-1}])$.

\begin{rem} \label{two-very-flat-cotorsion-theories-remark}
 Alternatively, one could define a \emph{very flat\/
$\fR$\+contramodule} as an $\fR$\+contramodule $F$ such that
the $R/r^nR$\+module $F/r^nF$ is very flat for all $n\ge1$.
 This also leads to a complete cotorsion theory in $\fR\contra$
\cite[Example~7.12(3)]{PR}.
 In fact, such a definition of a very flat $r$\+contramodule $R$\+module
for a commutative ring $R$ with bounded $r$\+torsion is equivalent to
the above, so the two ``very flat cotorsion theories in
$R\modl_{r\ctra}$'' coincide.
 This can be shown by comparing
Theorem~\ref{contraadjusted-r-contramodules-described} with
the results of~\cite[Section~D.4]{Pcosh}, or more
specifically~\cite[Corollary~D.4.8]{Pcosh} (implying that the
classes of contraadjusted objects in the two ``very flat cotorsion
theories'' coincide).
 We will discuss the theory of~\cite[Section~D.4]{Pcosh} below in
Section~\ref{veryflat-cotorsion-theory-quotseparated-subsecn}.
\end{rem}

\subsection{Separated $r$-contramodule $R$-modules}
\label{separated-contramodules-subsecn}
 Let $R$ be a commutative ring and $r\in R$ be an element.
 An $R$\+module $C$ is said to be \emph{$r$\+complete}, if the natural
map
$$
 \lambda_{r,C}\:C\lrarrow \Lambda_r(C)=
 \varprojlim\nolimits_{n\ge1}C/r^nC
$$
is surjective, and \emph{$r$\+separated}, if the map~$\lambda_{r,C}$
is injective (cf.\ Section~\ref{contramodules-outline}).

 Clearly, an $R$\+module $C$ is $r$\+separated if and only if
the intersection $\bigcap_{n\ge1}r^nC\subset C$ vanishes.
 It follows that any $R$\+submodule of an $r$\+separated $R$\+module
is $r$\+separated.

 Any $r$\+contraadjusted $R$\+module is
$r$\+complete~\cite[Theorem~2.3(a)]{Pcta}.
 Any $r$\+sepa\-rated and $r$\+complete $R$\+module is
an $r$\+contramodule~\cite[Theorem~2.4(b)]{Pcta}.
 However, an $r$\+contramodule $R$\+module does not need to be
$r$\+separated~\cite[Example~4.33]{PSY}, 
\cite[Example~2.7(1)]{Pcta}.
 The kernel $\bigcap_{n\ge1}r^nC\subset C$ of the $r$\+completion
morphism~$\lambda_{r,C}$ for an $r$\+contramodule $R$\+module $C$
was computed in Sublemma~\ref{delta-lambda-r-sequence}.

 The following corollaries pick out the aspects of the results of
Sections~\ref{noetherian-flat-cotorsion-theory-subsecn}\+-%
\ref{bounded-torsion-veryflat-cotorsion-theory-subsecn} relevant
for the proofs of Main Lemmas~\ref{noetherian-main-lemma}
and~\ref{bounded-torsion-main-lemma} in
Section~\ref{noetherian-mlemma-secn} below.
 We recall that the terms ``$R$\+flat'' and ``$R$\+contraadjusted'' are
being used as shorthands for ``flat as an $R$\+module'' and
``contraadjusted as an $R$\+module''.

\begin{cor} \label{noetherian-cokernel-of-contraadjusted-separated}
 Let $R$ be a Noetherian commutative ring, $r\in R$ be an element,
and $C$ be an $R$\+contraadjusted $r$\+contramodule $R$\+module.
 Then the $R$\+module $C$ can be presented as the cokernel of
an injective morphism of $R$\+contraadjusted $r$\+separated
$r$\+complete $R$\+modules.
\end{cor}

\begin{proof}
 By Theorem~\ref{noetherian-r-contra-flat-cotorsion-theory-thm},
or more specifically by
Lemma~\ref{noetherian-r-contra-special-flat-precover},
there exists a short exact sequence of $r$\+contramodule $R$\+modules
$0\rarrow K\rarrow F\rarrow C\rarrow 0$, where the $R$\+module $F$
is flat and the $R$\+module $K$ is cotorsion.
 What is important for us is that the $R$\+module $K$ is
contraadjusted; since the $R$\+module $C$ is contraadjusted by
assumption, it follows that the $R$\+module $F$ is contraadjusted, too.

 Furthermore, any $R$\+flat $r$\+contramodule $R$\+module is
$r$\+separated~\cite[Corollary~10.3(b)]{Pcta}.
 The $R$\+module $F$ being $r$\+separated, it follows that its submodule
$K$ is $r$\+separated, too.
 Thus $K\rarrow F$ is an injective morphism of $R$\+contraadjusted
$r$\+separated $r$\+contramodules with the cokernel~$C$.
\end{proof}

\begin{cor} \label{bounded-torsion-cokernel-of-contraadjusted-separated}
 Let $R$ be a commutative ring and $r\in R$ be an element such that
the $r$\+torsion in $R$ is bounded.
 Then any $R$\+contraadjusted $r$\+contramodule $R$\+module can be
presented as the cokernel of an injective morphism of
$R$\+contraadjusted $r$\+separated $r$\+complete $R$\+modules.
\end{cor}

\begin{proof}
 By Theorem~\ref{generated-by-set-of-flat-contra-thm}, the very flat
cotorsion theory in $\fR\contra$ (as defined in
Section~\ref{bounded-torsion-veryflat-cotorsion-theory-subsecn})
is complete.
 By Theorem~\ref{bounded-torsion-contramodule-category-equivalence},
the categories $\fR\contra$ and $R\modl_{r\ctra}$ coincide.
 Hence for any $r$\+contramodule $R$\+module $C$ there exists
a short exact sequence $0\rarrow K\rarrow F\rarrow C\rarrow0$, where
$K$ is a contraadjusted $r$\+contramodule $R$\+module and
$F$ is a very flat $r$\+contramodule $R$\+module.

 By Theorem~\ref{contraadjusted-r-contramodules-described}, $K$ is
a contraadjusted $R$\+module.
 Assuming that $C$ is a contraadjusted $R$\+module, we can conclude
that the $R$\+module $F$ is contraadjusted, too.
 By Proposition~\ref{very-flat-contramodules-are-flat}, $F$ is
a flat $\fR$\+contramodule.
 By~\cite[Corollary~D.1.7]{Pcosh} or~\cite[Corollary~6.15]{PR},
all flat $\fR$\+contramodules are $r$\+separated.
 So the $R$\+module $F$ is $r$\+separated, and it follows that its
submodule $K$ is $r$\+separated, too.

 Thus $K\rarrow F$ is an injective morphism of $R$\+contraadjusted
$r$\+separated $r$\+contramodule $R$\+modules with the cokernel~$C$.
\end{proof}

\begin{rem} \label{bounded-torsion-weakened-remark}
 Corollary~\ref{bounded-torsion-cokernel-of-contraadjusted-separated}
is a generalization of
Corollary~\ref{noetherian-cokernel-of-contraadjusted-separated}, and
Corollary~\ref{bounded-torsion-cokernel-of-contraadjusted-separated}
as it is stated can be, in turn, generalized a bit further, though
admittedly not much.

 As above, let $h_r(M)\subset M$ denote the maximal $r$\+divisible
submodule in an $R$\+module~$M$. 
 In particular, $h_r(R)$ is an ideal in~$R$.
 Let $C$ be an $r$\+contramodule $R$\+module.
 For any element $c\in C$, consider the $R$\+module morphism
$f_c\:R\rarrow C$ taking~$1$ to~$c$.
 Since there are no $r$\+divisible submodules in $C$, we have
$f_c(h_r(R))=0$.
 Hence the ideal $h_r(R)\subset R$ acts by zero in $C$, so $C$ is
an $R/h_r(R)$\+module.

 Furthermore, an $R/h_r(R)$\+module is contraadjusted if and only if
it is contraadjusted as an $R$\+module~\cite[Lemma~1.6.6(a)]{Pcosh},
and the property of an $R$\+module to be an $r$\+contramodule only
depends on its underlying abelian group structure and the action of
the operator~$r$ \cite[Remark~5.5]{Pcta}.
 So any $R$\+contraadjusted $r$\+contramodule $R$\+module $C$ is at
the same time an $R/h_r(R)$\+contraadjusted $\bar r$\+contramodule
$R/h_r(R)$\+module, where $\bar r\in R/h_r(R)$ denotes the image of
the element $r\in R$.

 Assume that the $r$\+torsion (\,$=$~$\bar r$\+torsion) in $R/h_r(R)$
is bounded.
 Applying
Corollary~\ref{bounded-torsion-cokernel-of-contraadjusted-separated}
to the $R/h_r(R)$\+module $C$, we conclude that it is the cokernel of
an injective morphism of $R/h_r(R)$\+contraadjusted $\bar r$\+separated
$\bar r$\+contramodule $R/h_r(R)$\+modules, or, which is the same,
$R$\+contraadjusted $r$\+separated $r$\+contramodule $R$\+modules.

 For any $R$\+module $M$ and integer $n\ge1$, any element
$\bar x\in M/h_r(M)$ such that $r^n\bar x=0$ in $M/h_r(M)$ can be
lifted to an element $x\in M$ such that $r^nx=0$ in~$M$.
 Hence the $r$\+torsion in $M/h_r(M)$ is bounded whenever
the $r$\+torsion in $M$ is bounded, but the converse is not true.
 In fact, the $r$\+torsion in $M/h_r(M)$ is bounded if and only if
the $r$\+torsion in $M$ is a sum of bounded $r$\+torsion and
$r$\+divisible $r$\+torsion.

 To sum up, the condition that the $r$\+torsion in $R$ is bounded
in Corollary~\ref{bounded-torsion-cokernel-of-contraadjusted-separated}
can be replaced with the weaker condition that the $r$\+torsion in
$R/h_r(R)$ is bounded.
 In other words, $r$\+divisible $r$\+torsion in $R$ does not yet
present a problem for
Corollary~\ref{bounded-torsion-cokernel-of-contraadjusted-separated},
but any non-$r$-divisible unbounded $r$\+torsion does
(cf.\ Proposition~\ref{all-quotseparated-if-and-only-if} below).
 Consequently, the same applies to
Main Lemma~\ref{bounded-torsion-main-lemma}
(cf.\ the proof in Section~\ref{noetherian-mlemma-secn}).
 So we can prove the conclusion of
Main Lemma~\ref{bounded-torsion-main-lemma} with our methods
assuming that the $r$\+torsion in $R$ is a sum of bounded
$r$\+torsion and $r$\+divisible $r$\+torsion.
\end{rem}

\subsection{Quotseparated $r$-contramodule $R$-modules}
\label{quotseparated-subsecn}
 The rest of Section~\ref{contramodule-approx-secn} is devoted to
the discussion of a partial extension of
Corollary~\ref{bounded-torsion-cokernel-of-contraadjusted-separated}
to arbitrary commutative rings $R$ with an element $r\in R$.

 Notice first of all that, dropping the contraadjustness requirements,
a given $r$\+contramodule $R$\+module needs to be a quotient of
an $r$\+separated $r$\+contramodule $R$\+module for the conclusion
Corollary~\ref{bounded-torsion-cokernel-of-contraadjusted-separated}
to have a chance to hold for it.
 This idea is captured in the following definition.

\begin{defn}
 Let $R$ be a commutative ring and $r\in R$ be an element.
 An $r$\+contramodule $R$\+module is called \emph{quotseparated}
if it is a quotient $R$\+module of an $r$\+separated
$r$\+contramodule $R$\+module.
 We denote the full subcategory of quotseparated $r$\+contramodule
$R$\+modules by $R\modl_{r\ctra}^\qs\subset R\modl_{r\ctra}$.
\end{defn}

\begin{lem}
 Let $R$ be a commutative ring and $r\in R$ be an element.
 Then \par
\textup{(a)} the full subcategory $R\modl_{r\ctra}^\qs$ is closed
under subobjects, quotient objects, and infinite products
in $R\modl_{r\ctra}$; \par
\textup{(b)} any $r$\+contramodule $R$\+module is an extension
of two quotseparated $r$\+contra\-module $R$\+modules.
\end{lem}

\begin{proof}
 Part~(a): the full subcategory $R\modl_{r\ctra}^\qs\subset
R\modl_{r\ctra}$ is closed under quotient objects by definition,
and it is closed under subobjects and infinite products because
the full subcategory of $r$\+separated $r$\+contramodules is
closed under subobjects and infinite products in $R\modl_{r\ctra}$.
 Part~(b): for any $r$\+contramodule $R$\+module $C$, one has
$\Delta_r(C)=C$, hence $C$ is the middle term of the short exact
sequence of $r$\+contramodule $R$\+modules provided by
Sublemma~\ref{delta-lambda-r-sequence}.
 Now $\Lambda_r(C)=\varprojlim_{n\ge1}C/r^nC$ is an $r$\+separated
$r$\+contramodule, and $\varprojlim_{n\ge1}^1\.{}_{r^n}C$ is
a quotient $r$\+contramodule $R$\+module of an $r$\+separated
$r$\+contramodule $R$\+module $\prod_{n\ge1}\.{}_{r^n}C$.
\end{proof}

 The following theorem is a generalization of
Theorem~\ref{bounded-torsion-contramodule-category-equivalence}
to the unbounded torsion case.
 As in Section~\ref{bounded-torsion-veryflat-cotorsion-theory-subsecn},
we set $\fR=\varprojlim_{n\ge1} R/r^nR$.

\begin{thm} \label{quotseparated-contramodule-category-equivalence}
 Let $R$ be a commutative ring and $r\in R$ be an element.
 Then the forgetful functor\/ $\fR\contra\rarrow R\modl$ induces
an equivalence of abelian categories\/ $\fR\contra\simeq
R\modl_{r\ctra}^\qs$.
\end{thm}

\begin{proof}
 For any $r$\+contramodule $R$\+module $C$, consider the $R$\+module
$R[C]$ freely generated by the elements of~$C$.
 By the adjunction property of the functor~$\Delta_r$,
the natural surjective $R$\+module morphism $R[C]\rarrow C$ induces
an $r$\+contramodule $R$\+module morphism $\Delta_r(R[C])\rarrow C$.
 According to the discussion in~\cite[Example~3.6(3)]{Pper},
the forgetful functor identifies the category $\fR\contra$ with
the full subcategory in $R\modl_{r\ctra}$ consisting of all
the $r$\+contramodule $R$\+modules $C$ for which the morphism
$\Delta_r(R[C])\rarrow C$ factorizes as $\Delta_r(R[C])\rarrow
\Lambda_r(R[C])\rarrow C$.

 Now the $r$\+contramodule $R$\+module $\Lambda_r(R[C])$ is
$r$\+separated, hence the latter condition implies that
the $r$\+contramodule $R$\+module $C$ is quotseparated.
 Conversely, if $C$ is the quotient of an $r$\+separated
$r$\+contramodule $R$\+module $B$, then the morphism
$\Delta_r(R[C])\rarrow C$ factorizes through $B$, as $\Delta_r(R[C])$
is a projective object of $R\modl_{r\ctra}$.
 Finally, for any $R$\+module $A$, any morphism from $\Delta_r(A)$ to
an $r$\+separated $R$\+module factorizes through $\Lambda_r(A)$,
as $\Lambda_r(A)=\Lambda_r(\Delta_r(A))$ is the maximal $r$\+separated
quotient $R$\+module of $\Delta_r(A)$.
\end{proof}

 Let us point out that both the abelian categories $R\modl_{r\ctra}$ and
$R\modl_{r\ctra}^\qs$ have enough projective objects~\cite{PR,Pper},
but the projective quotseparated $r$\+contramodule $R$\+modules are
quite different from the projective $r$\+contramodule $R$\+modules.

 The next proposition should be compared with the discussion
in Remark~\ref{bounded-torsion-weakened-remark}.

\begin{prop} \label{all-quotseparated-if-and-only-if}
 Let $R$ be a commutative ring and $r\in R$ be an element.
 Then all the $r$\+contramodule $R$\+modules are quotseparated
if and only if the $r$\+torsion in the ring $R/h_r(R)$ is bounded. 
\end{prop}

\begin{proof}
 Notice that $\fR=\Lambda_r(R)=\Lambda_r(R/h_r(R))$.
 So if the $r$\+torsion in $R/h_r(R)$ is bounded, then
$R\modl_{r\ctra}=R/h_r(R)\modl_{\bar r\ctra}=\fR\contra$ by
Theorem~\ref{bounded-torsion-contramodule-category-equivalence}.
 Conversely, if $R\modl_{r\ctra}=\fR\contra$, then the $r$\+torsion
in $\Delta_r(R)$ is bounded according to the discussion
in~\cite[Example~5.2(7)]{Pper}, and $R/h_r(R)$ is a subring
in $\Delta_r(R)$.
\end{proof}

\begin{rem}
 The following aspect of the notion of a quotseparated
$r$\+contramodule $R$\+module is worth pointing out.
 All the conditions of $r$\+contraadjustedness, $r$\+contramoduleness,
$r$\+separatedness, and $r$\+completeness of an $R$\+module $C$ only
depend on the underlying abelian group of $C$ and the additive
operator $r\:C\rarrow C$.
 So they can be equivalently considered as properties of modules
over the ring of polynomials $\boZ[r]$ rather than over a ring~$R$.

 On the other hand, all $r$\+contramodule $\boZ[r]$\+modules are
quotseparated by Proposition~\ref{all-quotseparated-if-and-only-if},
but $r$\+contramodule $R$\+modules may be not.
 So an $r$\+contramodule $R$\+module $C$ always can be presented as
a quotient $\boZ[r]$\+module of an $r$\+separated $r$\+contramodule
$\boZ[r]$\+module $B$, but defining an action of the ring $R$ in $B$
compatible with the given action of $R$ in $C$ may be an impossible
task when the ring $R$ contains too much $r$\+torsion.
\end{rem}

\subsection{Very flat cotorsion theory in quotseparated
$r$-contramodule $R$\+modules}
\label{veryflat-cotorsion-theory-quotseparated-subsecn}
 The aim of this section is to prove the following generalization
of Corollaries~\ref{noetherian-cokernel-of-contraadjusted-separated}
and~\ref{bounded-torsion-cokernel-of-contraadjusted-separated} to
arbitrary commutative rings $R$ with an element $r\in R$.
 It claims that the quotseparatedness condition, which is necessary
as we explained in the beginning of Section~\ref{quotseparated-subsecn},
is in fact also sufficient.
{\hfuzz=3.2pt\par}

\begin{prop} \label{quotseparated-cokernel-of-contraadjusted-separated}
 Let $R$ be a commutative ring and $r\in R$ be an element.
 Then any $R$\+contraadjusted quotseparated $r$\+contramodule
$R$\+module can be presented as the cokernel of an injective
morphism of $R$\+contraadjusted $r$\+separated $r$\+complete
$R$\+modules.
\end{prop}

 The proof of
Proposition~\ref{quotseparated-cokernel-of-contraadjusted-separated}
uses an approximation sequence in the very flat cotorsion theory on
the category $R\modl_{r\ctra}^\qs=\fR\contra$, which is provided by
the results of~\cite[Sections~D.1 and~D.4]{Pcosh}.
 The exposition below in this section contains a sketch of some
arguments from~\cite[Section~D]{Pcosh} specialized to the situation
at hand.

 Let $R$ be a commutative ring and $r\in R$ be an element.
 We keep the notation $\fR=\varprojlim_{n\ge1}R/r^nR$.
 Recall from
Section~\ref{bounded-torsion-veryflat-cotorsion-theory-subsecn}
that an $\fR$\+contramodule $F$ is said to be \emph{flat} if
the $R/r^nR$\+module $F/r^nF$ is flat for every $n\ge1$.
 There are enough projective objects in $\fR\contra$, and all
the projective $\fR$\+contramodules are flat~\cite[Lemma~D.1.6]{Pcosh},
\cite[Lemma~6.9]{PR}.
 All the flat $\fR$\+contramodules are
$r$\+separated~\cite[Corollary~D.1.7]{Pcosh}, \cite[Corollary~6.15]{PR}.
 The class of all flat $\fR$\+contramodules is closed under extensions,
the kernels of epimorphisms, and filtered inductive limits in
$\fR\contra$ \cite[Lemmas~D.1.4\+-5]{Pcosh},
\cite[Corollary~7.1]{Pcosh}.

 A quotseparated $r$\+contramodule $R$\+module $F$ is said to be
\emph{very flat} if the $R/r^nR$\+module $F/r^nF$ is very flat
for every $n\ge1$ (cf.~\cite[Section~C.3]{Pcosh}) and
Remark~\ref{two-very-flat-cotorsion-theories-remark} above).
 This is equivalent to the $R/r^nR$\+module $F/r^nF$ being flat for
every $n\ge1$ and the $R/rR$\+module $F/rF$ being very
flat~\cite[Lemma~1.6.8(b)]{Pcosh}.
 All the projective quotseparated $r$\+contramodule $R$\+modules are
very flat; and the class of all very flat quotseparated
$r$\+contramodule $R$\+modules is closed under extensions and
the kernels of epimorphisms in $\fR\modl_{r\ctra}^\qs$.
 The projective dimension of any very flat quotseparated
$r$\+contramodule $R$\+module (as an object of $R\modl_{r\ctra}^\qs$)
does not exceed~1 \,\cite[Corollary~D.4.1]{Pcosh}.

 One can say that a quotseparated $r$\+contramodule $R$\+module $K$ is
\emph{contraadjusted} if $\Ext^1_{R\modl_{r\ctra}^\qs}(F,K)=0$ for all
very flat quotseparated $r$\+contramodule $R$\+modules~$F$.
 Then the claim that the pair of full subcategories (very flat
quotseparated $r$\+contramodule $R$\+modules, contraadjusted
quotseparated $r$\+contramodule $R$\+modules) is a complete cotorsion
theory in $R\modl_{r\ctra}^\qs$ follows already from the general
results of~\cite[Example~7.12(3)]{PR}.
 The trick is to understand what the contraadjusted quotseparated
$r$\+contramodule $R$\+modules are.

 Recall that an $R/rR$\+module is contraadjusted if and only if it is
contraadjusted as an $R$\+module~\cite[Lemma~1.6.6(a)]{Pcosh}
(cf.\ Lemma~\ref{contraadjustness-reflected-lem} below).

\begin{lem} \label{contraadjusted-r-projective-system}
 Let $K$ be an $R$\+module such that the $R/rR$\+module $K/rK$ is
contraadjusted.
 Then for any $n\ge1$ the kernel of the natural surjective map
$K/r^{n+1}K\rarrow K/r^n K$ is a contraadjusted $R/rR$\+module.
\end{lem}

\begin{proof}
 We have an exact sequence $K/rK\overset{r^n}\rarrow K/r^{n+1}K
\rarrow K/r^nK\rarrow0$, so the kernel of the map
$K/r^{n+1}K\rarrow K/r^nK$ is a quotient $R/rR$\+module of~$K/rK$.
\end{proof}

 Let us introduce the following piece of terminology relevant for
the argument in the rest of this section.
 We will say that an $R$\+module $K$ is \emph{$r$\+solid} if it is
$r$\+separated, $r$\+complete, and the $R/rR$\+module $K/rK$
is contraadjusted.
 Obviously, all the $r$\+solid $R$\+modules belong to the category
$R\modl_{r\ctra}^\qs$.

\begin{lem} \label{contraadjusted-as-module}
 Let $F$ be a flat $R$\+module such that the $R/rR$\+module $F/rF$
is very flat, and let $K$ be an $r$\+solid $R$\+module.
 Then\/ $\Ext_R^1(F,K)=0$.
\end{lem}

\begin{proof}
 Denoting the kernel of the morphism $K/r^{n+1}K\rarrow K/r^nK$ by
$L_n$, \,$n\ge0$, we have $\Ext_R^1(F,L_n)=\Ext^1_{R/rR}(F/rF,L_n)=0$
by Lemma~\ref{contraadjusted-r-projective-system}.
 Hence the assertion follows by virtue of the dual Eklof
Lemma~\cite[Proposition~18]{ET}.
\end{proof}

\begin{lem} \label{contraadjusted-as-quotseparated-contramodule}
 Let $F$ be a very flat quotseparated $r$\+contramodule $R$\+module
and $K$ be an $r$\+solid $R$\+module.
 Then\/ $\Ext_{R\modl_{r\ctra}^\qs}^1(F,K)=0$.
\end{lem}

\begin{proof}
 This is~\cite[Lemma~D.1.8]{Pcosh}.
 Applying the Eklof Lemma for abelian categories~\cite[Lemma~4.5]{PR}
to the category opposite to $R\modl_{r\ctra}^\qs$, it remains to
check that $\Ext_{R\modl_{r\ctra}^\qs}^1(F,L)=0$ for any contraadjusted
$R/rR$\+module~$L$.
 Let $0\rarrow H\rarrow G\rarrow F\rarrow0$ be a short exact sequence
of quotseparated $r$\+contramodule $R$\+modules with a projective
quotseparated $r$\+contramodule $R$\+module~$G$.
 By~\cite[Lemma~D.1.4]{Pcosh} or~\cite[Lemma~6.7]{PR}, the short
sequence of $R/rR$\+modules $0\rarrow H/rH\rarrow G/rG\rarrow F/rF
\rarrow0$ is exact.
 Since the $R/rR$\+module $F/rF$ is very flat, it follows that any
$R/rR$\+module morphism $H/rH\rarrow L$ can be extended to
an $R/rR$\+module morphism $G/rG\rarrow L$.
 Thus any $R$\+module morphism $H\rarrow L$ can be extended to
an $R$\+module morphism $G\rarrow L$.
\end{proof}

 Lemma~\ref{contraadjusted-as-module} tells that the $r$\+solid
$R$\+modules are contraadjusted as $R$\+modules, while
Lemma~\ref{contraadjusted-as-quotseparated-contramodule} shows they
are contraadjusted as quotseparated $r$\+contramodule $R$\+modules
(in the sense of the above definition).
 It follows, in particular, that the cokernel of any injective
morphism of $r$\+solid $R$\+modules is also contraadjusted both
in the sense of $R\modl$ and $R\modl_{r\ctra}^\qs$
(cf.\ Lemma~\ref{quotseparated-contraadjusted-preenvelope} below).

 The following construction plays a key role.

\begin{lem} \label{separated-contraadjusted-preenvelope-construction}
 Let $C$ be an $r$\+separated $r$\+contramodule $R$\+module.
 Then there exists a short exact sequence of quotseparated
$r$\+contramodule $R$\+modules\/ $0\rarrow C\rarrow K\rarrow F\rarrow0$
with a very flat quotseparated $r$\+contramodule $R$\+module $F$ and
an $r$\+solid $R$\+module~$K$.
\end{lem}

\begin{proof}
 This is a very simple particular case of~\cite[Lemma~D.4.4]{Pcosh}.
 Let $0\rarrow C\rarrow K'\rarrow F'\rarrow0$ be a short exact sequence
of $R$\+modules with a very flat $R$\+module $F'$ and a contraadjusted
$R$\+module~$K'$.
 Then for every $n\ge1$ the short sequence of $R/r^nR$\+modules
$0\rarrow C/r^nC\rarrow K'/r^nK'\rarrow F'/r^nF'\rarrow0$ is exact.
 Passing to the projective limit as $n\to\infty$, we obtain a short
exact sequence $0\rarrow C\rarrow K\rarrow F\rarrow 0$ with
$C=\varprojlim_n C/r^nC$ (by assumption), $K=\varprojlim_n K'/r^nK'$
and $F=\varprojlim_n F'/r^nF'$.
 As $K/rK=K'/rK'$ is a contraadjusted $R/rR$\+module and
$F/r^nF=F'/r^nF'$ are very flat $R/r^nR$\+modules, all the desired
properties hold.
\end{proof}

\begin{lem} \label{quotseparated-very-flat-precover}
 Any quotseparated $r$\+contramodule $R$\+module $C$ can be included
into a short exact sequence of quotseparated $r$\+contramodule
$R$\+modules\/ $0\rarrow K\rarrow F\rarrow C\rarrow0$ with a very flat
quotseparated $r$\+contramodule $R$\+module $F$ and
an $r$\+solid $R$\+module~$K$.
\end{lem}

\begin{proof}
 Let $0\rarrow B\rarrow P\rarrow C\rarrow0$ be a short exact
sequence of quotseparated $r$\+contramodule $R$\+modules with
a projective quotseparated $r$\+contramodule $R$\+module~$P$.
 Then the $R$\+module $B$ is $r$\+separated as a submodule of
an $r$\+separated $R$\+module $P$, so  the construction of
Lemma~\ref{separated-contraadjusted-preenvelope-construction} is
applicable to~$B$.
 It remains to use the construction from one of the Salce
Lemmas~\cite[second paragraph of the proof of Theorem~10]{ET}
in order to obtain the desired short exact sequence.
\end{proof}

\begin{lem} \label{quotseparated-contraadjusted-preenvelope}
 Any quotseparated $r$\+contramodule $R$\+module $C$ can be included
into a short exact sequence of quotseparated $r$\+contramodule
$R$\+modules\/ $0\rarrow C\rarrow K\rarrow F\rarrow0$ such that $F$
is a very flat quotseparated $r$\+contramodule $R$\+module and $K$ is
the cokernel of an injective morphism of $r$\+solid $R$\+modules.
\end{lem}

\begin{proof}
 Follows from Lemmas~\ref{quotseparated-very-flat-precover}
and~\ref{separated-contraadjusted-preenvelope-construction}
(the latter of which has to be applied to a very flat quotseparated
$r$\+contramodule $R$\+module---one needs to keep in mind that these
are $r$\+separated).
 We refer to~\cite[Corollary~D.4.6]{Pcosh} for the details.
\end{proof}

\begin{thm}
 The pair of full subcategories (very flat quotseparated
$r$\+contramodule $R$\+modules, $R$\+contraadjusted quotseparated
$r$\+contramodule $R$\+modules) is a hereditary complete cotorsion
theory in $R\modl_{r\ctra}^\qs$.
 Moreover, a quotseparated $r$\+contra\-module $R$\+module $K$ is
$R$\+contraadjusted if and only if the $R/rR$\+module $K/rK$
is contraadjusted.

 In other words, the three classes of quotseparated $r$\+contramodule
$R$\+modules $K$ defined by the conditions
\begin{enumerate}
\renewcommand{\theenumi}{\roman{enumi}}
\item $K$ is contraadjusted as a quotseparated $r$\+contramodule
$R$\+module;
\item $K$ is contraadjusted as an $R$\+module; and
\item the $R/rR$\+module $K/rK$ is contraadjusted
\end{enumerate}
coincide with each other.
\end{thm}

\begin{proof}
 This is (an improvement over)~\cite[Corollary~C.4.8]{Pcosh}.
 The approximation sequences (with an $R$\+contraadjusted
quotseparated $r$\+con\-tramodule $R$\+module~$K$) are provided by
Lemmas~\ref{quotseparated-very-flat-precover}
and~\ref{quotseparated-contraadjusted-preenvelope}, so in order to
prove all the assertions of the theorem it suffices to show that
$\Ext^1_{R\modl_{r\ctra}^\qs}(F,K)=0$ for any very flat quotseparated
$r$\+contramodule $R$\+module $F$ and any quotseparated
$r$\+contramodule $R$\+module $K$ such that the $R/rR$\+module
$K/rK$ is contraadjusted.

 Indeed, applying Lemma~\ref{quotseparated-very-flat-precover} to
the quotseparated $r$\+contramodule $R$\+module $K$, we obtain
a short exact sequence of quotseparated $r$\+contramodules
$R$\+modules $0\rarrow L\rarrow G\rarrow K\rarrow 0$ with
an $r$\+solid $R$\+module $L$ and a very flat quotseparated
$r$\+contramodule $R$\+module~$G$.
 Tensoring with $R/rR$ over~$R$, we get an exact sequence of
$R/rR$\+modules $L/rL\rarrow G/rG\rarrow K/rK\rarrow0$.
 Since the class of contraadjusted $R/rR$\+modules is closed under
extensions and quotients, it follows that $G/rG$ is a contraadjusted
$R/rR$\+module.
 Hence $G$ is an $r$\+solid $R$\+module, so
$\Ext^1_{R\modl_{r\ctra}^\qs}(F,G)=0$ by
Lemma~\ref{contraadjusted-as-quotseparated-contramodule}.
 As the projective dimension of the object $F\in R\modl_{r\ctra}^\qs$
does not exceed~$1$, it follows that
$\Ext^1_{R\modl_{r\ctra}^\qs}(F,K)=0$.
\end{proof}

\begin{proof}[Proof of
Proposition~\ref{quotseparated-cokernel-of-contraadjusted-separated}]
 Let $C$ be an $R$\+contraadjusted quotseparated $r$\+contra\-module
$R$\+module.
 Applying Lemma~\ref{quotseparated-very-flat-precover}, we obtain
a short exact sequence $0\rarrow K\rarrow F\rarrow C\rarrow0$ with
an $r$\+solid $R$\+module $K$ and a very flat quotseparated
$r$\+contramodule $R$\+module~$F$.
 By Lemma~\ref{contraadjusted-as-module}, the $R$\+module $K$ is
contraadjusted, and it follows that so is the $R$\+module~$F$.
 As any flat quotseparated $r$\+contramodule $R$\+module,
the $R$\+module $F$ is $r$\+separated.
 Thus $K\rarrow F$ is an injective morphism of $R$\+contraadjusted
$r$\+separated $r$\+contramodules with the cokernel~$C$.
\end{proof}

\Section{Obtainable Modules II}  \label{obtainable-II-secn}

 In this section, as in Section~\ref{obtainable-I-secn} before, we
work over an associative ring~$R$.

 Let $\sE$ and $\sF$ be two classes of left $R$\+modules and
$n\ge1$ be an integer.
 Denote by ${}^{\perp_{\ge n}}\.\sE\subset R\modl$ the class of all left
$R$\+modules such that $\Ext_R^i(F,E)=0$ for all $E\in\sE$ and $i\ge n$,
and by $\sF^{\perp_{\ge n}}\subset R\modl$ the class of all left
$R$\+modules $C$ such that $\Ext_R^i(F,C)=0$ for all $F\in\sF$ and
$i\ge n$.

 Let $\sE\subset R\modl$ be a fixed class of left $R$\+modules.
 Set $\sF={}^{\perp_{\ge1}}\.\sE$ and $\sC_n=\sF^{\perp_{\ge n}}$,
\,$n\ge1$.
 Clearly, one has $\sC_n\subset\sC_{n+1}$ for all $n\ge1$.

 As in Section~\ref{obtainable-I-secn}, our aim is to describe
the class $\sC=\sC_1$ in terms of the class~$\sE$.
 In this section we develop a more powerful procedure of obtaining
objects of the class $\sC_1$ from objects of the class~$\sE$, which
provides for the possibility of making detours through the classes
$\sC_n$ (or, more specifically, $\sC_2$) when producing objects
of the class $\sC_1$ from objects of the ``seed'' class~$\sE$.

\begin{lem} \label{classes-c-n-closedness}
 For any class of left $R$\+modules\/ $\sF\subset R\modl$, the classes
of left $R$\+modules\/ $\sC_n=\sF^{\perp_{\ge n}}$ have the following
properties:
\begin{enumerate}
\renewcommand{\theenumi}{\roman{enumi}}
\item one has\/ $\sC_n\subset\sC_{n+1}$ for all $n\ge1$;
\item for every $n\ge1$, the class\/ $\sC_n\subset R\modl$ is closed
under the passages to direct summands, extensions, cokernels of
injective morphisms, infinite products, and transfinitely iterated
extensions in the sense of the projective limit;
\item the kernel of any surjective morphism from an object of\/
$\sC_{n+1}$ to an object of\/ $\sC_n$ belongs to\/~$\sC_{n+1}$;
\item the cokernel of any injective morphism from an object of\/
$\sC_{n+1}$ to an object of\/ $\sC_n$ belongs to\/~$\sC_n$.
\end{enumerate}
\end{lem}

\begin{proof}
 The property~(i) holds by the definition, and (iii\+iv)~follow
from the long exact sequence of $\Ext^*_R(F,{-})$, \ $F\in\sF$.
 To prove~(ii), pick for every module $F\in\sF$ one of its
$(n-\nobreak1)$\+th syzygy modules, that is a module $G_{n-1}$ in
an exact sequence $0\rarrow G_{n-1}\rarrow P_{n-1}\rarrow\dotsb\rarrow
P_1\rarrow F\rarrow0$ with projective $R$\+modules~$P_i$.
 Denote by $\sG_{n-1}$, \,$n\ge2$ the class of all the left $R$\+modules
$G_{n-1}$ so obtained (and set $\sG_0=\sF$).
 Then one has $\sF^{\perp_{\ge n}}=\sG_{n-1}^{\perp_{\ge1}}$ for every
$n\ge1$, and the assertion follows from
Lemma~\ref{positive-ext-right-orthogonal-closedness} applied to
the class $\sG_{n-1}\subset R\modl$.
\end{proof}

\begin{defn} \label{1-2-right-obtainable-def}
 The pair of classes of left $R$\+modules \emph{right\/ $1$\+obtainable}
and \emph{right\/ $2$\+obtainable} from a given class $\sE\subset
R\modl$ is defined as the (obviously, unique) minimal pair of classes
of left $R$\+modules satisfying the following generation rules:
\begin{enumerate}
\renewcommand{\theenumi}{\roman{enumi}}
\item all the $R$\+modules from $\sE$ are right $1$\+obtainable;
all the right $1$\+obtainable $R$\+modules are right $2$\+obtainable;
\item all the $R$\+modules simply right obtainable (in the sense of
Definition~\ref{simply-right-obtainable-def}) from
right $1$\+obtainable $R$\+modules are right $1$\+obtainable;
all the $R$\+modules simply right obtainable from right $2$\+obtainable
$R$\+modules are right $2$\+obtainable;
\item the kernel of any surjective morphism from a right
$2$\+obtainable $R$\+module to a right $1$\+obtainable $R$\+module is
right $2$\+obtainable;
\item the cokernel of any injective morphism from a right
$2$\+obtainable $R$\+module to a right $1$\+obtainable $R$\+module is
right $1$\+obtainable.
\end{enumerate}
\end{defn}

 The purpose of Definition~\ref{1-2-right-obtainable-def} is that
it describes a procedure according to which, starting from a seed
class $\sE$, we first obtain some objects of the class $\sC_1$, then
use them to produce some objects of the class $\sC_2$, and then use
the latter in order to obtain even more objects of the class~$\sC_1$,
etc.
 A similar obtainability procedure involving $n$\+obtainable
modules for all $n\ge\nobreak2$ would be even more powerful, in
the sense that it would possibly allow to obtain more modules
(at least, in principle); but we will not use any $n$\+obtainable
modules beyond the $2$\+obtainable ones in the arguments in this
paper, so we do not define them.

\begin{lem} \label{1-2-right-obtainable-orthogonal}
 For any class of left $R$\+modules\/ $\sE\subset R\modl$, all the left
$R$\+modules right\/ $1$\+obtainable from\/ $\sE$ belong to the class\/
$\sC_1=({}^{\perp_{\ge1}}\.\sE)^{\perp_{\ge1}}\subset R\modl$, and
all the left $R$\+modules right\/ $2$\+obtainable from\/ $\sE$
belong to the class\/ $\sC_2=({}^{\perp_{\ge1}}\.\sE)^{\perp_{\ge2}}
\subset R\modl$.
\end{lem}

\begin{proof}
 Follows from Lemma~\ref{classes-c-n-closedness}.
\end{proof}

\begin{rem} \label{1-2-obtainability-limitations-of-use}
 As with the simple right obtainability rules of
Definition~\ref{simply-right-obtainable-def}, we will not use the full
strength of Definition~\ref{1-2-right-obtainable-def} in this paper.
 In addition to the limitations mentioned in
Remark~\ref{simple-obtainability-limitations-of-use-remark},
the following one should be mentioned here: the second part of
the rule~(i) and the rule~(iii) from
Definition~\ref{1-2-right-obtainable-def} will be only used in
the following weak conjunction:
\begin{enumerate}
\renewcommand{\theenumi}{\roman{enumi}$'$}
\setcounter{enumi}{2}
\item the kernel of any surjective morphism between two right
$1$\+obtainable $R$\+modules is right $2$\+obtainable.
\end{enumerate}
 Besides, not nearly all the complicated possibilities of transfinite
iteration of the generation rules~(i\+iv) will be used in
the actual constructions.
\end{rem}

\Section{Noetherian Main Lemma}  \label{noetherian-mlemma-secn}

 The proof of Main Lemmas~\ref{noetherian-main-lemma}
and~\ref{bounded-torsion-main-lemma} in this section is based on
the results of
Corollaries~\ref{noetherian-cokernel-of-contraadjusted-separated}
and~\ref{bounded-torsion-cokernel-of-contraadjusted-separated},
respectively.
 The main lemmas are deduced from the following main propositions.

\begin{mpr} \label{noetherian-main-prop}
 Let $R$ be a Noetherian commutative ring and $r\in R$ be an element.
 Then an $R$\+module is contraadjusted if and only if it is right\/
$1$\+obtainable from contraadjusted $R/rR$\+modules and
contraadjusted $R[r^{-1}]$\+modules.
\end{mpr}

\begin{mpr} \label{bounded-torsion-main-prop}
 Let $R$ be a commutative ring and $r\in R$ be an element such that
the $r$\+torsion in $R$ is bounded.
 Then an $R$\+module is contraadjusted if and only if it is right\/
$1$\+obtainable from contraadjusted $R/rR$\+modules and
contraadjusted $R[r^{-1}]$\+modules. 
\end{mpr}

 As the bounded torsion assumption in
Corollary~\ref{bounded-torsion-cokernel-of-contraadjusted-separated}
can be weakened to the assumption that the $r$\+torsion in $R$ is
a sum of bounded $r$\+torsion and $r$\+divisible $r$\+torsion
(see Remark~\ref{bounded-torsion-weakened-remark}), this assumption
in Main Proposition~\ref{bounded-torsion-main-prop} and
Main Lemma~\ref{bounded-torsion-main-lemma} also can be similarly
weakened.

 The following lemma shows that there is no ambiguity in
the formulations of the main propositions.

\begin{lem} \label{contraadjustness-reflected-lem}
 Let $f\:R\rarrow R'$ be a morphism of commutative rings such that
for every element $r'\in R'$ there exist an element $r\in R$ and
an invertible element $u\in R'$ for which $r'=uf(r)$.
 Then an $R'$\+module is contraadjusted if and only if it is
contraadjusted as an $R$\+module.
\end{lem}

\begin{proof}
 By~\cite[Lemma~1.2.2(a)]{Pcosh}, contraadjustness is preserved by
the restrictions of scalars, so all the contraadjusted $R'$\+modules
are contraadjusted as $R$\+modules.
 Conversely, let $C$ be an $R'$\+module that is contraadjusted
as an $R$\+module.
 We have to show that $\Ext_{R'}^1(R'[r'{}^{-1}],C)=0$ for all
$r'\in R'$.
 By assumption, we have $r'=uf(r)$, where $u\in R'$ is invertible
and $r\in R$.
 Hence $R'[r'{}^{-1}]=R'[f(r)^{-1}]=R'\ot_RR[r^{-1}]$ and
$\Ext_{R'}^1(R'[r'{}^{-1}],C)=\Ext_R^1(R[r^{-1}],C)=0$.
\end{proof}

 In particular, an $R/rR$\+module is contraadjusted if and only if
it is contraadjusted as an $R$\+module~\cite[Lemma~1.6.6(a)]{Pcosh},
and an $R[r^{-1}]$\+module is contraadjusted if and only if it is
contraadjusted as an $R$\+module.

\begin{lem} \label{contraadjusted-separated-obtainable-as-projlim}
 Let $R$ be a commutative ring and $r\in R$ be an element.
 Then any $R$\+contraadjusted $r$\+separated $r$\+complete $R$\+module
is an infinitely iterated extension, in the sense of the projective
limit, of contraadjusted $R/rR$\+modules.
\end{lem}

\begin{proof}
 Let $K$ be an $R$\+contraadjusted $r$\+separated $r$\+complete
$R$\+module.
 Then we have $K=\varprojlim_{n\ge1}K/r^nK$, and it remains to
apply Lemma~\ref{contraadjusted-r-projective-system}.
\end{proof}

\begin{lem} \label{hom-contraadjusted-lemma}
 Let $R$ be a commutative ring and $r\in R$ be an element.
 Then the $R[r^{-1}]$\+module\/ $\Hom_R(R[r^{-1}],C)$ is contraadjusted
for any contraadjusted $R$\+module~$C$.
\emergencystretch=0em\hfuzz=2.2pt
\end{lem}

\begin{proof}
 This is~\cite[Lemma~1.2.1(b) or Lemma~1.2.3(a)]{Pcosh}.
\end{proof}

\begin{lem} \label{2-obtainable-from-contraadjusted}
 For any commutative ring $R$, any $R$\+module is right $2$\+obtainable
from contraadjusted $R$\+modules.
\end{lem}

\begin{proof}
 Any $R$\+module $A$ can be embedded into a contraadjusted (e.~g.,
injective) $R$\+module $B$.
 The quotient $R$\+module $B/A$ is then also contraadjusted.
 Now $A$ is the kernel of the surjective morphism $B\rarrow B/A$,
so it is right $2$\+obtainable from $B$ and $B/A$ according to
the rules~(i) and~(iii) from
Definition~\ref{1-2-right-obtainable-def}, or according to
the rule~(iii$'$) from
Remark~\ref{1-2-obtainability-limitations-of-use}.
\end{proof}

\begin{proof}[Proof of Main Propositions~\ref{noetherian-main-prop}
and~\ref{bounded-torsion-main-prop}]
 The (trivial) proof of the ``if'' assertion does not depend on
the Noetherianity/bounded torsion assumption.
 Let $\sE$ be a class of contraadjusted $R$\+modules; then all the very
flat $R$\+modules $F$ belong to the class $\sF={}^{\perp_{\ge1}}\.\sE$.
 By Lemma~\ref{1-2-right-obtainable-orthogonal}, we have
$C\in\sF^{\perp_{\ge1}}\subset\{F\}^{\perp_{\ge1}}$ for all
the $R$\+modules $C$ right $1$\+obtainable from $\sE$; so all such
$R$\+modules $C$ are contraadjusted.

 To prove the ``only if'' in either of the respective assumptions on
the ring $R$ or the element~$r$, denote by $\sE\subset R\modl$
the class of all contraadjusted $R/rR$\+modules and contraadjusted
$R[r^{-1}]$\+modules (viewed as $R$\+modules via the restriction of
scalars).
 Let $C$ be a contraadjusted $R$\+module.
 Our argument is based on the exact
sequence~\eqref{r-contraadjusted-sequence}.
 In view of the rules~(ii) and~(iv) from
Definition~\ref{1-2-right-obtainable-def}, in order to show that
$C$ is right $1$\+obtainable from $\sE$ it suffices to check that
the $R$\+modules $\Hom_R(R[r^{-1}],C)$ and $\Delta_r(C)$ are
right $1$\+obtainable from $\sE$ and the $R$\+module
$\Hom_R(R[r^{-1}]/R,C)$ is right $2$\+obtainable from~$\sE$.
 Let us consider these three $R$\+modules one by one.

 The $R$\+module $\Hom_R(R[r^{-1}],C)$ is a contraadjusted
$R[r^{-1}]$\+module by Lemma~\ref{hom-contraadjusted-lemma}.
 So it already belongs to~$\sE$.

 The $R$\+module $\Delta_r(C)$ is contraadjusted as a quotient
module of a contraadjusted $R$\+module~$C$.
 Besides, $\Delta_r(C)$ is an $r$\+contramodule.
 Applying
Corollary~\ref{noetherian-cokernel-of-contraadjusted-separated}
or~\ref{bounded-torsion-cokernel-of-contraadjusted-separated},
we can present $\Delta_r(C)$ as the cokernel of an injective morphism
of $R$\+contraadjusted $r$\+separated $r$\+complete $R$\+modules
$K\rarrow L$.
 By Lemma~\ref{contraadjusted-separated-obtainable-as-projlim}, both
$K$ and $L$ are obtainable as infinitely iterated extensions of
contraadjusted $R/rR$\+modules.
 Thus the $R$\+module $\Delta_r(C)$ is simply right obtainable
from~$\sE$.

 The $R$\+module $\Hom_R(R[r^{-1}]/R,C)$ is an $r$\+contramodule, and
in fact even an $r$\+separated $r$\+complete $R$\+module
(see the proof of Toy Main Proposition~\ref{toy-main-proposition}).
 So it is simply right obtainable from $R/rR$\+modules by
Lemma~\ref{r-contramodule-obtainable}
(Sublemma~\ref{derived-projlim-obtainable}(a) is sufficient).
 By Lemma~\ref{2-obtainable-from-contraadjusted}, all
the $R/rR$\+modules are right $2$\+obtainable from contraadjusted
$R/rR$\+modules.
 Hence the $R$\+module $\Hom_R(R[r^{-1}]/R,C)$ is right
$2$\+obtainable from~$\sE$.
\end{proof}

\begin{proof}[Proof of Main Lemmas~\ref{noetherian-main-lemma}
and~\ref{bounded-torsion-main-lemma}]
 The ``only if'' assertion holds, because very flatness is preserved
by the extensions of scalars
(see Lemma~\ref{very-flatness-extension-of-scalars}(a)).
 The ``if'' assertion is the nontrivial part.

 Let $F$ be an $R$\+module such that the $R/rR$\+module $F/rF$ is
very flat and the $R[r^{-1}]$\+module $F[r^{-1}]$ is very flat.
 In order to show that the $R$\+module $F$ is very flat, we will
check that $\Ext_R^1(F,C)=0$ for every contraadjusted $R$\+module~$C$.

 As above, we denote by $\sE\subset R\modl$ the class of all
contraadjusted $R/rR$\+modules and contraadjusted $R[r^{-1}]$\+modules.
 Let us check that $\Ext^i_R(F,E)=0$ for all $E\in\sE$ and $i\ge1$.
 Indeed, if $E$ is a contraadjusted $R[r^{-1}]$\+module, then we have
$\Ext^i_R(F,E)=\Ext^i_{R[r^{-1}]}(F[r^{-1}],E)=0$ by
Lemma~\ref{change-of-scalars-vanishing-tor-lemma}(a), and if $E$ is
a contraadjusted $R/rR$\+module, we have
$\Ext^i_R(F,E)=\Ext^i_{R/rR}(F/rF,E)=0$, also by
Lemma~\ref{change-of-scalars-vanishing-tor-lemma}(a).
 Hence we have $F\in\sF={}^{\perp_{\ge1}}\.\sE$.

 According to Main Proposition~\ref{noetherian-main-prop}
or~\ref{bounded-torsion-main-prop}, the $R$\+module $C$ is right
$1$\+obtainable from~$\sE$.
 By Lemma~\ref{1-2-right-obtainable-orthogonal}, it follows that
$C\in\sF^{\perp_{\ge1}}$, so $\Ext_R^1(F,C)=0$.
\end{proof}

 The proof of Main Theorem~\ref{noetherian-module-main-theorem}
is now finished.

\Section{Finitely Very Flat Main Lemma}  \label{fvf-mlemma-secn}

 Let us recall the setting of
Sections~\ref{good-subsets-outline}\+-\ref{rtimes-vf-thm-outline}.
 We consider an arbitrary commutative ring $R$ and a finite set
of elements $\r=\{r_1,~\dotsc,r_m\}\subset R$.
 For every subset $J\subset\{1,\dotsc,m\}$, we denote by~$r_J$
the product $\prod_{j\in J}r_j\in R$.
 The set of all elements $\{r_J\}$ is denoted by~$\r^\times$.
 We put $K=\{1,\dotsc,m\}\setminus J$, and denote by $R_J$
the localization-quotient ring $(R/\sum_{k\in K}r_k R)[r_J^{-1}]$
obtained by annihilating all the elements $r_k$, \,$k\in K$,
and inverting all the elements $r_j$, \,$j\in J$ in the ring~$R$.

 An $R$\+module $C$ is said to be \emph{$\r$\+contraadjusted}
if $\Ext_R^1(R[r_j^{-1}],C)=0$ for all $1\le j\le m$, and
an $R$\+module $F$ is said to be \emph{$\r$\+very flat} if
$\Ext_R^1(F,C)=0$ for every $\r$\+contraadjusted $R$\+module~$C$
(see Section~\ref{fvf-outline} for a discussion).
 We are mostly interested in $\r^\times$\+contraadjusted and
$\r^\times$\+very flat $R$\+modules.

 In this section, we start with proving
Theorem~\ref{r-very-flat-theorem}, and then proceed to deduce
Main Lemma~\ref{fvf-main-lemma}.

\begin{mpr} \label{r-very-flat-main-prop}
 Let $R$ be a commutative ring and $r_1$,~\dots, $r_m\in R$ be
a finite collection of its elements.
 Then an $R$\+module is\/ $\r^\times$\+contraadjusted if and only if
it is right\/ $1$\+obtainable from $R_J$\+modules,
where $J\subset\{1,\dotsc,m\}$.
\end{mpr}

 Notice first of all that it follows from
Main Proposition~\ref{r-very-flat-main-prop}
(by means of Lemma~\ref{2-obtainable-from-contraadjusted})
that all $R$\+modules are right\/ $2$\+obtainable from $R_J$\+modules.

\begin{lem} \label{tuple-r-contramodule-obtainable-lemma}
 Let $R$ be a commutative ring and $r_1$,~\dots, $r_m\in R$ be
a finite set of elements.
 Let $C$ be an $R$\+module such that $C$ is an $r_k$\+contramodule
for every\/ $1\le k\le m$.
 Then $C$ is simply right obtainable from
$R/(r_1R+\dotsb+r_mR)$\+modules (viewed as $R$\+modules via
the restriction of scalars).
\end{lem}

\begin{proof}
 This is a generalization of
Lemma~\ref{r-contramodule-obtainable} to the case of a finite set
of elements in $R$ instead of just one element.
 One can prove it in a way similar to the proof of
Lemma~\ref{r-contramodule-obtainable}.
 Specifically, let $I\subset R$ denote the ideal generated by
the elements $r_1$,~\dots,~$r_m$.
 According to~\cite[Theorem~5.1]{Pcta}, the full subcategory of
all $R$\+modules that are $r_k$\+contramodules for all $1\le k\le m$
depends only on the ideal~$I$ (and, in fact, even only on its
radical $\sqrt{I}\subset R$) rather than on a specific set of its
generators~$r_k$.
 Such $R$\+modules are called \emph{$I$\+contramodules}, and
the full subcategory formed by them is denoted by $R\modl_{I\ctra}
\subset R\modl$.
 The embedding functor $R\modl_{I\ctra}\rarrow R\modl$ has a left
adjoint, which is denoted by $\Delta_I\:R\modl\rarrow R\modl_{I\ctra}$
\cite[Theorem~7.2]{Pcta}.

 For every $n\ge1$, denote by $I_n\subset R$ the ideal generated by
the elements $r_1^n$,~\dots, $r_m^n\in R$.
 So one has $I_n\subset I^n\subset R$, and for any integer $n\ge1$
there exists an integer $N > n$ such that $I^N\subset I_n$.

 The following two sublemmas allow to deduce
Lemma~\ref{tuple-r-contramodule-obtainable-lemma} in the same way
as Lemma~\ref{r-contramodule-obtainable} is deduced from
Sublemmas~\ref{delta-lambda-r-sequence}
and~\ref{derived-projlim-obtainable}.

\begin{subl} \label{delta-lambda-I-sequence}
 Let $R$ be a commutative ring, $r_1$,~\dots, $r_m\in R$ be a finite
set of its elements, and $I=(r_1,\dotsc,r_m)\subset R$ be the ideal
generated by them.
 Then for any $R$\+module $A$ there is a natural short exact sequence
of $R$\+modules
$$
 0\lrarrow\varprojlim\nolimits^1_{n\ge1}D_n(A)\lrarrow\Delta_I(A)
 \lrarrow\varprojlim\nolimits_{n\ge1}A/I^nA\lrarrow0,
$$
where $D_n(A)$ are certain $R/I_nR$\+modules forming a projective
system indexed by the positive integers (which depends functorially
on~$A$).
\end{subl}

\begin{proof}
 This is~\cite[Lemma~7.5]{Pcta}.
\end{proof}

\begin{subl} \label{tuple-derived-projlim-obtainable}
 Let $R$ be a commutative ring and $r_1$,~\dots, $r_m\in R$ be a finite
set of its elements.
 Let $D_1\larrow D_2\larrow D_3\larrow\dotsb$ be a projective system
of $R$\+modules such that $D_n$ is an $R/I_nR$\+module for every
$n\ge1$.
 Then the $R$\+modules \textup{(a)}~$\varprojlim_n D_n$ and
\textup{(b)}~$\varprojlim_n^1D_n$ are simply right obtainable from
$R/(r_1R+\dotsb +r_mR)$\+modules.
\end{subl}

\begin{proof}
 Similar to the proof of Sublemma~\ref{derived-projlim-obtainable}
(cf.~\cite[proof of Theorem~9.5]{Pcta}).
\end{proof}

 Alternatively, one can deduce
Lemma~\ref{tuple-r-contramodule-obtainable-lemma} from
(the proof of) Lemma~\ref{r-contramodule-obtainable} arguing by
induction in the cardinality~$m$ of the set of elements~$\{r_k\}$.
 Inspecting the proof of Lemma~\ref{r-contramodule-obtainable}, one
can observe that all the $R$\+modules involved in obtaining
an $r$\+contramodule $R$\+module $C$ from $R/rR$\+modules according
to this proof depend functorially on~$C$.
 Moreover, these functors are nothing but various compositions of
the basic operations of the passage to the kernel, cokernel, and
infinite product.

 For any second element $s\in R$, these operations preserve
the full subcategory of $s$\+contramodule $R$\+modules $R\modl_{s\ctra}
\subset R\modl$.
 Thus, given an $r$\+contramodule and $s$\+contramodule $R$\+module
$C$, the constructions of Lemma~\ref{r-contramodule-obtainable} show
that $C$ is simply right obtainable from $R/rR$\+modules that
are also $s$\+contramodules.

 Similarly, given an $R$\+module $C$ that is an $r_k$\+contramodule
for every $1\le k\le m$, the constructions of
Lemma~\ref{r-contramodule-obtainable} show that $C$ is simply right
obtainable from $R/r_mR$\+modules that are $r_k$\+contramodules
for every $1\le k\le m$.
 Passing to the ring $R/r_mR$ and proceeding by induction in~$m$,
one proves that $C$ is simply right obtainable from
$R/(r_1R+\dotsb+r_mR)$\+modules.
\end{proof}

\begin{lem} \label{hom-into-contramodule-is-contramodule}
 Let $R$ be a commutative ring and $r\in R$ be an element.
 Let $L^\bu\in\sD^\b(R\modl)$ be a bounded complex of $R$\+modules
and $C$ be an $r$\+contramodule $R$\+mod\-ule.
 Then the $R$\+modules\/ $\Hom_{\sD(R\modl)}(L^\bu,C[i])$, \
$i\in\boZ$ are also $r$\+contramodules.
\end{lem}

\begin{proof}
 This is~\cite[Lemma~6.2(b)]{Pcta}.
\end{proof}

\begin{lem} \label{hom-s-contraadjusted}
 Let $R$ be a commutative ring and $r$, $s\in R$ be two elements.
 Assume that an $R$\+module $C$ is $(rs)$\+contraadjusted.
 Then the $R$\+module\/ $\Hom_k(R[r^{-1}],C)$ is $s$\+contraadjusted.
\end{lem}

\begin{proof}
 For any three $R$\+modules $F$, $G$, and $C$ there is a pair of
spectral sequences
\begin{align*}
 '\!E_2^{pq}=\Ext_R^p(G,\Ext_R^q(F,C)) &\Longrightarrow
 H^{p+q}(\boR\Hom_R(G,\boR\Hom_R(F,C))) \\
 ''\!E_2^{pq}=\Ext_R^p(\Tor^R_q(F,G),C) &\Longrightarrow
 H^{p+q}(\boR\Hom_R(F\ot^\boL_RG,\>C))
\end{align*}
with the differentials $d_r^{pq}\:E_2^{pq}\rarrow E_2^{p+r,\.q-r+1}$,
converging to the same limit $E_\infty^{p+q}=
H^{p+q}(\boR\Hom_R(G,\boR\Hom_R(F,C)))
= H^{p+q}(\boR\Hom_R(F\ot^\boL_RG,\>C))$.
 In low degrees, each of them turns into an exact sequence
$$
 0\lrarrow E_2^{1,0}\lrarrow E_\infty^1\lrarrow E_2^{0,1}
 \lrarrow E_2^{2,0}\lrarrow E_\infty^2.
$$
 In particular, if $\Tor^R_1(F,G)=0$, then $''\!E_2^{0,1}=0$ implies
$''\!E_2^{1,0}=E_\infty^1$, leading to an injective morphism
$$
 \Ext_R^1(F,\Hom_R(G,C))\.=\.{}'\!E_2^{1,0}\lrarrow
 E_\infty^1\.=\.{}''\!E_2^{1,0}\.=\.\Ext_R^1(F\ot_RG,\>C).
$$

 Thus, in particular, for any $R$\+module $C$, there is a natural
injective $R$\+module morphism
$$
 \Ext_R^1(R[s^{-1}],\.\Hom_R(R[r^{-1}],C))\lrarrow
 \Ext_R^1(R[(rs)^{-1}],\.C),
$$
so the left-hand side vanishes whenever the right-hand side does
(cf.~\cite[proof of Lemma~1.2.1]{Pcosh}).
\end{proof}

\begin{proof}[Proof of Main Proposition~\ref{r-very-flat-main-prop}]
 Denote by\/ $\sE\subset R\modl$ the class of all $R_J$\+modules,
where $J$ runs over the subsets of\/ $\{1,\dotsc,m\}$, viewed
as $R$\+modules via the restriction of scalars.
 We have to show that an $R$\+module $C$ is $\r^\times$\+contraadjusted
if and only if it is right $1$\+obtainable from~$\sE$.

 To prove the ``if'', notice that all the $R$\+modules from
the class $\sE$ are $\r^\times$\+contra\-adjusted.
 Indeed, let $I$ and $J$ be two subsets in $\{1,\dotsc,m\}$,
and $E$ be an $R_J$\+module.
 Then we have
$$
 R_J\ot_R R[r_I^{-1}]=\begin{cases}
 R_J, &\text{if $I\subset J$,} \\
 0,   &\text{if $I\not\subset J$.} \end{cases} 
$$
 By Lemma~\ref{change-of-scalars-vanishing-tor-lemma}(a), it follows
that $\Ext_R^i(R[r_I^{-1}],E)=\Ext_{R_J}^i(R_J\ot_RR[r_I^{-1}],\>E)=0$
for $i\ge1$ in both cases.

 In other words, we have $R[r_I^{-1}]\in\sF={}^{\perp_{\ge1}}\.\sE$ for
all $I\subset\{1,\dotsc,m\}$.
 By Lemma~\ref{1-2-right-obtainable-orthogonal}, it follows that
$C\in\sF^{\perp_{\ge1}}\subset\{R[r_I^{-1}]\}^{\perp_{\ge1}}$ for every
$R$\+module $C$ right $1$\+obtainable from $\sE$, so
$\Ext^1_R(R[r_I^{-1}],C)=0$ and $C$ is $\r^\times$\+contraadjusted.

 The nontrivial part is the ``only if''.
 Before spelling out formally the induction procedure constituting
our argument, let explain the key ideas first in some informal terms.
 What we are doing here can be described as ``cutting modules into
pieces''.
 The main proposition claims that any $\r^\times$\+contraadjusted
$R$\+module can be cut into certain kind of pieces of simple nature.
  
 We have a finite collection of elements $r_1$,~\dots, $r_m$, and
so we work with them one by one.
 Using the exact sequence~\eqref{r-contraadjusted-sequence},
an $r_m$\+contraadjusted $R$\+module $C$ can be cut into three pieces
of two sorts (meaning that we need to show two of them to be
$1$\+obtainable from~$\sE$, while the third one can be only
$2$\+obtainable).
 One of these three pieces is an $R[r_m^{-1}]$\+module, while
the other two are $r_m$\+contramodule $R$\+modules.

 As with any such inductive transformation procedure, the trick is
to make sure that we do not lose important properties along the way.
 The original module $C$ is not just $r_m$\+contraadjusted; it is
$\r^\times$\+contraadjusted.
 The pieces into which we cut it should remain
$\r^\times$\+contraadjusted, or something close to that; or otherwise
our induction procedure would not work.
 To be more precise, when we get a piece about which we only need to
show that it is $2$\+obtainable, we are not concerned with its
contraadjustness properties anymore; but to show that a module is
$1$\+obtainable, we need to know that it is $r_J$\+contraadjusted
for various elements $r_J\in\r^\times$.

 The problem is that contraadjustness is a fragile property.
 It is destroyed by passages to the kernels of morphisms.
 Transforming modules using the exact
sequence~\eqref{r-contraadjusted-sequence} does not involve
passages to the kernels, but the constructions of
Sublemmas~\ref{delta-lambda-r-sequence}\+-%
\ref{derived-projlim-obtainable}
or~\ref{delta-lambda-I-sequence}\+-%
\ref{tuple-derived-projlim-obtainable} do.
 So we cannot just cut an $r_m$\+contramodule $R$\+module into
$R/r_mR$\+modules, using constructions of this kind, without possibly
destroying our contraadjustness properties with respect to the other
elements~$r_j$ or~$r_J$.

 The solution implemented in the argument below is to postpone all
the cutting of $r_j$\+contramodules into $R/r_jR$\+modules until
the moment after we have already applied
the exact sequence~\eqref{r-contraadjusted-sequence} along all
the elements $r=r_j$, \,$j=1$,~\dots,~$m$, replacing all
the contraadjustness with either invertibility of~$r_j$, or
$r_j$\+contramoduleness, for every $1\le j\le m$.
 (Or rather, strengthening all the contraadjustness to various
combinations of invertibility and contramoduleness.)
 Invertibility and contramoduleness are robust properties, preserved
by passages to the kernels, cokernels, and infinite products; so
applying the constructions of
Sublemmas~\ref{delta-lambda-I-sequence}\+-%
\ref{tuple-derived-projlim-obtainable} becomes possible at this point
and no vital information is destroyed.

 Here is the promised formal induction procedure.
 We will prove the following assertion by induction on $0\le n\le m$:
any $R$\+module that is $r_J$\+contraadjusted for all
$J\subset\{1,\dotsc,n\}$ and an $r_k$\+contramodule for all
$n+1\le k\le m$, is right $1$\+obtainable from $R_J$\+modules,
where $J\subset\{1,\dotsc,n\}\subset\{1,\dotsc,m\}$.
 For $n=m$, this is the assertion of the main proposition.

 Induction base: for $n=0$, this is the assertion of
Lemma~\ref{tuple-r-contramodule-obtainable-lemma}.

 Let $C$ be an $R$\+module that is $r_J$\+contraadjusted for all
$J\subset\{1,\dotsc,n\}$ and an $r_k$\+contramodule for all
$n+1\le k\le m$.
 Consider the exact sequence~\eqref{r-contraadjusted-sequence} for
$r=r_n$:
\begin{equation} \label{r-n-contraadjusted-sequence}
 0\lrarrow\Hom_R(R[r_n^{-1}]/R,C)\lrarrow\Hom_R(R[r_n^{-1}],C)\lrarrow C
 \\ \lrarrow\Delta_{r_n}(C)\lrarrow0.
\end{equation}
 Denote by $\sE_n\subset R\modl$ the class of all $R_J$\+modules
with $J\subset\{1,\dotsc,n\}$.
 To show that $C$ is right $1$\+obtainable from $\sE_n$, it suffices
to check that the $R$\+modules $\Hom_R(R[r_n^{-1}],C)$ and
$\Delta_{r_n}(C)$ are right $1$\+obtainable from $\sE_n$ and
the $R$\+module $\Hom_R(R[r_n^{-1}]/R,C)$ is right $2$\+obtainable
from~$\sE_n$.
 Let us consider these three $R$\+modules one by one.

 The $R$\+module $\Hom_R(R[r_n^{-1}],C)$ is an $R[r_n^{-1}]$\+module.
 By Lemma~\ref{hom-into-contramodule-is-contramodule}, it is also
an $r_k$\+contramodule for all $n+1\le k\le m$, and by
Lemma~\ref{hom-s-contraadjusted}, it is $r_J$\+contraadjusted for
all $J\subset\{1,\dotsc,n-1\}$.
 In order to conclude that the $R$\+module $\Hom_R(R[r_n^{-1}],C)$
is right $1$\+obtainable from $\sE_n$, it remains to apply
the induction assumption to the pair of integers $(n',m')=
(n-1,m-1)$, the ring $R'=R[r_n^{-1}]$, and the set of elements
$r_1'$,~\dots, $r_{m-1}'\in R'$ equal to the images of
the elements $r_1$,~\dots, $r_{n-1}$, $r_{n+1}$,~\dots, $r_m$
under the localization morphism $R\rarrow R[r_n^{-1}]$.

 The $R$\+module $\Delta_{r_n}(C)$ is $r_J$\+contraadjusted for all
$J\subset\{1,\dotsc,n-1\}$ as a quotient module of
an $r_J$\+contraadjusted $R$\+module~$C$.
 It is also an $r_n$\+contramodule (because the $R$\+module
$\Delta_r(A)$ is an $r$\+contramodule for any $R$\+module~$A$),
and it is an $r_k$\+contramodule for all $n+1\le k\le m$
by Lemma~\ref{hom-into-contramodule-is-contramodule}.
 So $\Delta_{r_n}(C)$ is an $r_k$\+contramodule for all $n\le k\le m$.
 The claim that the $R$\+module $\Delta_{r_n}(C)$ is right
$1$\+obtainable from $\sE_n$ (in fact, even from $\sE_{n-1}$) now
follows from the induction assumption applied to the pair of integers
$(n',m')=(n-1,m)$, the same ring $R'=R$, and the same set of elements
$r'_1=r_1$,~\dots, $r'_m=r_m$.

 The $R$\+module $\Hom_R(R[r_n^{-1}]/R,C)$ is an $r_n$\+contramodule
(in fact, even an $r_n$\+separated $r_n$\+complete $R$\+module) and
an $r_k$\+contramodule for all $n+1\le k\le m$
(by Lemma~\ref{hom-into-contramodule-is-contramodule}).
 By Lemma~\ref{tuple-r-contramodule-obtainable-lemma},
the $R$\+module $\Hom_R(R[r_n^{-1}]/R,C)$ is simply right obtainable
from $R/(r_nR+\dotsb+r_mR)$\+modules.

 Now we apply the induction assumption to the pair of integers
$(n',m')=(n-1,\allowbreak n-1)$, the ring $R'=R/(r_nR+\dotsb+r_mR)$,
and the sequence of elements $r'_1$,~\dots, $r'_{n-1}\in R'$ equal
to the images of the elements $r_1$,~\dots, $r_{n-1}$ under
the surjective ring homomorphism $R\rarrow R'$.
 This is simply the assertion of the main proposition for the ring
$R'$ with $n-1$ elements $r'_1$,~\dots,~$r'_{n-1}$.
 By Lemma~\ref{2-obtainable-from-contraadjusted}, this assertion implies
that every $R'$\+module is right $2$\+obtainable from~$\sE_{n-1}$.
 Thus the $R$\+module $\Hom_R(R[r_n^{-1}]/R,C)$ is right
$2$\+obtainable from~$\sE_{n-1}$.
\end{proof}

\begin{proof}[Proof of Theorem~\ref{r-very-flat-theorem}]
 As in the proof of Toy Main Lemma~\ref{toy-main-lemma} in
Section~\ref{toy-secn}, the ``only if'' assertion holds, because
the functors $F\longmapsto R_J\ot_RF$, \ $J\subset\{1,\dotsc,m\}$
preserve transfinitely iterated extensions, in the sense of
the inductive limit, of flat $R$\+modules, and take the $R$\+modules
$R[r_I^{-1}]$, \ $I\subset\{1,\dotsc,m\}$ to free $R_J$\+modules
(with $1$ or~$0$ generators).
 The ``if'' assertion is the nontrivial part.

 To show that the $R$\+module $F$ is $\r^\times$\+very flat, we need
to check that $\Ext_R^1(F,C)=0$ for all $\r^\times$\+contraadjusted
$R$\+modules~$C$.
 As in the proof of Main Proposition~\ref{r-very-flat-main-prop},
we denote by $\sE\subset R\modl$ the class of all $R_J$\+modules,
$J\subset\{1,\dotsc,m\}$.

 By Lemma~\ref{change-of-scalars-vanishing-tor-lemma}(a), it follows
from the assumption of flatness of the $R$\+module $F$ and
projectivity of the $R_J$\+modules $R_J\ot_RF$ that $\Ext^i_R(F,E)=0$
for all $E\in\sE$ and $i\ge1$.
 So $F\in{}^{\perp_{\ge1}}\.\sE$.
 By Main Proposition~\ref{r-very-flat-main-prop}, the $R$\+module $C$
is right $1$\+obtainable from~$\sE$.
 Applying Lemma~\ref{1-2-right-obtainable-orthogonal}, we can conclude
that $C\in({}^{\perp_{\ge1}}\.\sE)^{\perp_{\ge1}}\subset
\{F\}^{\perp_{\ge1}}$, hence $\Ext_R^1(F,C)=0$.
\end{proof}

\begin{proof}[Proof of Main Lemma~\ref{fvf-main-lemma}]
 The ``only if'' assertion holds, because finite very flatness is
preserved by the extensions of scalars
(see Lemma~\ref{very-flatness-extension-of-scalars}(b)).
 The ``if'' assertion is the nontrivial part.

 Let us recall the discussion from Section~\ref{good-subsets-outline},
with a slight change in notation.
 We have a finite set of elements $\s=\{s_1,\dotsc,s_p\}$ in
the quotient ring $R/rR$ such that the $R/rR$\+module $F/rF$
is $\s^\times$\+very flat, and a finite set $\t=\{t_1,\dotsc,t_q\}$
of elements in the ring of fractions $R[r^{-1}]$ such that
the $R[r^{-1}]$\+module $F[r^{-1}]$ is $\t^\times$\+very flat.
 Lift the elements $s_i\in R/rR$ to some elements $\tilde s_i\in R$,
and choose elements $\tilde t_l\in R$ such that $t_l=\tilde t_l/r^{n_l}$
in $R[r^{-1}]$ for some exponents $n_l\ge0$.

 Set $r_0=r$, \ $r_i=\tilde s_i$ for $1\le i\le p$, and $r_{p+l}=
\tilde t_l$ for $1\le l\le q$.
 Denote the set of all elements $r_0$,~\dots, $r_{p+q}\in R$ by~$\r$.
 We claim that the $R$\+module $F$ is $\r^\times$\+very flat.

 According to Theorem~\ref{r-very-flat-theorem}, it suffices to check
that the $R_J$\+module $R_J\ot_RF$ is projective for all subsets
$J\subset\{0,\dotsc,p+q\}$.
 Let us consider two cases separately.

 If $0\notin J$, then the ring homomorphism $R\rarrow R_J$ factorizes
as $R\rarrow R/rR\rarrow R_J$.
 Denote by $J'$ the set $J\cap\{1,\dotsc,p\}$.
 Then, moreover, the ring homomorphism $R/rR\rarrow R_J$ factorizes
as $R/rR\rarrow (R/rR)_{J'}\rarrow R_J$, where the notation is
$K'=\{1,\dotsc,p\}\setminus J'$, \ $s_{J'}=\prod_{j\in J'}s_j$, and
$(R/rR)_{J'}=((R/rR)/\sum_{k\in K'}s_k(R/rR))[s_{J'}^{-1}]$.

 Since the $R/rR$\+module $F/rF$ is $\s^\times$\+very flat,
by the ``only if'' assertion of Theorem~\ref{r-very-flat-theorem}
the $(R/rR)_{J'}$\+module $(R/rR)_{J'}\ot_RF=(R/rR)_{J'}\ot_{R/rR}F/rF$
is projective, and it follows that the $R_J$\+module $R_J\ot_RF$
is also projective.

 If $0\in J$, then the ring homomorphism $R\rarrow R_J$ factorizes
as $R\rarrow R[r^{-1}]\rarrow R_J$.
 Denote by $J''\subset\{1,\dotsc,q\}$ the set of all integers
$j-p$, where $j\in J\cap\{p+1,\allowbreak\dotsc,p+q\}$.
 Then, moreover, the ring homomorphism $R[r^{-1}]\rarrow R_J$
factorizes as $R[r^{-1}]\rarrow R[r^{-1}]_{J''}\rarrow R_J$, where
the notation is $K''=\{1,\dotsc,q\}\setminus J''$, \ 
$t_{J''}=\prod_{j\in J''}t_j$, and $R[r^{-1}]_{J''}=
(R[r^{-1}]/\sum_{k\in K''}t_kR[r^{-1}])[t_{J''}^{-1}]$.

 Since the $R[r^{-1}]$\+module $F[r^{-1}]$ is $\t^\times$\+very flat,
by the ``only if'' assertion of Theorem~\ref{r-very-flat-theorem}
the $R[r^{-1}]$\+module $R[r^{-1}]_{J''}\ot_RF=R[r^{-1}]_{J''}
\ot_{R[r^{-1}]}F[r^{-1}]$ is projective, and it follows that
the $R_J$\+module $R_J\ot_RF$ is also projective.
\end{proof}

 This finishes the proof of Main Theorem~\ref{fvf-module-main-theorem},
and consequently also of Main Theorems~\ref{general-module-main-theorem}
and~\ref{general-algebra-main-theorem}.

\Section{Examples and Applications}

 Let $R$ be a commutative ring.
 An $R$\+algebra is said to be \emph{flat} if it is a flat $R$\+module.
 A commutative $R$\+algebra $S$ is called \emph{very flat} if
the $R$\+module $S[s^{-1}]$ is very flat for every element $s\in S$
(see the discussion in Section~\ref{very-flat-morphisms-introd}).

\begin{cor} \label{very-flat-morphism-main-cor}
 Let $R$ be a commutative ring.
 Then any finitely presented, flat $R$\+algebra is very flat.
\end{cor}

\begin{proof}
 Let $S$ be a finitely presented flat $R$\+algebra and $s\in S$
be an element.
 Then the $R$\+module $S[s^{-1}]$ is flat, because it is a directed
inductive limit of copies of the free $S$\+module~$S$, which is
a flat $R$\+module by assumption.
 Furthermore, the $R$\+algebra $S[s^{-1}]$ is finitely presented,
since the $R$\+algebra $S$~is.
 Indeed, if $S$ is the quotient algebra of the algebra of polynomials
$R[x_1,\dotsc,x_m]$ by a finitely generated ideal~$I$, then
$S[s^{-1}]$ is the quotient algebra of $R[x_1,\dotsc,x_m,z]$ by
the finitely generated ideal~$J$ whose generators are the generators
of the ideal~$I$ and the element $zq(x_1,\dotsc,x_m)-\nobreak1$, where
$q\in R[x_1,\dotsc,x_m]$ is a preimage of the element $s\in S$.
 Applying Main Theorem~\ref{general-algebra-main-theorem}, we conclude
that $S[s^{-1}]$ is a very flat $R$\+module.
\end{proof}

 Let $R$ be a commutative ring.
 A commutative $R$\+algebra $S$ is said to be \emph{finitely very flat}
if the $R$\+module $S[s^{-1}]$ is finitely very flat for every
element $s\in S$.
 The following corollary is a stronger version of
Corollary~\ref{very-flat-morphism-main-cor}.

\begin{cor} \label{fvf-morphism-main-cor}
 Let $R$ be a commutative ring.
 Then any finitely presented, flat $R$\+algebra is finitely very flat.
\end{cor}

\begin{proof}
 Let $S$ be a finitely presented flat $R$\+algebra and $s\in S$ be
an element.
 As it was explained in the proof of
Corollary~\ref{very-flat-morphism-main-cor}, \,$S[s^{-1}]$ is
a finitely presented flat $R$\+algebra, too.
 Applying Main Theorem~\ref{fvf-module-main-theorem} to the $R$\+algebra
$S[s^{-1}]$ and the $S[s^{-1}]$\+module $F=S[s^{-1}]$, we conclude that
the $R$\+module $S[s^{-1}]$ is finitely very flat.
\end{proof}

 The next lemma explains the importance of (finitely) very flat
morphisms of commutative rings.

\begin{lem} \label{very-flat-restriction-of-scalars-lemma}
 Let $R$ be a commutative ring and $S$ be a commutative $R$\+algebra.
 Then \par
\textup{(a)} the $R$\+algebra $S$ is very flat if and only if every
very flat $S$\+module is a very flat $R$\+module; \par
\textup{(b)} the $R$\+algebra $S$ is finitely very flat if and only if
every finitely very flat $S$\+module is a finitely very flat
$R$\+module.
\end{lem}

\begin{proof}
 Part~(a): to prove the ``if'', it suffices to consider the very flat
$S$\+module $S[s^{-1}]$.
 The assertion ``only if'' is~\cite[Lemma~1.2.3(b)]{Pcosh}.
 Part~(b): once again, to prove the ``if'', it suffices to consider
the finitely very flat $S$\+module $S[s^{-1}]$.
 ``Only if'': let $F$ be a finitely very flat $S$\+module, and let
$\s=\{s_1,\dotsc,s_m\}$ be a finite set of elements in $S$ such
that the $S$\+module $F$ is $\s$\+very flat.
 Set $s_0=1\in S$, and for every $0\le j\le m$ choose a finite
set of elements $\r_j=\{r_{j,1},\dotsc,r_{j,n_j}\}\subset R$ such that
the $R$\+module $S[s_j^{-1}]$ is $\r_j$\+very flat.
 Denote by $\r\subset R$ the finite set of all elements $\{r_{j,k}\}$,
\,$0\le j\le m$, \,$1\le k\le n_j$.
 Then the $R$\+module $F$ is $\r$\+very flat.
\end{proof}

\begin{ex}
 This example, suggested to us by Jan Trlifaj, shows that the finite
presentability condition on the $R$\+algebra $S$ or the $S$\+module
$F$ cannot be replaced by a finite generatedness condition in
the formulations of Main Theorems~\ref{general-algebra-main-theorem}
and~\ref{general-module-main-theorem}.

 Let $R$ be a von Neumann regular commutative ring.
 Then all $R$\+modules are flat, while all very flat $R$\+modules
are projective~\cite[Example~2.9]{ST}.
 Let $I\subset R$ be an infinitely generated ideal.
 Then $S=R/I$ is a finitely generated flat $R$\+algebra which is
\emph{not} a very flat $R$\+module.
 Alternatively, set $S=R$ and $F=R/I$; then $S$ is a finitely
presented $R$\+algebra and $F$ is a finitely generated $S$\+module
which is a flat, but \emph{not} a very flat $R$\+module.

 Main Theorem~\ref{general-module-main-theorem} tells, on the other
hand, that if $R$ is a von Neumann regular commutative ring, $S$
is a finitely presented $R$\+algebra, and $F$ is a finitely
presented $S$\+module, then $F$ is a projective $R$\+module.
\end{ex}

 The following corollary is an application of
Main Theorem~\ref{general-module-main-theorem}
or~\ref{fvf-module-main-theorem}.

\begin{cor} \label{inverting-operator-cor}
 Let $R$ be a commutative ring, $P$ be a finitely generated projective
$R$\+module, and $x\:P\rarrow P$ be an $R$\+linear map.
 Denote by $P[x^{-1}]$ the inductive limit of the sequence of maps
\begin{equation}
 P\overset{x}\lrarrow P\overset{x}\lrarrow P\overset{x}\lrarrow\dotsb
\end{equation}
 Then $P[x^{-1}]$ is a finitely very flat $R$\+module.
\end{cor}

\begin{proof}
 Consider the algebra of polynomials $R[x]$ in one variable~$x$ over
the ring~$R$.
 Then $P$ can be viewed as an $R[x]$\+module in which the element
$x\in R[x]$ acts by the operator $x\:P\rarrow P$.
 Notice that $P$ is a finitely presented $R[x]$\+module.

 Indeed, let $p_1$,~\dots, $p_m\in P$ be a set of generators of
the $R$\+module~$P$.
 A finitely generated projective $R$\+module is finitely presentable;
specifically, if $f\:R^m\rarrow P$ is the surjective map taking
the free generator $e_i\in R^m$ to the generator $p_i\in P$
and $g\:P\rarrow R^m$ is a splitting of the $R$\+module morphism~$f$,
then $P$ is the cokernel of the $R$\+module morphism $\id_{R^m}-gf\:
R^m\rarrow R^m$.
 So the elements $e_i-gf(e_i)\in R^m$ form a finite set of relations
defining the $R$\+module $P$ as a quotient $R$\+module of the free
$R$\+module~$R^m$.
 Now let $x_{ij}\in R$, \ $1\le i$, $j\le m$ be some elements such
that $x(p_i)=\sum_{j=1}^mx_{ij}p_j$ for all $1\le i\le m$.
 Then the elements $e_i-gf(e_i)\in R^m\subset R[x]^m$ and
$xe_i-\sum_{j=1}^mx_{ij}e_j\in R[x]^m$ can be used as a finite set
of relations defining the $R[x]$\+module $P$ as a quotient module
of the free $R[x]$\+module~$R[x]^m$.

 Denoting by $R[x,x^{-1}]$ the algebra of Laurent polynomials in
the variable~$x$ over the ring $R$, it follows that $P[x^{-1}]\simeq
R[x,x^{-1}]\ot_{R[x]}P$ is a finitely presented $R[x,x^{-1}]$\+module.
 Clearly, $P[x^{-1}]$ is also a flat $R$\+module, while $R[x,x^{-1}]$
is a finitely presented $R$\+algebra.
 Applying Main Theorem~\ref{general-module-main-theorem} to
the $R$\+algebra $S=R[x,x^{-1}]$ and the $S$\+module $F=P[x^{-1}]$,
we conclude that $P[x^{-1}]$ is a very flat $R$\+module.
 Applying Main Theorem~\ref{fvf-module-main-theorem}, we see
that $P[x^{-1}]$ is even a finitely very flat $R$\+module.

 The above discussion of finite presentability of a finitely generated
projective module could be avoided by replacing a projective
$R$\+module $P$ with an $R$\+linear map $x\:P\rarrow P$ by a direct
sum $(P,x)\oplus (Q,y)$, where $Q$ is a finitely generated projective
$R$\+module such that the $R$\+module $P\oplus Q$ is free, and
$y\:Q\rarrow Q$ is an arbitrary $R$\+linear map.
 Then the $R$\+module $P[x^{-1}]$ is a direct summand of
$(P\oplus Q)[(x\oplus y)^{-1}]$, and the passage to a direct summand
preserves (finite) very flatness.
 Taking $y=0$, one even gets $(P\oplus Q)[(x\oplus y)^{-1}]=P[x^{-1}]$.
 Denote the finitely generated free $R$\+module $P\oplus Q$ by $R^m$
and the $R$\+linear operator $x\oplus y$ by $z\:R^m\rarrow R^m$.
 Then one can define the $R[x]$\+module $R^m$, with the element
$x\in R[x]$ acting in it by the operator~$z$, by the finite set of
relations $xe_i-\sum_{z_{ij}e_j}z_{ij}e_j\in R[x]^m$, where $e_i\in R^m$, \
$1\le i\le m$ are the free generators of the $R$\+module $R^m$ and
$z_{ij}\in R$, \ $1\le i$, $j\le m$ are the elements such that
$z(e_i)=\sum_{j=1}^mz_{ij}e_j$ for all $1\le i\le m$.

 Alternatively, one could say that there is a subring $\oR\subset R$
finitely generated over $\boZ$ and a finitely generated projective
$\oR$\+module $\oP$ with an $\oR$\+linear map $\bar x\:\oP\rarrow\oP$
such that $P=R\ot_\oR\oP$ and $x=R\ot_\oR\bar x$ (because only
a finite set of elements of $R$ are involved as the matrix entries
of an idempotent linear operator $R^m\rarrow R^m$ defining $P$
and the operator~$x\:P\rarrow P$).
 Then $\oP$ is a finitely generated module over a Noetherian ring
$\oR[x]$ and $\oP[\bar x^{-1}]$ is a finitely generated/presented
module over $\oR[x,x^{-1}]$, so one can apply
Main Theorem~\ref{fvf-module-main-theorem} to the $\oR$\+algebra
$\oR[x,x^{-1}]$ and the $\oR[x,x^{-1}]$\+module $\oP[\bar x^{-1}]$
in order to show that the $\oR$\+module $\oP[\bar x^{-1}]$ is
finitely very flat.
 Hence the $R$\+module $P[x^{-1}]=R\ot_\oR\oP[\bar x^{-1}]$ is finitely
very flat.
\end{proof}

\begin{ex} \label{inverting-operator-counterex}
 The following counterexample shows that the condition that
the projective $R$\+module $P$ should be finitely generated cannot
be dropped in Corollary~\ref{inverting-operator-cor}.
 Let $R=\boZ$ be the ring of integers and $P$ be a free abelian group
with an infinite set of generators $p_1$, $p_2$, $p_3$,~\dots{}
 Let the additive operator $x\:P\rarrow P$ act by the rule
$x(p_n)=(n+1)p_{n+1}$, \,$n\ge1$.
 Then the abelian group $P[x^{-1}]$ is a vector space of countable
dimension over the field of rational numbers~$\boQ$, which is
\emph{not} a very flat $\boZ$\+module~\cite[Example~1.7.7]{Pcosh}.
\end{ex}

\begin{ex} \label{very-flat-algebra-counterex}
 This is a counterexample showing that a commutative $R$\+algebra
$S$ can be a free $R$\+module without being a very flat
$R$\+algebra (cf.\ Section~\ref{very-flat-morphisms-introd}).
 Once again, set $R=\boZ$, and let $S$ be the ``trivial ring
extension'' $S=\boZ[x]\oplus P$, where $P$ is the abelian group from
Example~\ref{inverting-operator-counterex}.
 The multiplication in $S$ is defined by the rule $(f,p)(g,q)=
(fg,\>g(p)+f(q))$, where $f$, $g\in\boZ[x]$, \ $p$, $q\in P$, and
the $\boZ[x]$\+module structure on $P$ is given by the operator
$x\:P\rarrow P$ from Example~\ref{inverting-operator-counterex}.
 This formula means that, by the definition, the product of any two
elements of $P$ is zero in~$S$.
 So $S$ is a free abelian group and a commutative $\boZ$\+algebra
with an element $x\in S$, but $S[x^{-1}]=\boZ[x,x^{-1}]\oplus P[x^{-1}]$
is not a very flat $\boZ$\+module.
\end{ex}

 Finally, let us prove the last property~(VF16) announced in
Section~\ref{VF-new-properties-introd}.

\begin{cor} \label{universally-very-flat-cor}
 Let $k$~be a field, $R$ and $S$ be commutative $k$\+algebras, and
$t\in R\ot_kS$ be an element.
 Then the $R$\+module $(R\ot_kS)[t^{-1}]$ is very flat.
\end{cor}

\begin{proof}
 Notice that the assertion of the corollary can be generalized as
follows.
 Let $M$ be an $S$\+module.
 Then $R\ot_kM$ is an $(R\ot_kS)$\+module, and one can consider
the $(R\ot_kS)[t^{-1}]$\+module $(R\ot_kM)[t^{-1}]=
(R\ot_kS)[t^{-1}]\ot_{R\ot_kS}(R\ot_kM)$.
 We claim that $(R\ot_kM)[t^{-1}]$ is a very flat $R$\+module.

 Let $r_1$,~\dots, $r_m\in R$ and $s_1$,~\dots, $s_m\in S$ be some
elements such that $t=r_1\ot s_1+\dotsb+r_m\ot s_m\in R\ot_kS$.
 Denote by $\oS$ the $k$\+subalgebra generated by $s_1$,~\dots,
$s_m$ in~$S$.
 Then $t\in R\ot_k\oS\subset R\ot_kS$, and the $R$\+module
$(R\ot_kM)[t^{-1}]$ only depends on the $\oS$\+module structure on $M$,
rather than on the whole $S$\+module structure.
 Replacing $S$ with $\oS$, for any $\oS$\+module $M$ we can consider
the $(R\ot_k\oS)[t^{-1}]$\+module $(R\ot_k M)[t^{-1}]$.
 Let us show that $(R\ot_k M)[t^{-1}]$ is a very flat $R$\+module
for any $\oS$\+module~$M$.

 Any module over a commutative ring $\oS$ is a transfinitely iterated
extension, in the sense of the inductive limit, of the $\oS$\+modules
$\oS/I$, where $I\subset R$ are ideals.
 The functor $M\longmapsto (R\ot_kM)[t^{-1}]$ is exact and preserves
inductive limits, so it preserves transfinitely iterated extensions.
 Thus the $(R\ot_k\oS)[t^{-1}]$\+module $(R\ot_k M)[t^{-1}]$ is
a transfinitely iterated extension of the $(R\ot_k\oS)[t^{-1}]$\+modules
$(R\ot_k\oS/I)[t^{-1}]$.

 It remains to show that the $R$\+module $(R\ot_k\oS/I)[t^{-1}]$ is
very flat for every ideal $I\subset\oS$.
 Now $\oS/I$ is a finitely generated $k$\+algebra, hence
$R\ot_k\oS/I$ is a finitely presented $R$\+algebra.
 Furthermore, $R\ot_k\oS/I$ is a free $R$\+module.
 Applying Corollary~\ref{very-flat-morphism-main-cor}, we conclude
that $R\ot_k\oS/I$ is a very flat $R$\+algebra, so the $R$\+module
$(R\ot_k\oS/I)[t^{-1}]$ is very flat.
 (In fact, the $R$\+module $(R\ot_k\oS/I)[t^{-1}]$ is even finitely
very flat, by Corollary~\ref{fvf-morphism-main-cor}, but finite very
flatness as such is not preserved by transfinitely iterated extensions,
so the $R$\+module $(R\ot_kS)[t^{-1}]$ is still only very flat.)
\end{proof}

\Section{Descent for Surjective Morphisms}  \label{descent-secn}

 The proofs of Theorems~\ref{noetherian-surj-descent-thm}
and~\ref{fvf-surj-descent-thm} are similar to (but slightly more
complicated than) the argument deducing the main theorems from
the main lemmas in Section~\ref{implies-main-theorem-secn}.

 The most important technical property of the class of all morphisms
of commutative rings inducing surjective maps of the spectra is its
stability under the base change.
 Let us start with the following trivial observation: the spectrum
of a commutative ring $R$ is empty if and only if $R=0$.

\begin{lem} \label{surjectivity-base-change}
 Let $R\rarrow S$ be a homomorphism of commutative rings such that
the induced map of the spectra\/ $\Spec S\rarrow\Spec R$ is surjective.
 Let $R\rarrow R'$ be an arbitrary morphism of commutative rings.
 Then the map\/ $\Spec(R'\ot_RS)\rarrow\Spec R'$ induced by
the homomorphism of commutative rings $R'\rarrow R'\ot_RS$
is surjective.
\end{lem}

\begin{proof}
 Let $\p'$ be a prime ideal in the ring $R'$ and $\p$ be its image
under the map $\Spec R'\rarrow\Spec R$.
 By assumption, there exists a prime ideal $\q$ in the ring $S$ whose
image under the map $\Spec S\rarrow \Spec R$ is equal to~$\p$.
 Let $k_{\p'}(R')$, \,$k_\p(R)$, and $k_\q(S)$ denote the residue fields
of the prime ideals $\p'\subset R'$, \ $\p\subset R$, and $\q\subset S$.
 Then the ring $k_{\p'}(R')\ot_{k_\p(R)}k_\q(S)$ is nonzero (as the tensor
product of nonzero vector spaces over a field), so it has a maximal
ideal $\m\subset k_{\p'}(R')\ot_{k_\p(R)}k_\q(S)$.
 There is a natural ring homomorphism $R'\ot_RS\rarrow
k_{\p'}(R')\ot_{k_\p(R)}k_\q(S)$; denote the full preimage of~$\m$ under
this ring homomorphism by $\q'\subset R'\ot_RS$.
 Then $\q'\in\Spec(R'\ot_RS)$ is a point whose image under the map
$\Spec(R'\ot_RS)\rarrow\Spec R'$ is equal to~$\p'$.
\end{proof}

 Let $R$ be a commutative ring.
 A projective $R$\+module $P$ is said to be \emph{faithfully projective}
if it is a projective generator of the abelian category $R\modl$, that
is $\Hom_R(P,M)=0$ implies $M=0$ for any $R$\+module~$M$.
 Equivalently, it means that the free $R$\+module $R$ is a direct
summand of a (finite) direct sum of copies of~$P$.

\begin{lem} \label{general-faithfully-projective-descent}
 Let $R$ be a commutative ring and $S$ be a commutative $R$\+algebra
such that $S$ is a faithfully projective $R$\+module.
 Let $F$ be an $R$\+module. \par
\textup{(a)} Assume that the $R$\+algebra $S$ is very flat.
 Then the $R$\+module $F$ is very flat if and only if
the $S$\+module $S\ot_RF$ is very flat. \par
\textup{(b)} Assume that the $R$\+algebra $S$ is finitely very flat.
 Then the $R$\+module $F$ is finitely very flat if and only if
the $S$\+module $S\ot_RF$ is finitely very flat.
\end{lem}

\begin{proof}
 This is a version of~\cite[Lemma~1.7.12]{Pcosh}.
 Part~(a): the ``only if'' holds by
Lemma~\ref{very-flatness-extension-of-scalars}(a).
 To prove the ``if'', recall that all the very flat $S$\+modules are
very flat $R$\+modules by
Lemma~\ref{very-flat-restriction-of-scalars-lemma}(a).
 In particular, the very flat $S$\+module $S\ot_RF$ is a very
flat $R$\+module.
 It remains to use the assumption that the $R$\+module $R$ is a direct
summand of a finite direct sum of copies of the $R$\+module $S$, so
the $R$\+module $F$ is a direct summand of a finite direct sum
of copies of the $R$\+module $S\ot_RF$.

 Part~(b): the ``only if'' holds by
Lemma~\ref{very-flatness-extension-of-scalars}(b).
 To prove the ``if'', recall that  all the finitely very flat
$S$\+modules are finitely very flat $R$\+modules by
Lemma~\ref{very-flat-restriction-of-scalars-lemma}(b).
 In particular, the finitely very flat $S$\+module $S\ot_RF$ is
a finitely very flat $R$\+module.
 The argument finishes in the same way as the proof of part~(a).
\end{proof}

\begin{lem} \label{fin-pres-faithfully-projective-descent}
 Let $R$ be a commutative ring and $S$ be a finitely presented
commutative $R$\+algebra such that $S$ is a faithfully projective
$R$\+module.
 Let $F$ be an $R$\+module.
 Then \par
\textup{(a)} the $R$\+module $F$ is very flat if and only if
the $S$\+module $S\ot_RF$ is very flat; \par
\textup{(b)} the $R$\+module $F$ is finitely very flat if and only if
the $S$\+module $S\ot_RF$ is finitely very flat.
\end{lem}

\begin{proof}
 Part~(a): the $R$\+algebra $S$ is very flat
by Corollary~\ref{very-flat-morphism-main-cor}, so
Lemma~\ref{general-faithfully-projective-descent}(a) can be applied.
 Part~(b): the $R$\+algebra $S$ is finitely very flat
by Corollary~\ref{fvf-morphism-main-cor}, so
Lemma~\ref{general-faithfully-projective-descent}(b) can be applied.
\end{proof}

 The following proposition is a particular case of (the ``if''
assertion of) Theorem~\ref{noetherian-surj-descent-thm} from which
the general case is deduced.

\begin{prop} \label{noetherian-surj-descent-proof-prop}
 Let $R$ be a Noetherian commutative ring and $S$ be a finitely
generated commutative $R$\+algebra such that the induced map of
the spectra\/ $\Spec S\rarrow\Spec R$ is surjective.
 Let $F$ be a flat $R$\+module.
 Assume that the $S$\+module $S\ot_RF$ is very flat, and
the $R/rR$\+module $F/rF$ is very flat for every nonzero element
$r\in R$.
 Then $F$ is a very flat $R$\+module.
\end{prop}

\begin{proof}
 Case~I: the ring $R$ contains zero-divisors.
 Let $a\ne0\ne b$ be two elements in $R$ such that $ab=0$ in~$R$.
 By assumption, $F/aF$ is a very flat $R/aR$\+module and $F/bF$ is
a very flat $R/bR$\+module.
 As in the proof of
Proposition~\ref{noetherian-main-theorem-proof-prop}, it follows
from the latter that $F[a^{-1}]$ is a very flat $R[a^{-1}]$\+module.
 By Main Lemma~\ref{noetherian-main-lemma}, we conclude that $F$ is
a very flat $R$\+module.

 Case~II: the ring $R$ is an integral domain.
 By Lemma~\ref{generic-freeness-lemma}, there exists a nonzero element
$a\in R$ such that the $R[a^{-1}]$\+module $S[a^{-1}]$ is free.
 By Lemma~\ref{surjectivity-base-change}, the map of spectra
$\Spec S[a^{-1}]\rarrow\Spec R[a^{-1}]$ is surjective, so the ring
$S[a^{-1}]$ is nonzero whenever the ring $R[a^{-1}]$ is nonzero.
 Hence $S[a^{-1}]$ is a faithfully projective $R[a^{-1}]$\+module.
 The $S[a^{-1}]$\+module $S[a^{-1}]\ot_RF$ is very flat, since
the $S$\+module $S\ot_RF$ is very flat.
 Applying Lemma~\ref{fin-pres-faithfully-projective-descent}(a) to
the ring $R[a^{-1}]$, the algebra $S[a^{-1}]$ over it, and
the module $F[a^{-1}]$, we conclude that the $R[a^{-1}]$\+module
$F[a^{-1}]$ is very flat.

 By assumption, the $R/aR$\+module $F/aF$ is very flat.
 By Main Lemma~\ref{noetherian-main-lemma}, it follows that $F$ is
a very flat $R$\+module.
\end{proof}

\begin{proof}[Proof of Theorem~\ref{noetherian-surj-descent-thm}]
 The assertion ``only if'' is a particular case of
Lemma~\ref{very-flatness-extension-of-scalars}(a).
 The assertion ``if'' is the nontrivial part.

 Assume that the $R$\+module $F$ is not very flat.
 By Proposition~\ref{noetherian-surj-descent-proof-prop}, it follows
that there exists a nonzero element $r\in R$ such that
the $R/rR$\+module $F/rF$ is not very flat.
 Now $R_1=R/rR$ is a Noetherian commutative ring, $S_1=S/rS$ is
a finitely generated commutative $R_1$\+algebra, and the map
$\Spec S_1\rarrow\Spec R_1$ is surjective by
Lemma~\ref{surjectivity-base-change}.
 Furthermore, $F_1=F/rF$ is a flat $R/rR$\+module.

 So Proposition~\ref{noetherian-surj-descent-proof-prop} can be
applied again in order to produce a nonzero element $r_1\in R_1$
such that the $R_1/r_1R_1$\+module $F_1/r_1F_1$ is not very flat, etc.
 The argument finishes in the same way as the proof of
Main Theorem~\ref{noetherian-module-main-theorem}
in Section~\ref{implies-main-theorem-secn}.
\end{proof}

 Similarly, the next proposition is a particular case of
Theorem~\ref{fvf-surj-descent-thm} from which the general case is
deduced.

\begin{prop} \label{fvf-surj-descent-proof-prop}
 Let $\oR$ be a Noetherian commutative ring and $\oS$ be a finitely
generated commutative $\oR$\+algebra.
 Let $R$ be a commutative $\oR$\+algebra; set $S=R\ot_\oR\nobreak\oS$
and assume that the induced map of spectra\/ $\Spec S\rarrow \Spec R$
is surjective.
 Let $F$ be a flat $R$\+module.
 Assume that the $S$\+module $S\ot_RF$ is finitely very flat, and
the $R/rR$\+module $F/rF$ is finitely very flat for every nonzero
element $r\in\oR$.
 Then $F$ is a finitely very flat $R$\+module.
\end{prop}

\begin{proof}
 Case~I: the ring $\oR$ contains zero-divisors.
 The argument in this case is similar to the proofs of Case~I
in Propositions~\ref{fvf-main-theorem-proof-prop}
and~\ref{noetherian-surj-descent-proof-prop}.
 It does not use the rings $\oS$ and $S$ and the assumptions
related to them (but uses
Lemma~\ref{very-flatness-extension-of-scalars}(b)
and Main Lemma~\ref{fvf-main-lemma}).

 Case~II: the ring $\oR$ is an integral domain.
 By Lemma~\ref{generic-freeness-lemma}, there exists a nonzero
element $a\in\oR$ such that the $\oR[a^{-1}]$\+module $\oS$ is free.
 Then the $R[a^{-1}]$\+module $S[a^{-1}]$ is free, too.
 Arguing as in the proof of
Proposition~\ref{noetherian-surj-descent-proof-prop}, we see that
$S[a^{-1}]$ is a faithfully projective $R[a^{-1}]$\+module.
 The $S[a^{-1}]$\+module $S[a^{-1}]\ot_RF$ is finitely very flat,
since the $S$\+module $S\ot_RF$ is finitely very flat.
 Applying Lemma~\ref{fin-pres-faithfully-projective-descent}(b), we
conclude that the $R[a^{-1}]$\+module $F[a^{-1}]$ is finitely very flat;
and the argument finishes as in the proof of
Proposition~\ref{noetherian-surj-descent-proof-prop},
using Main Lemma~\ref{fvf-main-lemma}.
\end{proof}

\begin{proof}[Proof of Theorem~\ref{fvf-surj-descent-thm}]
 The assertion ``only if'' is a particular case of
Lemma~\ref{very-flatness-extension-of-scalars}(b).
 ``If'' is the nontrivial part.
 First of all, we apply Lemma~\ref{finitely-generated-ring-lemma}
in order to produce a Noetherian commutative ring $\oR$ with
a ring homomorphism $\oR\rarrow R$ and a finitely generated
commutative $\oR$\+algebra $\oS$ such that $S=R\ot_\oR\oS$.

 Assume that the $R$\+module $F$ is not finitely very flat.
 By Proposition~\ref{fvf-surj-descent-proof-prop}, it then follows
that there exists a nonzero element $r\in\oR$ such that
the $R/rR$\+module $F/rF$ is not finitely very flat.
 The argument finishes in the way similar to the proof of
Main Theorem~\ref{fvf-module-main-theorem} in
Section~\ref{implies-main-theorem-secn} and the above proof
of Theorem~\ref{noetherian-surj-descent-thm}.
\end{proof}

 Let $R$ be a commutative ring.
 A flat $R$\+module $F$ is said to be \emph{faithfully flat} if
$F\ot_RN=0$ implies $N=0$ for any $R$\+module~$N$.
 The \emph{P\+support} $\PSupp_RM\subset\Spec R$ of an $R$\+module $M$
is the set of all prime ideals $\p\subset R$ such that
$k_\p(R)\ot_RM\ne0$, where $k_\p(R)$ denotes the residue field of
the prime ideal~$\p$ in~$R$ (cf.\ Section~\ref{very-flat-qcoh-introd},
\cite[Section~1.7]{Pcosh}, and~\cite[Definition~2.7]{ST}).
 A flat $R$\+module $F$ is faithfully flat if and only if its
P\+support $\PSupp_RF$ coincides with the whole of $\Spec R$.

 An $R$\+algebra is said to be \emph{faithfully flat} if it is
faithfully flat as an $R$\+module.
 A flat commutative $R$\+algebra $S$ is faithfully flat if and only if
the induced map of the spectra $\Spec S\rarrow\Spec R$ is surjective.

\begin{lem} \label{flat-descent-of-flatness}
 Let $R$ be a commutative ring and $S$ be a faithfully flat
commutative $R$\+algebra.
 Then an $R$\+module $F$ is flat if and only if the $S$\+module
$S\ot_RF$ is flat.
\end{lem}

\begin{proof}
 The ``only if'' holds because flatness of modules is preserved by
extensions of scalars.
 To prove the ``if'', notice the isomorphisms
$S\ot_R\Tor_i^R(F,N)\simeq\Tor_i^R(S\ot_RF,\>N)\simeq
\Tor_i^S(S\ot_RF,\>S\ot_RN)$, which hold for any flat $R$\+algebra $S$,
any $R$\+modules $F$ and~$N$, and all $i\ge0$.
 In particular, if $S$ is a faithfully flat $R$\+algebra, then
$\Tor_1^S(S\ot_RF,\>S\ot_RN)=0$ implies $\Tor_1^R(F,N)=0$.
\end{proof}

\begin{cor} \label{flat-descent-of-very-flatness}
 Let $R$ be a Noetherian commutative ring and $S$ be a finitely
generated faithfully flat commutative $R$\+algebra.
 Then an $R$\+module $F$ is very flat if and only if the $S$\+module
$S\ot_RF$ is very flat.
\end{cor}

\begin{proof}
 The ``only if'' holds by
Lemma~\ref{very-flatness-extension-of-scalars}(a).
 To prove the ``if'', assume that $F$ is an $R$\+module such that
the $S$\+module $S\ot_RF$ is very flat.
 Then, in particular, the $S$\+module $S\ot_RF$ is flat, hence by
Lemma~\ref{flat-descent-of-flatness} the $R$\+module $F$ is flat.
 Now we can apply Theorem~\ref{noetherian-surj-descent-thm}
in order to conclude that the $R$\+module $F$ is very flat.
\end{proof}

\begin{cor} \label{flat-descent-of-finite-very-flatness}
 Let $R$ be a commutative ring and $S$ be a finitely presented
faithfully flat commutative $R$\+algebra.
 Then an $R$\+module $F$ is finitely very flat if and only if
the $S$\+module $S\ot_RF$ is finitely very flat. 
\end{cor}

\begin{proof}
 Similar to the proof of Corollary~\ref{flat-descent-of-very-flatness},
using Lemma~\ref{very-flatness-extension-of-scalars}(b)
(for the ``only if'') and Lemma~\ref{flat-descent-of-flatness} with
Theorem~\ref{fvf-surj-descent-thm} (for the ``if'').
\end{proof}

 Notice that Corollary~\ref{flat-descent-of-finite-very-flatness}
is a generalization of
Lemma~\ref{fin-pres-faithfully-projective-descent}(b),
but Corollary~\ref{flat-descent-of-very-flatness} does not imply
Lemma~\ref{fin-pres-faithfully-projective-descent}(a)
(because there is a Noetherianity assumption in
Corollary~\ref{flat-descent-of-very-flatness}, which was not needed
in Lemma~\ref{fin-pres-faithfully-projective-descent}(a)).

\end{document}